\documentclass[a4paper,12pt,oneside,reqno]{amsart}
\usepackage[latin1]{inputenc}
\usepackage[english]{babel}
\usepackage[dvips]{graphicx}
\usepackage{amssymb}
\usepackage{amsmath}
\usepackage{amsthm}
\usepackage{url}
\usepackage{a4wide}

\usepackage[hmargin=1.8cm,vmargin=2.3cm]{geometry} 

\def\[#1\]{\begin{align*}#1\end{align*}}
\def\be#1\ee{\begin{align}#1\end{align}}
\def\bea#1\eea{\begin{align}#1\end{align}}
\def\ben#1\een{\begin{align*}#1\end{align*}}

\newcommand{\n}[1]{\left\Vert #1\right\Vert}
\newcommand{\la}{\left\langle}
\newcommand{\ra}{\right\rangle}
\newcommand{\R}{\mathbb{R}}

\newcommand{\N}{\mathbb{N}}

\newcommand{\mc}[1]{\mathcal{#1}}
\newcommand{\s}{\mc}
\newcommand{\p}[2]{\frac{\partial #1}{\partial #2}}

\def\ol#1{\overline{#1}}

\newcommand{\qmatrix}[1]{ \left( \begin{matrix} #1 \end{matrix} \right) }

\def\[#1\]{\begin{align*}#1\end{align*}}
\def\be#1\ee{\begin{align}#1\end{align}}
\def\bea#1\eea{\begin{align}#1\end{align}}
\def\ben#1\een{\begin{align*}#1\end{align*}}

\newtheoremstyle{theorem}{0.5cm}{0.5cm}%
   {}
   {}
   {\bfseries}
   {}
   {2ex}
   {\thmname{#1}\thmnumber{ #2}\thmnote{ #3}}
\theoremstyle{theorem}
\newtheorem{theorem}{Theorem}[section]

\newtheorem{example}[theorem]{Example}
\newtheorem{corollary}[theorem]{Corollary}
\newtheorem{remark}[theorem]{Remark}

\newtheorem{lemma}[theorem]{Lemma}

\begin{document}

\title[On  regularity of solutions for  Boltzmann transport equations]
{On regularity of solutions for certain linear Boltzmann transport equations }

\author{Jouko Tervo}
\address{University of Eastern Finland, Department of Applied Physics, P.O.Box 1627, 70211 Kuopio, Finland}
\email{040tervo@gmail.com}
 
\date{\today}

\maketitle

\begin{abstract}
The paper considers a class of linear Boltzmann transport equations 
which models a charged particle transport.
The equation is an approximation of the  original exact transport equation which
involves hyper-singular integrals in their collision terms. 
Hyper-singular integrals can be approximated by partial differential operators together with partial integral operators which leads to an approximation under consideration. 
This type of approximation is applied, for example in the dose calculation of radiation therapy. 
The related transport problem is a characteristic initial inflow boundary value problem.
Regularity results of solutions are verified utilizing the scales of relevant anisotropic  Sobolev spaces.
\end{abstract}

{\small \textbf{Keywords: } linear Boltzmann transport equation,  continuous slowing down approximation,\\
\indent regularity of solutions, characteristic boundary value problems, charged particle transport}

{\small \textbf {AMS-Classification:}  35Q20, 35R09, 35L04, 45K05

\section{Introduction}\label{intro}

We consider the existence and regularity of solutions of the {\it linear  Boltzmann transport continuous slowing down equation} 
\be\label{i1}
T\psi:=a(x,E){\p \psi{E}}
+c(x,E)\Delta_S\psi
+d(x,\omega,E)\cdot\nabla_S\psi
+\omega\cdot\nabla_x\psi+\Sigma(x,\omega,E)\psi-K_{r}\psi=f
\ee
in $G\times S\times I$. 
Here $G\subset\R^3$ is a spatial domain,  $S:=S_2\subset\R^3$ is the unit sphere (velocity direction domain) and $I=[E_0,E_{\rm m}]$ is the energy interval.
The solution satisfies the inflow  boundary condition
\be\label{i2}
{\psi}_{|\Gamma_-}=g
\ee
where $\Gamma_-$ is "the inflow boundary" $\{(y,\omega,E)\in (\partial G)\times S\times I|\ \omega\cdot\nu(y)<0\}$ (see section \ref{prem} below).
In addition, the solution obeys the initial condition
\be\label{i1-i}
\psi(.,.,E_{\rm m})=0
\ee
where $E_{\rm m}$ is the co-called {\it cut-off energy}.
The conditions (\ref{i2}), (\ref{i1-i}) guarantee that the overall initial inflow boundary value problem (\ref{i1}), (\ref{i2}), (\ref{i1-i}) is well-posed.
The \emph{restricted collision operator} $K_r$ is a { partial integral type operator} (see section \ref{res-coll} below) of the form
\bea\label{i2-k}
&
K_r\psi=
\int_{S'}\int_{I'}\sigma^1(x,\omega',\omega,E',E)\psi(x,\omega',E') dE' d\omega'
+
\int_{ S'}\sigma^2(x,\omega',\omega,E)\psi(x,\omega',E) d\omega'
\nonumber\\
&
+
\int_{I'}\int_{0}^{2\pi}
\sigma^3(x,E',E)
\psi(x,\gamma(E',E,\omega)(s),E')ds dE'.
\eea
The function $f=f(x,\omega,E)$ in (\ref{i1}) represents the { internal source}.
and $g=g(y,\omega,E)$ in \eqref{i2}  represents the { boundary source}.
The solution $\psi$ of the problem \eqref{i1}, \eqref{i1}, \eqref{i1-i} 
describes the fluence of the considered particle.

In  \cite{tervoarxiv-18} (section 6), \cite{tervo19} we have given some reasons to use the equations like (\ref{i1}) for modelling charged particle transport, for example for the propagation of electrons and positrons in radiation therapy dose calculation.  
The starting  point is that differential cross sections for charged particles may contain hyper-singularities with respect to energy variable and hence the corresponding exact (original) collision operator  is  a partial {\it hyper-singular integral operator}. This singular integral operator can be reasonably approximated which leads to an equation of the form (\ref{i1}).

In this paper 
we restrict ourselves to the case where the spatial  dimension $n=3$ ($G\subset\R^3$) and the Lebesgue index $p=2$. The assumption $n=3$ is not essential.
The regularity results are formulated utilizing certain {anisotropic Sobolev-Slobodevskij spaces} $H^{(s_1,s_2,s_3)}(G\times S\times I^\circ)$. 
In section \ref{fs} we  introduce these spaces for integer indexes $(m_1,m_2,m_3)$  and bring up  some of their properties. 
These spaces are relevant in formulations of regularity results for transport equations since they enable different degree of regularity for variables $x,\ \omega,\ E$.

Literature contains numeral contributions concerning for existence and regularity analysis of   {partial differential} initial boundary value problems for hyperbolic and formally dissipative equations
beginning from \cite{lax},  \cite{friedrich58}.  More recent results can be found e.g. in  \cite{rauch74}, \cite{rauch85}, \cite[p. 47]{rauch12}, \cite[p. 134]{pazy83}, \cite[Chapter XXIII]{hormander85}, \cite{nishitani96}, \cite{nishitani98}, \cite{morando}, \cite{takayama02}.
The well-known challenge in the initial boundary  value theory are  
\emph{characteristic} problems for which the so called boundary matrix (section \ref{rbp}) has a non zero kernel. Especially the problems where the boundary matrix has a 
\emph{variable multiplicity} (the rank of the boundary matrix is not constant) are challenging. 
We remark that characteristic boundary property depends only on the equation and the boundary but it does not depend explicitly on the boundary condition.
The transport problems considered here are of variable multiplicity. The other challenge in the present case is the inclusion of the restricted collision operator $K_r$ which is a partial integral operator.

Some specific results concerning for existence and regularity of solutions of transport problems like (\ref{i1})-(\ref{i1-i}) can  be found in literature as well.  In the case where $a=c=d=0$
existence of solutions for transport problems   like (\ref{i1}), (\ref{i2})
has been studied for  single equations e.g. in \cite{egger14}, \cite{dautraylionsv6}, \cite{agoshkov} and
for  coupled systems  in  \cite{tervo17-up}. 
In the above references it is assumed that the restricted collision operator $K_r$  satisfies a (partial) \emph{ Schur criterion } for boundedness (\cite{halmos}, p. 22).
In \cite{tervo17}, \cite{tervo18}, \cite{frank10}, \cite{frank-goudon} one has given
existence results of solutions for the case where  only $c=d=0$.
In \cite{agoshkov},
(Chapter 4)  some regularity results of solutions are exposed for the case where $a=c=d=0$. Therein a single mono-kinetic BTE is considered and one of the main results (from our perspective) is given in Theorem 4.14 where sufficient conditions under which the solution own a regularity of fractional order $\alpha$ with respect to spatial variable is obtained.  Proofs are based on the use of (fractional) differences.  
The equation (\ref{i1}) is closely related to time-dependent Fokker-Planck equations (note that  the time-dependent  equation  is of the form (\ref{i1}) when we replace time $t$ with energy $E$).
Existence and regularity  results for (deterministic and linear) Fokker-Planck type equations can be found
e.g. in \cite[Appendix A]{degond}, \cite{tian}, \cite{chupin},  \cite{le-bris}, \cite{chen}.  
In \cite{bouchut} one has shown that a time-dependent transport problem 
satisfies  
a kind of "gain in $x$-regularity with the help of $\omega$-regularity". 
In \cite{herty}, \cite{jorres}  existence results are obtained in the context of dose calculation  for optimal radiation treatment planning.  
Some $W^{p,1}$-regularity results for a stationary mono-kinetic problem is obtained in \cite[section 3]{kawagoe}.
In \cite[especially Theorem 5.3]{alonso}  (see also its references)  one has considered regularity results for the mono-kinetic transport equation as well.

Certain  collision operators obey the so called smoothing property (\cite{golse}, \cite{mokhtarkharroubi91}, \cite{zeghal})
In \cite{chen20} one has applied the  smoothing property   to retrieve regularity results in the case $a=c=d=0$. One has shown therein that under relevant assumptions the solutions of the problem
(\ref{i1}),  (\ref{i2}) obey Sobolev regularity with respect to spatial variable up to order $1-\epsilon$
where $\epsilon>0$. The spatial domain $G\subset\R^3$ is assumed to be a convex bounded set whose boundary $\partial G$ obeys certain geometric restrictions (uniform interior sphere condition). In \cite{guo17} one has obtained the first order ($W^{p,1}$ or weighted $W^{p,1}$) regularity results for a non-linear time-dependent BTE of the form (here $v$ is the velocity variable)
\be\label{non-lin}
{\p \psi{t}}+v\cdot\nabla_x\psi=Q(\psi,\psi)
\ee
under three alternative boundary conditions. The spatial domain $G$ is assumed to be a strictly convex and bounded set. In \cite{guo16} bounded mean variation regularity results are retrieved for the same kind of equation under diffuse boundary conditions.  The spatial domain therein is allowed to be non-convex. 
In \cite{alexandre} one has studied the non-linear equation like (\ref{non-lin})
in the global case $G=\R^3$. Besides of proving existence of non-negative solutions near an  equilibrium state (Maxwellian) one has obtained under certain assumptions weighted regularity of solutions with respect to all variables, more precisely one has shown that the solutions $\psi=\psi(t,x,v)$ belong to $L^\infty([0,\infty[,H_l^k(\R^6))$ where $H_l^k(\R^6):=\{f\in S'(\R^6)|\ (1+|v|^2)^{{l\over 2}}f \in H^k(\R^6)\}$. The corresponding linearized problem has been studied in section 4.2 therein. The results are valid for certain class of non-integrable (non-cutoff) cross-sections of  collision operators. Further researches in this direction have been exposed e.g. in the introduction of \cite{alexandre}.
In part of the above references (as in (\ref{non-lin})) the state space is $G\times V$ where $V$ is velocity space but they can be formulated for $G\times S\times I$ via the (local)
diffeomorphism 
$
h(\omega,E):=\sqrt{E}\omega.
$

In section \ref{ex-sol} we recall  existence and uniqueness results of solutions for the problem (\ref{i1}), (\ref{i2}), (\ref{i1-i})    
given in \cite{tervo19}. The existence results therein  founded on the {Lions-Lax-Milgram Theorem} and on relevant inflow trace theorems. In \cite{tervo18-up}, section 6.2 we applied same kind of technicalities in the case where $c=d=0$.
In \cite{tervo17}, \cite{tervo18} we considered related results based on the $m$-dissipativity of the smallest closed extension of the partial differential part of the transport operator and the methods therein offer an alternative approach for existence analysis.

Section \ref{rbp}  considers the regularity of solutions by applying the well-known explicit formulas of solutions. 
We proceed in the increasing order of complexity. For the first instance, in section \ref{reg-up-to}
we  deal with merely the convection-scattering equation 
\[
\omega\cdot\nabla_x\psi+\Sigma\psi=f.
\]
We focus on spatial regularity ($x$-regularity) but some outlines for regularity  with respect to all variables $x,\ \omega,\ E$ are exposed.
After that in section \ref{csda-ex} some regularity results  are given for the equation of the form
\[
-{\p {(a\psi)}E}+\omega\cdot\nabla_x\psi+
\Sigma\psi
= f.
\]

In section \ref{w-x-reg} we study the regularity of solutions   for an equation  involving 
the integral operator  $K_r$. We apply therein  technicalities based on \emph{Neumann series}. The treatments confine to a special case of transport equations (\ref{i1}).

In section \ref{outlines} we depict some outlines for the so called tangential (co-normal) regularity arising from  general theory of the first order PDEs. We remark that full Sobolev regularity can not be expected for characteristic   problems. However, the partial (that is, tangential) regularity can be proved for problems with constant multiplicity (\cite{rauch85}, \cite{morando09}). For problems with variable multiplicity (such as the present problems) even the tangential regularity need not be valid.   

We remark that in the global case $G=\R^3$ 
the "regularity of solution may increase according to the data".  
Actually, we conjectured in \cite{tervo19-b} (Discussion and Remark 4.9 therein) that in the case $G=\R^3$ 
under relevant assumptions the assumption
\[
f\in \cap_{k=0}^{m_2}\cap_{j=0}^{m_3}H^{(m_1+k+j,m_2-k,m_3-j)}(\R^3\times S\times I^\circ)
\]
implies that $\psi\in H^{(m_1,m_2,m_3)}(G\times S\times I^\circ)$.
For $G\not=\R^3$ 
the corresponding  problem  is characteristic which  causes limited regularity of solutions.
In fact, due to the needed inflow boundary condition (\ref{i2})  the Sobolev regularity of solutions is depending strongly on the properties of the (non-smooth) escape time mapping $t(x,\omega)$ (section \ref{prem} below) {contrary to the global case $G=\R^3$}. This can be seen transparently from the explicit solution formulas applied in section \ref{rbp}.
The regularity of solutions is not necessarily "increasing along the data, at least with respect to $(x,\omega)$-variable".  
It is known that the regularity of the solution of the transport equations  
with respect to $x$-variable, for example is limited up to space $H^{(s_1,0,0)}(G\times S\times I^\circ),\ s_1<3/2$   regardless of the smoothness of the data 
(see counterexample given in \cite{tervo17-up}, section 7.1). We  remark, however that the  regularity with respect to $E$-variable up to any order may be obtained (see section \ref{rev} below).

\section{Preliminaries}\label{prem}

We assume that $G$ is an open bounded set in $\R^3$ (equipped with the Lebesgue measure) whose boundary $\partial G$  belongs to  $C^{1}$ 
(see e.g. \cite{grisvard}, section 1.2). Moreover, we assume that $G$ is \emph{convex} which simplifies some formulations. 
The unit outward pointing normal on $\partial G$  is denoted by $\nu$. The surface measure on $\partial G$ is $d\sigma$.
Let $S=S_2$ be the unit sphere in $\R^3$ equipped with the usual surface measure $d\omega$.
Furthermore, let $I$ be an interval $[E_0,E_{\rm m}]$  of $\R$  where
$E_0\geq 0, \ E_0<E_{\rm m}<\infty$.  The interior $]E_0,E_{\rm m}[$ is denoted by $I^\circ$.
We use abbreviations
\[
\Gamma:=\partial G\times S\times I,
 \
\Gamma':=\partial G\times S,
\]
\[
\Gamma_{-}:=\{(y,\omega,E)\in \partial G\times S\times I|\ \omega\cdot\nu(y)<0\},
\
\Gamma'_{-}:=\{(y,\omega)\in \partial G\times S|\ \omega\cdot\nu(y)<0\}
\]
\[
\Gamma_{+}:=\{(y,\omega,E)\in \partial G\times S\times I|\ \omega\cdot\nu(y)>0\},
\
\Gamma'_{+}:=\{(y,\omega)\in \partial G\times S|\ \omega\cdot\nu(y)>0\}.
\]

For  $(x,\omega)\in G\times  S$ the {\it escape time mapping $t(x,\omega)$ } is defined by
\be
t(x,\omega)&=\inf\{s>0|\ x-s\omega\not\in G\}\\ \nonumber
&=\sup\{t>0|\ x-s\omega\in G\ {\rm for\ all}\ 0< s <t\}.
\ee
Furthermore, for $(y,\omega)\in \Gamma'_-\cup\Gamma'_+$ the {\it escape-time mapping $\tau(y,\omega)$ from boundary to boundary } is defined by
\bea
&
\tau_-(y,\omega):=
\inf\{s>0|\ y+s\omega\not\in G\},\ (y,\omega)\in \Gamma'_-\\
&
\tau_+(y,\omega):=
\inf\{s>0|\ y-s\omega\not\in G\},\ (y,\omega)\in \Gamma'_+ .
\eea
We find that
\[
0\leq t(x,\omega)\leq d:={\rm diam}(G),\ 0\leq\tau_{\pm}(y,\omega)\leq d
\]
where ${\rm diam}(G)$ is the diameter of $G$.

\subsection{Basic function  spaces}\label{fs}

All function spaces below are real valued.
Define the  space $W^2(G\times S\times I)$  by
\bea\label{fseq1}
W^2(G\times S\times I)
=\{\psi\in L^2(G\times S\times I)\ |\  \omega\cdot\nabla_x \psi\in L^2(G\times S\times I) \}.
\eea
The space $W^2(G\times S\times I)$ is equipped  with the inner product
\be\label{fs4}
\la {\psi},v\ra_{W^2(G\times S\times I)}=\la {\psi},v\ra_{L^2(G\times S\times I)}+
\la\omega\cdot\nabla_x\psi,\omega\cdot\nabla_x v\ra_{L^2(G\times S\times I)}.
\ee
Then $W^2(G\times S\times I)$  is a Hilbert space.

Let $X$ be a Hilbert space.
The space $H^k(I,X),\ k\in \N_0$ denotes the standard $X$-valued Sobolev space that is,
\[
H^k(I,X)=\{\psi|\ {{\partial^l\psi}\over{\partial E^l}}\in X,\  l\leq k\}
\]
and $H^k(I,X)$ is equipped with the standard inner product.
Let for $k\in\N$
\[
C^k(\ol G\times S\times I):=\{\psi\in C^k(G\times S\times I^\circ)|\ \psi=f_{|G\times S\times I^\circ},\ f\in C_0^k(\R^3\times S\times\R)\}.
\] 
The space $C^1(\ol G\times S\times I)$
is a dense subspace of $W^2(G\times S\times I)$ and $H^k(I,L^2(G\times S))$.

The space of $L^2$-functions on $\Gamma_-$ with respect to the measure
$|\omega\cdot\nu|\ d\sigma d\omega dE$  is denoted by $T^2(\Gamma_-)$
that is, $T^2(\Gamma_-)=L^2(\Gamma_-,|\omega\cdot\nu|\ d\sigma d\omega dE)$.
$T^2(\Gamma_-)$   is a Hilbert space and its natural inner product is
\be\label{fs8}
\la g_1,g_2\ra_{T^2(\Gamma_-)}=\int_{\Gamma_-}g_1(y,\omega,E)g_2(y,\omega,E)|\omega\cdot\nu|\ d\sigma(y) d\omega dE.
\ee
In patches $\Gamma_-^i=\{(y,\omega,E)\in (\partial G)_i\times S\times I|\ |\omega\cdot\nu(y)<0\}$ the inner product (\ref{fs8}) is computed by
\[
&
\la g_1,g_2\ra_{T^2(\Gamma_-^i)}
\\
&
=\int_{S\times I}\int_{U_{i,-,\omega}}g_1(h_i(z),\omega,E)g_2(h_i(z),\omega,E)|\omega\cdot\nu(h_i(z)) \n{(\partial_{z_1}h_i\times\partial_{z_2}h_i)(z)} dz d\omega dE
\]
where $U_{i,-,\omega}:=\{z\in U_i|\ \omega\cdot\nu(h_i(z))<0\}$ and $h_i:= \varphi_i^{-1}$ 
when $(U_i,\varphi_i)$ is the corresponding chart of $(\partial G)_i$.

The spaces $T^2(\Gamma_+)$ and $T^2(\Gamma)$ and their inner products are similarly defined.
In addition, we define the (Hilbert) space
\[
T_{ \tau_{\pm}}^2(\Gamma_{\pm})=L^2(\Gamma_{\pm},\tau_{\pm}(.,.)|\omega\cdot\nu|\ d\sigma d\omega dE)
\]
where the canonical inner product
\[
\la g_1,g_2\ra_{T_{\tau_{\pm}}^2(\Gamma_\pm)}=\int_{\Gamma_\pm}g_1(y,\omega,E)g_2(y,\omega,E) \tau_{\pm}(y,\omega)|\omega\cdot\nu|\ d\sigma(y) d\omega dE
\]
is used.

Since $\tau_{\pm}(y,\omega)\leq {\rm diam}(G)=d<\infty$ we obtain
from \cite{dautraylionsv6}, p. 252, \cite{cessenat84} or \cite{tervo18-up}, Theorem 2.16 

\begin{theorem}\label{tth}
The inflow trace operators 
\[
\gamma_{\pm}:{ W}^2(G\times S\times I)\to T_{\tau_{\pm}}^2(\Gamma_{\pm}) \ {\rm defined\ by}\ \gamma_{\pm}(\psi):=\psi_{|\Gamma_{\pm}}
\] 
are (well-defined) bounded surjective operators and they have bounded right inverses $L_\pm:T^2_{\tau_\pm}(\Gamma_\pm)\to { W}^2(G\times S\times I)$ that is, $\gamma_\pm\circ L_\pm=I$. The operators $L_\pm$  are called {\it lifts}.
\end{theorem}

\begin{remark}\label{lift}
The lift $L_-$ can be given explicitly, for instance by
\[
(L_-g)(x,\omega,E)=e^{-\lambda t(x,\omega)}g(x-t(x,\omega)\omega,\omega,E)
\]
where $\lambda\geq 0$. For $\lambda=0$ the  lift $L_-$ is an isometrics. Note that the lift $L_-$ is not unique. Analogous observations are valid for $L_+$.
\end{remark}

Because for any compact $K\subset\Gamma_\pm$ we have $\tau_\pm(y,\omega)\geq c_K>0$ for all $(y,\omega,E)\in K$,  it follows from Theorem \ref{tth}  that
the trace operators $\gamma_\pm:{ W}^2(G\times S\times I)\to \ L^2_{\rm loc}(\Gamma_\pm,
|\omega\cdot\nu|\ d\sigma d\omega dE)$ are continuous.

The trace $\gamma(\psi),\ \psi\in { W}^2(G\times S\times I)$ is not necessarily in the space $T^2(\Gamma)$.  Hence we define the spaces
\be\label{fs12}
\widetilde{ W}^2(G\times S\times I)=\{\psi\in { W}^2(G\times S\times I)|\  \gamma(\psi)\in T^2(\Gamma) \}.
\ee
The space $\widetilde { W}^2(G\times S\times I)$ is equipped with the inner product
\be\label{fs14}
\la \psi,v\ra_{\widetilde { W}^2(G\times S\times I)}=\la \psi,v\ra_{{ W}^2(G\times S\times I)}+ \la \gamma(\psi),\gamma(v)\ra_{T^2(\Gamma)}.
\ee
Then $\widetilde { W}^2(G\times S\times I)$ is a Hilbert space (\cite{tervo17-up}, Proposition 2.5).
For convex domains $G$ 
the space $\widetilde { W}^2(G\times S\times I)$ is the completion of $C^1(\ol G\times S\times I)$ with respect to the norm (\ref{fs14}) that is, $C^1(\ol G\times S\times I)$ is dense in
$\widetilde { W}^2(G\times S\times I)$ (\cite{tervo18-up}, Corollary 2.22).

For $v\in \widetilde { W}^2( G\times S\times I)$ and $\psi\in\widetilde { W}^2(G\times S\times I)$ it holds the Green's formula 
\begin{align}\label{green}
\int_{G\times S\times I}(\omega\cdot \nabla_x \psi)v\ dxd\omega dE
+\int_{G\times S\times I}(\omega\cdot \nabla_x v)\psi\ dxd\omega dE=
\int_{\partial G\times S\times I}(\omega\cdot \nu) v\ \psi\ d\sigma d\omega dE
\end{align}
which is obtained by Stokes Theorem  for $v,\ \psi\in C^1(\ol G\times S\times  I)$ and then by the density arguments for general $v\in \widetilde { W}^2(G\times S\times I)$ and $\psi\in\widetilde { W}^2(G\times S\times I)$.

\subsection{Anisotropic Sobolev spaces}\label{an-spaces}

Let $m=(m_1,m_2,m_3)\in \N^3$ be a multi-index (for simplicity we restrict ourselves here to integer multi-indexes but the definitions below can be generalized for $s=(s_1,s_2,s_3)\in [0,\infty[^3$ in the standard way).
Define an {\it  anisotropic ( or mixed-norm)  Sobolev spaces} $H^m(G\times S\times I^\circ):=H^{m,2}(G\times S\times I^\circ)$ by 
\bea
&
H^{m}(G\times S\times I^\circ)
:=\Big\{\psi\in L^2(G\times S\times I)\ |\ \partial_x^\alpha\partial_{{\omega}}^\beta\partial_E^l\psi \in L^2(G\times S\times I),\nonumber\\
&
 {\rm for\ all}\ |\alpha|\leq m_1,\ |\beta|\leq m_2,\ l\leq m_3\Big\}.
\eea
Here we denote more shortly ${{\partial^\alpha}\over{\partial x^\alpha}}=\partial_x^\alpha$
and similarly $\partial_\omega^\beta=\partial_{\omega_1}^{\beta_1}\partial_{\omega_2}^{\beta_2}$
where $\{\partial_{\omega_j},\ j=1,2\}$ is a local  basis of the tangent space $T(S)$. 
We recall that for sufficiently smooth functions $f:S\to\R$ 
\[
{\p f{\omega_j}}_{\Big|\omega}={\partial\over{\partial w_j}}(f\circ h)_{\Big|u=h^{-1}(\omega)},\ j=1,2
\]
where  $h:W\to S\setminus S_0$ is a parametrization of  $S\setminus S_0$ where $S_0$ has the surface measure zero.
Moreover, recall that $f\in L^2(S)$  if and only if $f\circ h\in L^2(W,\n{\partial_1h\times\partial_2h}du)$. 
For $f\in L^2(S)$ we define the inner product
\be\label{sn}
\la f_1,f_2\ra_{L^2(S)}=\int_Sf_1 f_2 d\omega=\int_W(f_1\circ h)(f_2\circ h)\n{\partial_1h\times\partial_2h}du.
\ee
The space $H^{m}(G\times S\times I^\circ)$  is a Hilbert space when equipped with the  inner product
\be\label{hminner}
\la\psi,v\ra_{H^{m}(G\times S\times I^\circ)}
:=\sum_{|\alpha|\leq m_1,|\beta|\leq m_2,
l\leq m_3}
\la\partial_x^\alpha\partial_{\omega}^\beta\partial_E^l\psi,
\partial_x^\alpha\partial_{\omega}^\beta\partial_E^l v\ra_{L^2(G\times S\times I)}.
\ee
The corresponding norm is
\[
\n{\psi}_{H^{m}(G\times S\times I^\circ)}=\Big(\sum_{|\alpha|\leq m_1}\sum_{|\beta|\leq m_2}\sum_{l\leq m_3}
\n{\partial_x^\alpha\partial_{\omega}^\beta\partial_E^l\psi }_{L^2(G\times S\times I)}^2\Big)^{{1\over 2}}.
\]
In $L^2(\partial G)$ we  define an inner product analogously to (\ref{sn})
that is, 
\[
\la f_1,f_2\ra_{L^2(\partial G)}=\int_{\partial G}f_1 f_2 d\sigma.
\]
The definition of spaces $H^{m}((\partial G)\times S\times I^\circ)$ and their
inner products are analogous to that of $H^{m}(G\times S\times I^\circ)$.

Note that for $m'\geq m$ (this means that $m_j'\geq m_j$ for $j=1,2,3$)
\[
H^{m'}(G\times S\times I^\circ)\subset H^{m}(G\times S\times I^\circ)
\]
and that (since $G\times S\times I^\circ$ is bounded)
\[
W^{\infty,m}(G\times S\times I^\circ)\subset
H^{m}(G\times S\times I^\circ).
\]
where for $m\in\N_0^3$
\bea
&
W^{\infty,m}(G\times S\times I^\circ)
:=\Big\{\psi\in L^\infty(G\times S\times I)\big|\ 
\n{\partial_x^\alpha\partial_{\omega}^\beta\partial_E^l\psi}_{L^\infty(G\times S\times I)}<\infty \\
&
{\rm for }\ 
|\alpha|\leq m_1,\ |\beta|\leq m_2,\  l\leq m_3\Big\}
\eea
equipped with the natural norm.

The tensor product $H^{m_1}(G)\otimes H^{m_2}(S)\otimes  H^{m_3}(I^\circ)$ 
is a dense subspace of $H^{m}(G\times S\times I^\circ)$ but generally the spaces $H^{m_1}(G)\otimes H^{m_2}(S)\otimes  H^{m_3}(I^\circ)$ 
and $H^{m}(G\times S\times I^\circ)$ are not equal
(principles of these kind of results can be found e.g. in
\cite{aubin}, Chapter 12). 
In the case where the boundary $\partial G\in C^{m_1}$
the space $C^{m_1}(\ol G\times S\times I)$
is dense in $H^{m}(G\times S\times I^\circ)$.

\begin{remark}\label{re-spaces}
For multi-indexes of the form $m=(m_1,0,m_3)$ the spaces $H^{m}(G\times S\times I^\circ)$ can be characterized by using 
the \emph{partial Fourier transforms} with respect to $(x,E)$-variable. 
The partial Fourier transform of $\psi\in L^1(\R^3\times S\times \R)$ is given by
\[ 
(\mc F_{(x,E)}\psi)(\xi,\omega,\eta):=\int_{\R^3}\int_{\R}\psi(x,\omega,E)
e^{-{\rm i}\la (\xi,\eta),(x,E)\ra} dE  dx,\ 
(\xi,\eta)\in \R^3\times\R,\ \omega\in S.
\]
In the case where $m_2\not=0$ the difficulties of  characterizations by Fourier transforms arise from the definition of Fourier transform on $S$. We omit all formulations for these approaches.

\end{remark}

\section{Existence of solutions}\label{ex-sol}

\subsection{Assumptions for the restricted collision operator}\label{res-coll}

We assume that the restricted collision operator is  the sum (for  some details see \cite{tervo18-up}, section 5)
\be\label{esols1} 
K_r=K_r^1+K_r^2+K_r^3.
\ee
Here $K_r^1$ is of the form 
\[
(K_r^1\psi)(x,\omega,E)=\int_{S'\times I'}\sigma^1(x,\omega',\omega,E',E)\psi(x,\omega',E')d\omega' dE',  
\]
where $\sigma^1:G\times S^2\times I^2\to\R$ is a non-negative measurable function such that 
\bea\label{ass5-a}
&\int_{S'\times I'}\sigma^1(x,\omega',\omega,E',E)d\omega' dE'\leq M_1,\nonumber\\
&\int_{S'\times I'}\sigma^1(x,\omega,\omega',E,E')d\omega' dE'\leq M_2,
\eea
for a.e. $(x,\omega,E)\in G\times S\times I$.

The operator
$K_r^2$ is of the form 
\[
(K_r^2\psi)(x,\omega,E)=\int_{ S'}\sigma^2(x,\omega',\omega,E)\psi(x,\omega',E) d\omega',  
\]
where $\sigma^2:G\times S^2\times I\to\R$ is a non-negative  measurable function  such that
\bea\label{ass7}
&\int_{S'}\sigma^2(x,\omega',\omega,E)d\omega'\leq M_1,\nonumber\\
&\int_{S'}\sigma^2(x,\omega,\omega',E) d\omega'\leq M_2,
\eea
for a.e. $(x,\omega,E)\in G\times S\times I$.

Finally, $K_r^3$ is of the form 
\[
(K_r^3\psi)(x,\omega,E)
=
\int_{I'}\int_{0}^{2\pi}
\sigma^3(x,E',E)
\psi(x,\gamma(E',E,\omega)(s),E')ds dE'
\]
where
$\gamma=\gamma(E',E,\omega):[0,2\pi]\to S$
is a parametrization of the curve (an example for the choice of $\gamma=\gamma(E',E,\omega)$ is given in \cite{tervo18-up}, section 3.2)
\[
\Gamma(E',E,\omega)=\{\omega'\in S\ |\ \omega'\cdot\omega-\mu(E',E)=0\}
\]
where $\mu(E',E):=\sqrt{{{E(E'+2)}\over{E'(E+2)}}}$.
Moreover,
${\sigma}^3:G\times I^2\to\R$ is a non-negative measurable function such that
\bea\label{ass-8}
&\int_{I'}{\sigma}^3(x,E',E)dE'\leq M_1, \nonumber\\
&\int_{I'}{\sigma}^3(x,E,E')dE'\leq M_2,
\eea
for a.e. $(x,E)\in G\times I$. 

The following result is shown  in \cite{tervo18-up}, 

\begin{theorem}\label{esol-th1}
The operators  $K_r^j,\ j=1,2,3$  are bounded operators
$L^2(G\times S\times I)\to L^2(G\times S\times I)$ and
\be\label{k-norm-a}
\n{K_r^1}\leq \sqrt{M_1M_2},
\ee
\be\label{k-norm-b}
\n{K_r^2}\leq \sqrt{M_1M_2},
\ee
\be\label{k-norm-3}
\n{K_r^3}\leq 2\pi\ \sqrt{M_1M_2}.
\ee
Hence the operator $K_r:L^2(G\times S\times I)\to L^2(G\times S\times I)$ is bounded.
\end{theorem} 

\begin{proof}
\cite{tervo18-up}, Theorem 5.13.
\end{proof}

In order to render the operator $\Sigma-K_r$ coercive (accretive),
we shall assume that 
\be\label{ass-for-s}
\Sigma\in L^\infty(G\times S\times I),\ \Sigma\geq 0\ {\rm a.e.}
\ee
and
\begin{multline}
\Sigma(x,\omega,E)-\int_{S'\times I'}\sigma^1(x,\omega,\omega',E,E') d\omega' dE'
\\ \label{ass8-aa}
-\int_{S'}\sigma^2(x,\omega,\omega',E)  d\omega'
-2\pi\int_{I'}{\sigma}^3(x,E,E')dE'
\geq c',
\end{multline}
and
\begin{multline}
\Sigma(x,\omega,E)-\int_{S'\times I'}\sigma^1(x,\omega',\omega,E',E) d\omega' dE'
\\ \label{ass9-a}
-\int_{S'}\sigma^2(x,\omega',\omega,E)  d\omega'
-2\pi\int_{I'}{\sigma}^3(x,E',E)dE'
\geq c',
\end{multline}
for a.e. $(x,\omega,E)\in G\times S\times I$.
In the sequel we assume that $c'>0$.

Next result addresses coercivity (accretivity) of $K_r-\Sigma$ under the above assumptions. The result  is proven  in \cite{tervo18-up}, 

\begin{theorem}\label{SK-dissip}
Suppose that the assumptions (\ref{ass-for-s}),  (\ref{ass5-a}), (\ref{ass7}),  (\ref{ass-8}), (\ref{ass8-aa}) and (\ref{ass9-a}) are valid. 
Then 
\be\label{K-coer}
\la (\Sigma-K_r)\psi,\psi\ra_{L^2(G\times S\times I)}\geq c'\n{\psi}^2_{L^2(G\times S\times I)}\quad \forall \psi\in L^2(G\times S\times I).
\ee
\end{theorem}

\begin{proof}
\cite{tervo18-up}, Theorem 5.14.
\end{proof}

\subsection{Existence of solutions for the initial inflow boundary value problem}\label{ex-i-i-b-p}

Consider the  problem
\be\label{trunc-ex-1}
T\psi=f,\ \psi_{|\Gamma_-}=g,\ \psi(.,.,E_m)=0
\ee
where $f\in L^2(G\times S\times I),\ g\in T^2(\Gamma_-)$ and where
$T$ is of the form 
\be\label{trunc-ex-2}
T\psi=
a(x,E){\p \psi{E}}
+ c(x,E)\Delta_S\psi
+d(x,\omega,E)\cdot\nabla_S\psi
+\omega\cdot\nabla_x\psi+\Sigma\psi-K_{r}\psi.
\ee
Denote
\[
b(x,\omega,E,\partial_\omega)\psi
:= c(x,E)\Delta_S\psi
+d(x,\omega,E)\cdot\nabla_S\psi.
\]
Above $d=(d_1,d_2)\sim d_1{\partial\over{\partial\omega_1}}+ d_2{\partial\over{\partial\omega_2}}$ and $d\cdot\nabla_S$ is a Riemannian inner product on $T(S)$.

Define in $C^1(\ol G\times I,C^2(S))$  inner products
\bea \label{inH0}
&
\la \psi,v\ra_{\mc H}:=\la \psi,v\ra_{L^2(G\times S\times I)}+
\la\gamma(\psi),\gamma(v)\ra_{T^2(\Gamma)}\nonumber\\
&
+\la\psi(.,.,E_0),v(.,.,E_0)\ra_{L^2(G\times S)}+\la\psi(.,.,E_m),v(.,.,E_m)\ra_{L^2(G\times S)}\nonumber\\
&
+
\la \psi,v\ra_{L^2(G\times I,H^1(S))}
\eea
and 
\be\label{inhatH0}
\la \psi,v\ra_{\widehat {\mc H}}
=\la \psi,v\ra_{\mc H}
+
\la\omega\cdot\nabla_x \psi,\omega\cdot\nabla_x v\ra_{L^2(G\times S\times I)}
+
\la {\p \psi{E}},{\p v{E}}\ra_{L^2(G\times S\times I)}.
\ee
Let $\mc H$ and $\widehat{\mc H} $ be the completions of
$C^1(\ol G\times I, C^2(S))$ with respect to inner products $\la .,.\ra_{\mc H}$ and $\la .,.\ra_{\mc {\widehat H}}$, respectively.

We assume for the coefficients:
\be\label{trunc-18}
a\in W^{\infty,1}(G\times I^\circ), c\in L^\infty(G\times I),
\ \ d_j\in L^\infty(G\times I,W^{\infty,1}(S))
\ee
\be\label{trunc-14}
-\Big({\p a{E}}(x,E)+({\rm div}_S\ d)(x,\omega,E)\Big)\geq q_1>0,\ {\rm a.e.},
\ee
\be\label{trunc-19}
-c(x,E)\geq q_2>0,  \ {\rm a.e.},
\ee
\be\label{trunc-20}
-a(x,E_0)\geq q_3>0,\ -a(x,E_m)\geq q_3>0,  \ {\rm a.e.}
\ee
\be\label{assfor-a}
|a(x,E)|\geq q_4>0\ {\rm a.e.}.
\ee

Let $P(x,\omega,E,D)$ be the differential part of $T$ that is,
\[
P(x,\omega,E,D)\psi:= 
a(x,E){\p \psi{E}}
+b(x,\omega,E,\partial_\omega)\psi
+ \omega\cdot\nabla_x\psi.
\]
With this notation, the equation \eqref{trunc-ex-1}  can be written as
\begin{align*}
P(x,\omega,E,D)\psi+\Sigma\psi - K_r\psi = f.
\end{align*}
The space
\begin{multline}
{\s H}_P(G\times S\times I^\circ):=\{\psi\in L^2(G\times S\times I)\ | \\
P(x,\omega,E,D)\psi\in L^2(G\times S\times I)\ {\rm in\ the \ weak\ sense}\}
\end{multline}
is a Hilbert space when equipped with the inner product 
\[
\la \psi,v\ra_{{\s H}_P(G\times S\times I^\circ)}=\la \psi,v\ra_{L^2(G\times S\times I)}+\la P(x,\omega,E,D)\psi,P(x,\omega,E,D)v\ra_{L^2(G\times S\times I)}.
\]

In \cite{tervo19} we showed the following existence result for the initial and inflow boundary value problem (\ref{i1}), (\ref{i2}), (\ref{i1-i}).

\begin{theorem}\label{cor-csdath3}
Suppose that the assumptions (\ref{ass-for-s}), (\ref{ass5-a}), (\ref{ass7}), (\ref{ass-8}), (\ref{ass8-aa}), (\ref{ass9-a}),
(\ref{trunc-18}), (\ref{trunc-14}), (\ref{trunc-19}), (\ref{trunc-20}), (\ref{assfor-a}) are valid.
Let ${ f}\in L^2(G\times S\times I)$ and $g\in T^2(\Gamma_-)$.
Then the  transport problem 
\bea\label{trunc-ex-3-a}
&
a(x,E){\p \psi{E}}
+c(x,E)\Delta_S\psi
+d(x,\omega,E)\cdot\nabla_S\psi
+\omega\cdot\nabla_x\psi+\Sigma\psi-K_{r}\psi=f\nonumber\\
&
\psi_{|\Gamma_-}=g,\ \psi(.,.,E_m)=0
\eea
has a unique solution $\psi\in \mc H\cap {\s H}_P(G\times S\times I^\circ)$.
In addition,   the solution $\psi$
obeys the a priori  estimate
\be\label{csda40aa}
\n{\psi}_{{\mc H}}\leq C
\big(\n{{ f}}_{L^2(G\times S\times I)}+\n{{g}}_{T^2(\Gamma_-)}\big).
\ee
\end{theorem}

\begin{proof}
See \cite[Theorem 6.13]{tervo19}.
\end{proof}
 
The a priori estimate (\ref{csda40aa}) means that the solution $\psi$
obeys 
\be\label{ex-g9}
\n{\psi}_{\mc H}\leq C(\n{{ T\psi}}_{L^2(G\times S\times I)}+\n{\psi_{|\Gamma_-}}_{T^2(\Gamma_-)}).
\ee

\section{On regularity of solutions of initial inflow boundary value problem by using explicit formulas}\label{rbp}

For $G\not=\R^3$ the relevant problem is  an initial inflow boundary value problem of the form
\be\label{i-i-pr}
T\psi=f, \ \psi_{|\Gamma_-}=g,\ \psi(.,.,E_m)=0.
\ee
The regularity of the solution $\psi$  is limited 
(\cite{tervo17-up}, Example 7.4).  
To explain the nature of the initial inflow boundary value problem under consideration we assume (for simplicity) that $c=0$. 
Then the equation  is the first order partial differential equation
\be\label{ds0-1}
a(x,E){\p \psi{E}}+\omega\cdot\nabla_x\psi
+d(x,\omega,E)\cdot\nabla_S\psi
+\Sigma\psi-K_{r}\psi
 ={ f}
\ee
or equivalently
\be\label{syst-2}
a(x,E)\partial_E\psi
+\sum_{j=1}^3\omega_j\partial_{x_j}\psi
+\sum_{j=1}^2d_j(x,\omega,E)\partial_{\omega_j}\psi
+(\Sigma-K_{r})\psi=f.
\ee
Using the terminology of \cite{nishitani96} and noting that $S$ is a manifold without boundary the
{\it boundary matrix} is                          
\[
A_{\nu}(y,\omega)=
\sum_{j=1}^3\nu_j(y)\omega_j
=\omega\cdot\nu(y), \ (y,\omega)\in\Gamma'.
\] 
We see that the present transport problem  is of variable multiplicity
that is, ${\rm rank}( A_\nu)$ is not constant on $\Gamma'$.

Suppose that $a>0$.
Since the problem  is scaler-valued the operator
(\ref{syst-2}) { is  symmetric hyperbolic}  in its wide sense.
The general theory for hyperbolic problems is largely  valid only for problems where the kernel ${\rm Ker}(A_\nu)$ is zero on $\Gamma'$ or in some extent to the problems where ${\rm rank}(A_\nu)$ is constant on $\Gamma'$ (that is, for problems with constant multiplicity). 
Some results  {for problems with variable multiplicity} can also be found  
in literature but they require 
an additional assumption which is concerning for  the smooth $(n-2)$-dimensional manifold $\Gamma_0'$ (see section \ref{outlines} below).

\subsection{On regularity of solutions for convection-scattering equation  using explicit formulas}\label{reg-up-to}

In the sequel we  calculate some regularity results emerging from explicit formulas of solutions. We assume everywhere in this section that the restricted collision operator $K_r=0$.

Consider 
the  transport problem 
\bea\label{stif1}
&
\omega\cdot\nabla_x\psi+\Sigma\psi=f\ {\rm in}\ G\times S\times I\nonumber\\
&
\psi_{|\Gamma_-}=g
\eea
where $f\in L^2(G\times S\times I)$ and $g\in T^2(\Gamma_-)$. The solution $\psi$ can be given explicitly (see e.g. \cite{tervo17-up}, section 4 and \cite{tervo18-up}, Lemmas 2.12 and 2.14) 
\bea\label{stif2}
&
\psi=\int_0^{t(x,\omega)} e^{-\int_0^t\Sigma(x-s\omega,\omega,E)ds}f(x-t\omega,\omega,E)dt
\nonumber\\
&
+
e^{-\int_0^{t(x,\omega)}\Sigma(x-s\omega,\omega,E)ds}g(x-t(x,\omega)\omega,\omega,E).
\eea
When $G\not=\R^3$ we are not able to deduce regularity results similar to the global case \cite{tervo19-b}. The principle that the solution is more and more regular { in the Sobolev sense} when the data is more and more regular is not valid. This is due to the presence of inflow boundary conditions and one must impose certain regularity properties also for the escape time mapping $t(x,\omega)$. 

We will need the
following form of Fubini's
theorem  (\cite[Remark 2.18]{tervo18-up}, \cite[Lemma 2.1]{choulli}).

\begin{lemma}[(Fubini)]\label{pr:fubini}
Let $\psi\in L^1(G\times S\times I)$. Then
\[
\int_{G\times S\times I} \psi(x,\omega,E) dx d\omega dE
=
\int_{\Gamma_-}\int_0^{\tau_-(y,\omega)} \psi(y+t\omega,\omega,E) |\omega\cdot\nu(y)| dt d\sigma(y) d\omega dE.
\]
\end{lemma}

\subsubsection{First order regularity with respect to spatial variable}\label{foreg}

Let $(y,\omega)\in \Gamma_-'$. Since we assume that $G$ is convex
the escape time mapping is a $C^1$-function with respect to $x$-variable
in a neighbourhood of $(x,\omega)\in G\times S$ (\cite{tervo17-up}, Proposition 4.7.
Define  \emph{weight functions}
\[
m_{1,j}(y,\omega):=
\int_0^{\tau_-(y,\omega)}\big| {\p t{x_j}}(y+t\omega,\omega)\big|^2dt,\ j=1,2,3
\]
and  weighted norms
\[
\n{g}_{T_{m_{1,j}}^2(\Gamma_-)}^2:=\int_{\Gamma_-}|g(y,\omega,E)|^2m_{1,j}(y,\omega)|\omega\cdot\nu(y)| d\sigma(y) d\omega dE.
\]
Let
\[
T_{m_{1,j}}^2(\Gamma_-):=\{g\in L_{\rm loc}^2(\Gamma_-)|\ \n{g}_{T_{m_{1,j}}^2(\Gamma_-)}<\infty\}.
\]

At first we deal with the first order  regularity with respect to $x$-variable.  

\begin{theorem}\label{foregth1}
Assume that
\be\label{ashr4}
\Sigma\in W^{\infty,(1,0,0)}(G\times S\times I^\circ)\cap L^\infty(\Gamma),\ \Sigma\geq c>0,
\ee
\be\label{ashr1-a}
g\in T^2(\Gamma_-)\cap T^2_{m_{1,j}}(\Gamma_-), \ j=1,2,3
\ee
\be\label{ashr1-b} 
\nabla_{(\partial G)}g\in T^2(\Gamma_-)^2\cap T^2_{m_{1,j}}(\Gamma_-)^2,\ j=1,2,3
\ee
where $\nabla_{(\partial G)}$ is the gradient on $\partial G$,
\be\label{ashr2-a}
f\in H^{(1,0,0)}(G\times S\times I^\circ),
\ee
\be\label{ashr2-b}
f_{|\Gamma_-}\in T^2_{m_{1,j}}(\Gamma_-),\ j=1,2,3.
\ee
Then the solution $\psi$ of the problem (\ref{stif1}) belongs to
$H^{(1,0,0)}(G\times S\times I^\circ)$.
\end{theorem}

\begin{proof}
Denote
\[
\psi_1:=\int_0^{t(x,\omega)} e^{-\int_0^t\Sigma(x-s\omega,\omega,E)ds}f(x-t\omega,\omega,E)dt,
\]
\[
\psi_2:=
e^{-\int_0^{t(x,\omega)}\Sigma(x-s\omega,\omega,E)ds}g(x-t(x,\omega)\omega,\omega,E).
\]
Then by (\ref{stif2}) $\psi=\psi_1+\psi_2$.

A.
Consider at first $\psi_1$. We have
\bea\label{hreg2}
{\p {\psi_1}{x_j}}
&=\int_0^{t(x,\omega)}\big(-\int_0^t{\p \Sigma{x_j}}(x-s\omega,\omega,E)ds\big)
\cdot e^{-\int_0^t\Sigma(x-s\omega,\omega,E)ds}\cdot
f(x-t\omega,\omega,E)dt\nonumber\\
&
+\int_0^{t(x,\omega)} e^{-\int_0^t\Sigma(x-s\omega,\omega,E)ds} {\p {f}{x_j}}(x-t\omega,\omega,E)dt
\nonumber\\
&
+
e^{-\int_0^{t(x,\omega)}\Sigma(x-s\omega,\omega,E)ds}\ f(x-t(x,\omega)\omega,\omega,E)\ {\p t{x_j}}((x,\omega)\nonumber\\
&
=:h_1+h_2+h_3.
\eea

A.1.
Furthermore, since $\Sigma\geq 0$ and since $0\leq t\leq t(x,\omega)\leq d={\rm diam}(G)$ we have
\bea\label{hreg2-a}
&
|h_1(x,\omega,E)|\leq
\int_0^{t(x,\omega)}\big|\int_0^t{\p \Sigma{x_j}}(x-s\omega,\omega,E)ds\big|
\cdot e^{-\int_0^t\Sigma(x-s\omega,\omega,E)ds}\cdot
|f(x-t\omega,\omega,E)|dt\nonumber
\\
&
\leq 
\int_0^{t(x,\omega)}\int_0^t\big|{\p \Sigma{x_j}}(x-s\omega,\omega,E)\big|ds
\cdot
|f(x-t\omega,\omega,E)|dt\nonumber\\
&
\leq
d\n{{\p \Sigma{x_j}}}_{L^\infty(G\times S\times I)}
\int_0^{t(x,\omega)}|f(x-t\omega,\omega,E)|dt\nonumber\\
&
\leq
d\n{{\p \Sigma{x_j}}}_{L^\infty(G\times S\times I)}\sqrt{d}
\Big(\int_0^{t(x,\omega)}|f(x-t\omega,\omega,E)|^2dt\Big)^{1/2}
\eea
where we in the last step applied Cauchy-Schwartz's inequality.
Let $\ol f$ be the extension by zero of $f$  on 
$\R^3\times S\times I$.
Then  we get by (\ref{hreg2-a}) an estimate (note that $x-t\omega\in G$ for $0\leq t\leq t(x,\omega)$)
\bea\label{e-1}
&
\n{h_1}_{L^2(G\times S\times I)}^2
\leq
d^3\n{{\p \Sigma{x_j}}}_{L^\infty(G\times S\times I)}^2
\int_G\int_S\int_I\int_0^{t(x,\omega)}|f(x-t\omega,\omega,E)|^2dt dE d\omega dx
\nonumber\\
&
\leq
d^3\n{{\p \Sigma{x_j}}}_{L^\infty(G\times S\times I)}^2\int_G
\int_S\int_I\int_0^{d}|\ol f(x-t\omega,\omega,E)|^2dt dE d\omega dx
\nonumber\\
&
=
d^3\n{{\p \Sigma{x_j}}}_{L^\infty(G\times S\times I)}^2
\int_0^{d}\int_S\int_I\int_G|\ol f(x-t\omega,\omega,E)|^2dx dE d\omega dt
\nonumber\\
&
\leq 
d^3\n{{\p \Sigma{x_j}}}_{L^\infty(G\times S\times I)}^2
\int_0^{d}\int_S\int_I\int_{\R^3}|\ol f(x-t\omega,\omega,E)|^2dx dE d\omega dt
\nonumber\\
&
=
d^3\n{{\p \Sigma{x_j}}}_{L^\infty(G\times S\times I)}^2
\int_0^{d}\int_S\int_I\int_{\R^3}|\ol f(z,\omega,E)|^2dz dE d\omega dt
\nonumber\\
&
=
d^4\n{{\p \Sigma{x_j}}}_{L^\infty(G\times S\times I)}^2
\int_S\int_I\int_G|f(z,\omega,E)|^2 dz dE d\omega\nonumber\\
&
=d^4\n{{\p \Sigma{x_j}}}_{L^\infty(G\times S\times I)}^2
\n{f}_{L^2(G\times S\times I)}^2
\eea
where we applied the change of variables $z=x-t\omega$.

A.2. For the term $h_2$ we get
\[
&
|h_2(x,\omega,E)|\leq
\int_0^{t(x,\omega)} e^{-\int_0^t\Sigma(x-s\omega,\omega,E)ds}\cdot
\big|{\p f{x_j}}(x-t\omega,\omega,E)\big|dt
\leq
\int_0^{t(x,\omega)} 
\big|{\p f{x_j}}(x-t\omega,\omega,E)\big|dt
\\
&
\leq
\sqrt{d}
\Big(\int_0^{t(x,\omega)}\big|{\p f{x_j}}(x-t\omega,\omega,E)\big|^2dt\Big)^{1/2}
\]
and so as in Part A.1. we have  an estimate
\bea\label{e-2}
&
\n{h_2}_{L^2(G\times S\times I)}^2
\leq
d
\int_G\int_S\int_I\int_0^{t(x,\omega)}\big|{\p f{x_j}}(x-t\omega,\omega,E)\big|^2dt dE d\omega dx
\nonumber\\
&
\leq 
d\int_0^{d}\int_S\int_I\int_{\R^3}\big|\ol {{\p f{x_j}}}(x-t\omega,\omega,E)\big|^2dx dE d\omega dt
=
d^2
\int_S\int_I\int_G\big|{\p f{x_j}}(z,\omega,E)\big|^2 dz dE d\omega \nonumber\\
&
=d^2
\n{{\p f{x_j}}}_{L^2(G\times S\times I)}^2
.
\eea

A.3. For the term $h_3$ we have
\[
|h_3(x,\omega,E)|\leq
\big| f(x-t(x,\omega)\omega,\omega,E)\ {\p t{x_j}}((x,\omega)\big|.
\]
Utilizing Lemma \ref{pr:fubini} and noting that $t(y+t\omega,\omega)=t$ we have
\bea\label{e-3}
&
\n{h_3}_{L^2(G\times S\times I)}^2
\leq
\int_S\int_I\int_G\big| f(x-t(x,\omega)\omega,\omega,E)\ {\p t{x_j}}(x,\omega)\big|^2 dx d\omega dE\nonumber\\
&
=
\int_{\Gamma_-}\int_0^{\tau_-(y,\omega)}\big| f(y+t\omega-t(y+t\omega,\omega)\omega,\omega,E)\ {\p t{x_j}}(y+t\omega,\omega)\big|^2 |\omega\cdot\nu(y)|dt d\sigma(y) d\omega dE\nonumber\\
&
=
\int_{\Gamma_-}\int_0^{\tau_-(y,\omega)}\big| f(y,\omega,E)\ {\p t{x_j}}(y+t\omega,\omega)\big|^2 |\omega\cdot\nu(y)|dt d\sigma(y) d\omega dE
\nonumber\\
&
=
\int_{\Gamma_-}| f(y,\omega,E)|^2\int_0^{\tau_-(y,\omega)}\big| {\p t{x_j}}(y+t\omega,\omega)\big|^2 |\omega\cdot\nu(y)|dt d\sigma(y) d\omega dE
\nonumber\\
&
=\int_{\Gamma_-}| f(y,\omega,E)|^2m_{1,j}(y,\omega) |\omega\cdot\nu(y)| d\sigma(y) d\omega dE=\n{f_{|\Gamma_-}}_{T_{m_{1,j}}^2(\Gamma_-)}^2
\eea

B.
Consider the term $\psi_2$.
\[
\psi_2=
e^{-\int_0^{t(x,\omega)}\Sigma(x-s\omega,\omega,E)ds}g(x-t(x,\omega)\omega,\omega,E).
\]
By the chain rule we have for $(x,\omega,E)\in G\times S\times I$
\bea\label{e-4}
&
{\p {\psi_2}{x_j}}(x,\omega,E)= 
{\partial\over{\partial x_j}}\big(e^{-\int_0^{t(x,\omega)}\Sigma(x-s\omega,\omega,E)ds}\big)\ g(x-t(x,\omega)\omega,\omega,E)\nonumber\\
&
+
e^{-\int_0^{t(x,\omega)}\Sigma(x-s\omega,\omega,E)ds}\la \nabla_{(\partial G)}g(x-t(x,\omega)\omega,\omega,E),e_j-{\p t{x_j}}(x,\omega)\omega\ra
\eea
Note that for a fixed $\omega$ the term $x-t(x,\omega)\omega$ lies in $\partial G$.
Hence the mappings $\xi_j: s\to (x+se_j)-t(x+se_j,\omega)\omega$ are curves in $\partial G$ and so $\xi_j'(0)=e_j-{\p t{x_j}}(x,\omega)\omega\in T(\partial G)$ for $(x,\omega)\in G\times S$. More explicitly, the fact that $e_j-{\p t{x_j}}(x,\omega)\omega\in T(\partial G)$ can be seen as follows.
Let $\omega$ be fixed and let $x_0\in G$. Then there exists a 
neighbourhood $U$ of $x_0$ such that $\{y:=x-t(x,\omega)\omega|x\in U\}$ is a subset of $\partial G$. Let $y_0:=x_0-t(x_0,\omega)$.  Then there exists a neighbourhood $V$ of $y_0$ and a $C^1$-mapping $H:V\to\R$ such that  $\{z\in V|H(z)=0\}=V\cap \partial G$. Hence there exists a neighbourhood $U'$ of $x_0$ such that $H(x-t(x,\omega)\omega)=0$ for $x\in U'$ and so 
\[
{\partial\over{\partial x_j}}\big(H(x-t(x,\omega)\omega)\big)=\la\nabla_zH(x-t(x,\omega)\omega),e_j-{\p t{x_j}}(x,\omega)\omega\ra=0.
\] 
On the other hand, $\nabla_zH(x-t(x,\omega)\omega)\parallel \nu(x-t(x,\omega)\omega)$ which shows that\\ $\la\nu(x-t(x,\omega)\omega),e_j-{\p t{x_j}}(x,\omega)\omega\ra=0$, as desired.

Furthermore,
\[
&
{\partial\over{\partial x_j}}\big(e^{-\int_0^{t(x,\omega)}\Sigma(x-s\omega,\omega,E)ds}\big)\\
&
=
{\partial\over{\partial x_j}}\big(-\int_0^{t(x,\omega)}\Sigma(x-s\omega,\omega,E)ds\big)e^{-\int_0^{t(x,\omega)}\Sigma(x-s\omega,\omega,E)ds}\\
&
=
\Big(-\Sigma(x-t(x,\omega)\omega,\omega,E){\p t{x_j}}(x,\omega)
\\
&
-\int_0^{t(x,\omega)}{\p \Sigma{x_j}}(x-s\omega,\omega,E)ds\Big)
e^{-\int_0^{t(x,\omega)}\Sigma(x-s\omega,\omega,E)ds}.
\]
That is why,
\bea\label{e-5}
&
{\p {\psi_2}{x_j}}(x,\omega,E)= 
-\Sigma(x-t(x,\omega)\omega,\omega,E){\p t{x_j}}(x,\omega)e^{-\int_0^{t(x,\omega)}\Sigma(x-s\omega,\omega,E)ds}
\ g(x-t(x,\omega)\omega,\omega,E)
\\
&
-\int_0^{t(x,\omega)}{\p \Sigma{x_j}}(x-s\omega,\omega,E)ds\ 
e^{-\int_0^{t(x,\omega)}\Sigma(x-s\omega,\omega,E)ds} \ g(x-t(x,\omega)\omega,\omega,E)\nonumber\\
&
+
e^{-\int_0^{t(x,\omega)}\Sigma(x-s\omega,\omega,E)ds}\la \nabla_{(\partial G)}g(x-t(x,\omega)\omega,\omega,E),e_j-{\p t{x_j}}(x,\omega)\omega\ra
\nonumber\\
&
=:q_1+q_2+q_3.
\eea

B.1. Since $\Sigma\geq 0$ and since $x-t(x,\omega)\omega\in\partial G$ we have
for the term $q_1$
\[
&
|q_1(x,\omega,E)|\leq 
|\Sigma(x-t(x,\omega)\omega,\omega,E)|\ 
\ |g(x-t(x,\omega)\omega,\omega,E)|\ \big|{\p t{x_j}}(x,\omega)\big|\\
&
\leq 
\n{\Sigma}_{L^\infty(\Gamma)}
|g(x-t(x,\omega)\omega,\omega,E)|\ \big|{\p t{x_j}}(x,\omega)\big|
\]
and so we get as above by Lemma \ref{pr:fubini}
\bea\label{e-6}
&
\n{q_1}_{L^2(G\times S\times I)}^2
\leq 
\n{\Sigma}_{L^\infty(\Gamma)}^2
\int_S\int_I\int_G
\big|g(x-t(x,\omega)\omega,\omega,E)\big|^2\ \big|{\p t{x_j}}(x,\omega)\big|^2 dx d\omega dE\nonumber\\
&
=
\n{\Sigma}_{L^\infty(\Gamma)}^2
\int_{\Gamma_-}\int_0^{\tau_-(y,\omega)}
|g(y,\omega,E)|^2\ \big|{\p t{x_j}}(y+t\omega,\omega)\big|^2|\omega\cdot\nu(y)| dt\ d\sigma(y) d\omega dE\nonumber\\
&
=\n{\Sigma}_{L^\infty(\Gamma)}^2\n{g}_{T^2_{m_{1,j}}(\Gamma_-)}^2.
\eea

B.2. 
For the term $q_2$ we have (since $\Sigma\geq 0$ and $0\leq t(x,\omega)\leq d$)
\[
&
|q_2(x,\omega,E)|\leq 
\int_0^{t(x,\omega)}\big|{\p \Sigma{x_j}}(x-s\omega,\omega,E)\big|ds\ 
 \ |g(x-t(x,\omega)\omega,\omega,E)|\\
&
\leq 
d \n{{\p \Sigma{x_j}}}_{L^\infty(G\times S\times I)}\ |g(x-t(x,\omega)\omega,\omega,E)|
\] 
and so 
\bea\label{e-7}
&
\n{q_2}_{L^2(G\times S\times I)}^2
\leq
\int_G\int_S\int_I
d^2 \n{{\p \Sigma{x_j}}}_{L^\infty(G\times S\times I)}^2\ |g(x-t(x,\omega)\omega,\omega,E)|^2 dE d\omega dx\nonumber\\
&
=
d^2 \n{{\p \Sigma{x_j}}}_{L^\infty(G\times S\times I)}^2\n{L_-g}_{L^2(G\times S\times I)}^2
\eea
where $L_-g$ is the lift given in section \ref{fs} (see Remark \ref{lift}).

B.3. 
For the term $q_3$ we obtain  (since $\Sigma\geq 0$)
\[
&
|q_3(x,\omega,E)|=
e^{-\int_0^{t(x,\omega)}\Sigma(x-s\omega,\omega,E)ds}\big|\la \nabla_{(\partial G)}g(x-t(x,\omega)\omega,\omega,E),e_j-{\p t{x_j}}(x,\omega)\omega\ra\big|
\\
&
\leq 
\big|\la \nabla_{(\partial G)}g(x-t(x,\omega)\omega,\omega,E),e_j-{\p t{x_j}}(x,\omega)\omega\ra\big|
\]
which implies by Lemma \ref{pr:fubini} (since $t(y+t\omega,\omega)=t$)
\bea\label{e-8}
&
\n{q_3}_{L^2(G\times S\times I)}^2
\leq 
\int_S\int_I\int_G\big|\la \nabla_{(\partial G)}g(x-t(x,\omega)\omega,\omega,E),e_j-{\p t{x_j}}(x,\omega)\omega\ra\big|^2 dx dE d\omega \nonumber\\
&
=
\int_{\Gamma_-}\int_0^{\tau_-(y,\omega)}\big|\la \nabla_{(\partial G)}g(y,\omega,E),e_j-{\p t{x_j}}(y+t\omega,\omega)\omega\ra\big|^2|\omega\cdot\nu(y)| dt d\sigma(y) d\omega dE.
\eea
Since
\[
&
\big|\la \nabla_{(\partial G)}g(y,\omega,E),e_j-{\p t{x_j}}(y+t\omega,\omega)\omega\ra\big|
\leq 
\n{\nabla_{(\partial G)}g(y,\omega,E)}
\n{e_j-{\p t{x_j}}(y+t\omega,\omega)\omega}
\\
&
\leq 
\n{\nabla_{(\partial G)}g(y,\omega,E)}
(1+\big|{\p t{x_j}}(y+t\omega,\omega)\big|)
\]
we see that
\bea\label{e-9}
&
\n{q_3}_{L^2(G\times S\times I)}^2
\leq 
\int_{\Gamma_-}\int_0^{\tau_-(y,\omega)}(\n{\nabla_{(\partial G)}g(y,\omega,E)}
(1+\big|{\p t{x_j}}(y+t\omega,\omega)\big|))^2|\omega\cdot\nu(y)| dt d\sigma(y) d\omega dE
\nonumber\\
&
\leq
2\big(d\n{\nabla_{(\partial G)}g}_{{T^2(\Gamma_-)}^2}^2+
\n{\nabla_{(\partial G)}g}_{{T_{m_{1,j}}^2(\Gamma_-)}^2}^2\big).
\eea

From (\ref{e-1}), (\ref{e-2}), (\ref{e-3}), (\ref{e-6}), (\ref{e-7}), (\ref{e-9}) we obtain the estimate
\bea\label{e-10}
&
\n{{\p \psi{x_j}}}_{L^2(G\times S\times I)}\leq \n{{\p {\psi_1}{x_j}}}_{L^2(G\times S\times I)}+\n{{\p {\psi_2}{x_j}}}_{L^2(G\times S\times I)}\nonumber\\
&
\leq 
\n{h_1}_{L^2(G\times S\times I)}+\n{h_2}_{L^2(G\times S\times I)}
+\n{h_3}_{L^2(G\times S\times I)}\nonumber\\
&
+\n{q_1}_{L^2(G\times S\times I)}
+\n{q_2}_{L^2(G\times S\times I)}+\n{q_3}_{L^2(G\times S\times I)}\nonumber\\
&
\leq 
d^2\n{{\p \Sigma{x_j}}}_{L^\infty(G\times S\times I)}
\n{f}_{L^2(G\times S\times I)}
+
d
\n{{\p f{x_j}}}_{L^2(G\times S\times I)}
+
\n{f_{|\Gamma_-}}_{T_{m_{1,j}}^2(\Gamma_-)}\nonumber\\
&
+
\n{\Sigma}_{L^\infty(\Gamma)}\n{g}_{T^2_{m_{1,j}}(\Gamma_-)}
+
d \n{{\p \Sigma{x_j}}}_{L^\infty(G\times S\times I)}\n{L_-g}_{L^2(G\times S\times I)}\nonumber\\
&
+
\sqrt{2}\sqrt{d}\n{\nabla_{(\partial G)}g}_{T^2(\Gamma_-)^2}+
\sqrt{2}\n{\nabla_{(\partial G)}g}_{T_{m_{1,j}}^2(\Gamma_-)^2}.
\eea
Hence ${\p \psi{x_j}}\in L^2(G\times S\times I)$ for $j=1,2,3$ that is, $\psi\in H^{(1,0,0)}(G\times S\times I^\circ)$ which completes the proof.

\end{proof}

The next example shows that our assumptions using weighted space data is reasonable. 

\begin{example}\label{rbpex1-ss}

For an open ball $G=B(0,R)\subset\R^3$ we have
\[
t(x,\omega)= x\cdot\omega+\sqrt{ (x\cdot\omega) ^2+R^2-|x|^2}
\]
and so
\[
{\p t{x_j}}(x,\omega)=\omega_j+{{ (x\cdot\omega )\omega_j-x_j}\over{\sqrt{ (x\cdot\omega) ^2+R^2-|x|^2}}}.
\]
In addition, for $(y,\omega)\in \Gamma_-$ we have 
\[
\nu(y)=y/R\  {\rm and }\ 
\tau_-(y,\omega)=2|y\cdot\omega|.
\]

Recall that
\[
m_{1,j}(y,\omega)=
\int_0^{\tau_-(y,\omega)}\big| {\p t{x_j}}((y+t\omega,\omega)\big|^2dt,\ j=1,2,3
\]
and the weighted norm
\[
\n{g}_{T_{m_{1,j}}^2(\Gamma_-)}^2:=\int_{\Gamma_-}|g(y,\omega,E)|^2m_{1,j}(y,\omega)|\omega\cdot\nu(y)| d\sigma(y) d\omega dE.
\]
Noting that on $\Gamma_-$ 
\[
|y\cdot\omega|=R|\nu(y)\cdot\omega|=- R(\nu(y)\cdot\omega)=-y\cdot\omega
\] 
we obtain for $(y,\omega)\in \Gamma_-$ (since additionally $|y|=R$ and $\tau_-(y,\omega)=2|y\cdot\omega|$)
\bea
&
m_{1,j}(y,\omega)=
\int_0^{\tau_-(y,\omega)}\Big|
{\p t{x_j}}(y+t\omega,\omega)\Big|^2 dt
=
\int_0^{\tau_-(y,\omega)}\Big|\omega_j+{{(y\cdot\omega)\omega_j-y_j}\over{|y\cdot\omega|}}\big|^2dt
\nonumber\\
&
=2{{y_j^2}\over{|\omega\cdot y|}}={2\over R}{{y_j^2}\over{|\omega\cdot\nu(y)|}}.
\eea
Hence 
\be\label{gnorm-a}
\n{g}_{T_{m_{1,j}}^2(\Gamma_-)}^2={2\over R}\int_{\Gamma_-}|g(y,\omega,E)|^2y_j^2 d\sigma(y) d\omega dE
\ee
which implies that for example,
\[
L^2(\Gamma_-)\subset T^2_{m_{1,j}}(\Gamma_-).
\]
Note also that $L^2(\Gamma_-)\subset T^2(\Gamma_-)$.

In the setup of this example we have by  Theorem \ref{foregth1} the following regularity result:

Assume that
\be\label{ashr4-d}
\Sigma\in W^{\infty,(1,0,0)}(G\times S\times I^\circ)\cap L^\infty(\Gamma_-),\ \Sigma\geq c>0,
\ee
\be\label{ashr1-a-d}
g\in L^2(\Gamma_-), 
\ee
\be\label{ashr1-b-d} 
\nabla_{(\partial G)}g\in L^2(\Gamma_-)^2,
\ee
\be\label{ashr2-a-d}
f\in H^{(1,0,0)}(G\times S\times I^\circ),
\ee
\be\label{ashr2-a-d-b}
f_{|\Gamma-}\in L^2(\Gamma_-).
\ee
Then the solution $\psi$ of the problem (\ref{stif1}) belongs to
$H^{(1,0,0)}(G\times S\times I^\circ)$.

\end{example}

\subsubsection{First order regularity with respect to angular variable}. \label{foregan}

Next we consider regularity with respect to $\omega$-variable. We give only outlines since the details are analogous to the analysis of the previous section \ref{foreg}.
Define for $j=1,2$  \emph{weight functions} 
\[
m_{2,j}(y,\omega):=
\int_0^{\tau_-(y,\omega)}\big| {\p t{\omega_j}}((y+t\omega,\omega)\big|^2dt
\]
and  weighted norms
\[
\n{g}_{T_{m_{2,j}}^2(\Gamma_-)}^2:=\int_{\Gamma_-}|g(y,\omega,E)|^2m_{2,j}(y,\omega)|\omega\cdot\nu(y)| d\sigma(y) d\omega dE.
\]
Let
\[
T_{m_{2,j}}^2(\Gamma_-):=\{g\in L_{\rm loc}^2(\Gamma_-)|\ \n{g}_{T_{m_{2,j}}^2(\Gamma_-)}<\infty\}.
\]

We have the following theorem which is analogous to the above Theorem \ref{foregth1}.

\begin{theorem}\label{foregth2}
Assume that 
\be\label{om-1}
\Sigma\in W^{\infty,(1,0,0)}(G\times S\times I^\circ)\cap W^{\infty,(0,1,0)}(G\times S\times I^\circ)\cap L^\infty(\Gamma_-),\ \Sigma\geq c>0,
\ee
\be\label{om-2}
g\in T^2(\Gamma_-)\cap T^2_{m_{2,j}}(\Gamma_-), \ j=1,2
\ee
\be\label{om-3} 
\nabla_{(\partial G)}g\in T^2(\Gamma_-)^2\cap T^2_{m_{2,j}}(\Gamma_-)^2,\ j=1,2,
\ee
\be\label{om-3-a} 
{\p {g}{\omega_j}}\in T^2(\Gamma_-),\ j=1,2,
\ee
\be\label{om-4}
f\in H^{(1,0,0)}(G\times S\times I^\circ)\cap H^{(0,1,0)}(G\times S\times I^\circ),
\ee
\be\label{om-5}
f_{|\Gamma-}\in T^2_{m_{2,j}}(\Gamma_-),\ j=1,2.
\ee
Then the solution $\psi$ of the problem (\ref{stif1}) belongs to
$H^{(0,1,0)}(G\times S\times I^\circ)$.
\end{theorem}

\begin{proof}
Denote again
\[
\psi_1:=\int_0^{t(x,\omega)} e^{-\int_0^t\Sigma(x-s\omega,\omega,E)ds}f(x-t\omega,\omega,E)dt,
\]
\[
\psi_2:=
e^{-\int_0^{t(x,\omega)}\Sigma(x-s\omega,\omega,E)ds}g(x-t(x,\omega)\omega,\omega,E).
\]
Then by (\ref{stif2}) $\psi=\psi_1+\psi_2$.

Consider the term $\psi_1$.  
We have by similar computations as in (\ref{hreg2})
\bea\label{om-6}
{\p {\psi_1}{\omega_j}}
&=\int_0^{t(x,\omega)}\big(-\int_0^t\Big(\la -s\ol\Omega_j(\omega), \nabla_x\Sigma\ra+{{\partial \Sigma}\over{\partial\omega_j}}\Big)(x-s\omega,\omega,E)ds\big)\nonumber\\
&
\cdot e^{-\int_0^t\Sigma(x-s\omega,\omega,E)ds}\cdot
f(x-t\omega,\omega,E)dt\nonumber\\
&
+\int_0^{t(x,\omega)} e^{-\int_0^t\Sigma(x-s\omega,\omega,E)ds} \Big(\la -t\ol\Omega_j(\omega), \nabla_xf\ra+{{\partial f}\over{\partial\omega_j}}\Big)(x-t\omega,\omega,E)dt
\nonumber\\
&
+
e^{-\int_0^{t(x,\omega)}\Sigma(x-s\omega,\omega,E)ds}\ f(x-t(x,\omega)\omega,\omega,E)\ {\p t{\omega_j}}((x,\omega)\nonumber\\
&
=:H_1+H_2+H_3
\eea
where $\ol\Omega_j(\omega)\in T(S)$ are
the tangent vectors at $\omega$. Note that  ${\p {I_{S}}{\omega_j}}_{\big|\omega}=\ol\Omega_j(\omega)$ where  $I_S:S\to S$ is the identity mapping and $\ol\Omega_j\in T_\omega(S)$ are tangent vectors at $\omega\in S$ that is,
\[
&
\ol\Omega_1(\omega)=(-\omega_2,\omega_1,0),\\
&
\ol\Omega_2(\omega)
=\big({{\omega_1\omega_3}\over{\sqrt{\omega_1^2+\omega_2^2}}}, {{\omega_2\omega_3}\over{\sqrt{\omega_1^2+\omega_2^2}}},-\sqrt{1-\omega_3^2}\big)
\]
when we use the parametrization $h$ of $S$ given in the below Example \ref{rbpex1-aa}.

Furthermore, similarly as in (\ref{e-5})
\bea\label{e-5-a}
&
{\p {\psi_2}{\omega_j}}(x,\omega,E)= 
-\Sigma(x-t(x,\omega)\omega,\omega,E){\p t{\omega_j}}(x,\omega)
e^{-\int_0^{t(x,\omega)}\Sigma(x-s\omega,\omega,E)ds} 
\ g(x-t(x,\omega)\omega,\omega,E)\nonumber
\\
&
+\int_0^{t(x,\omega)}\Big(\la -s\ol\Omega_j(\omega), \nabla_x\Sigma\ra+{{\partial \Sigma}\over{\partial\omega_j}}\Big)(x-s\omega,\omega,E)ds\  e^{-\int_0^{t(x,\omega)}\Sigma(x-s\omega,\omega,E)ds}
 \ g(x-t(x,\omega)\omega,\omega,E)\nonumber\\
&
+
e^{-\int_0^{t(x,\omega)}\Sigma(x-s\omega,\omega,E)ds}\la \nabla_{(\partial G)}g(x-t(x,\omega)\omega,\omega,E),{\p t{\omega_j}}(x,\omega)\omega+
t(x,\omega)\ol\Omega_j(\omega)\ra
\nonumber\\
&
+e^{-\int_0^{t(x,\omega)}\Sigma(x-s\omega,\omega,E)ds}{\p g{\omega_j}}(x-t(x,\omega)\omega,\omega,E)
=:Q_1+Q_2+Q_3+Q_4.
\eea

Utilizing similar estimations as in the proof of Theorem \ref{foregth1}
the terms $H_1,\ H_2,\ H_3$ and $Q_1,\ Q_2,\ Q_3,\ Q_4$ can be seen to lie in $L^2(G\times S\times I)$.
Hence  ${\p \psi{\omega_j}}\in L^2(G\times S\times I)$ for $j=1,2$ that is, $\psi\in H^{(0,1,0)}(G\times S\times I^\circ)$. We omit further details.

\end{proof}

\begin{example}\label{rbpex1-aa}
Consider the transport problem of Example \ref{rbpex1-ss}.
Recall that 
\[
{\p t{\omega_1}}_{\Big|\omega}={\partial\over{\partial\phi}}(t\circ h)_{\Big|(\phi,\theta)=h^{-1}(\omega)},\ 
{\p t{\omega_2}}_{\Big|\omega}={\partial\over{\partial\theta}}(t\circ h)_{\Big|(\phi,\theta)=h^{-1}(\omega)}
\]
where $h:]0,2\pi[\times ]0,\pi[\to S$ is the parametrization
\[
h(\phi,\theta)=(R\cos\phi\sin\theta,R\sin\phi\sin\theta,R\cos\theta).
\]
By routine computation we find that
\bea
&
{\p t{\omega_1}}_{\Big|\omega}=
-x_1\omega_2+x_2\omega_1+
{{(-x_1\omega_2+x_2\omega_1)(x\cdot\omega)}\over{\sqrt{(x\cdot\omega)^2+R^2-|x|^2}}}, 
\nonumber\\
&
{\p t{\omega_2}}_{\Big|\omega}=
{{(x\cdot\omega) \omega_3-x_3}\over{\sqrt{\omega_1^2+\omega_2^2}}}
+{{((x\cdot\omega) \omega_3-x_3)(x\cdot\omega)}\over{\sqrt{\omega_1^2+\omega_2^2}
\sqrt{(x\cdot\omega)^2+R^2-|x|^2}}}.
\eea
Hence we have 
\be 
{\p t{\omega_1}}(y+t\omega,\omega)=
-y_1\omega_2+y_2\omega_1+{{-y_1\omega_2+y_2\omega_1}\over{|y\cdot\omega|}}((y\cdot\omega)+t)
\ee
and
\be\label{hreg10-b}
{\p t{\omega_2}}(y+t\omega,\omega)=
{{ (y\cdot\omega) \omega_3-y_3}\over{\sqrt{\omega_1^2+\omega_2^2}}}
+{{(y\cdot\omega) \omega_3-y_3}\over{\sqrt{\omega_1^2+\omega_2^2}
|( y\cdot\omega)|}}((y\cdot\omega)+t).
\ee 
Recalling that in this example $\tau_-(y,\omega)=2|y\cdot\omega|$, $ \nu(y)=y/R$ and noting that on $\Gamma_-$
\[
{{\omega\cdot\nu}\over{|\omega\cdot\nu|}}=-1
\]
we find that on $\Gamma_-$
\[
& 
m_{2,1}(y,\omega)=\int_0^{\tau_-(y,\omega)}\Big|
{\p t{\omega_1}}(y+t\omega,\omega)\Big|^2 dt  \\
&
=\int_0^{2|y\cdot\omega|}\Big(
-y_1\omega_2+y_2\omega_1+{{-y_1\omega_2+y_2\omega_1}\over{|y\cdot\omega|}}((y\cdot\omega)+t)\Big)^2 dt\\
&
=\int_0^{2|y\cdot\omega|}\Big((-y_1\omega_2+y_2\omega_1){t\over{|y\cdot\omega|}}
\Big)^2dt\\
&
={8\over 3}(-y_1\omega_2+y_2\omega_1)^2|y\cdot\omega|
\]
and similarly on $\Gamma_-$
\[
&
m_{2,2}(y,\omega)=\int_0^{\tau_-(y,\omega)}\Big|
{\p t{\tilde\omega_2}}(y+t\omega,\omega)\Big|^2 dt\\
&
=\int_0^{2|y\cdot\omega|}\Big(
{{ (y\cdot\omega) \omega_3-y_3}\over{\sqrt{\omega_1^2+\omega_2^2}}}
+{{(y\cdot\omega) \omega_3-y_3}\over{\sqrt{\omega_1^2+\omega_2^2}
|( y\cdot\omega)|}}((y\cdot\omega)+t)\Big)^2 dt\\
&
={8\over 3}
\Big|{{ (y\cdot\omega) \omega_3-y_3}\over{\sqrt{\omega_1^2+\omega_2^2}}}\Big|^2
|y\cdot\omega|.
\]

Since
\bea
&
\Big|
{{(y\cdot\omega) \omega_3-y_3}\over{\sqrt{\omega_1^2+\omega_2^2}}}\Big|^2=
\Big|
{{y_1\omega_1\omega_3+y_2\omega_2\omega_3+y_3\omega_3^2-y_3}\over{\sqrt{\omega_1^2+\omega_2^2}}}
\Big|^2\nonumber\\
&
=
\Big|
{{y_1\omega_1\omega_3+y_2\omega_2\omega_3+y_3(1-\omega_1^2-\omega_2^2)-y_3}\over{\sqrt{\omega_1^2+\omega_2^2}}}
\Big|^2
=\Big|
{{y_1\omega_1\omega_3+y_2\omega_2\omega_3-y_3(\omega_1^2+\omega_2^2)}\over{\sqrt{\omega_1^2+\omega_2^2}}}
\Big|^2
\nonumber\\
&
\leq (3R)^2
\eea
we see that on $\Gamma_-$
\be\label{ub} 
m_{2,1}(y,\omega)\leq {8\over 3}R^2|y\cdot\omega|,\ m_{2,2}(y,\omega)\leq {72\over 3}R^2|y\cdot\omega|
\ee
and so, for example
\be\label{ss}
T^2(\Gamma_-)\subset T^2_{m_{2,j}}(\Gamma_-).
\ee
Hence we are able to formulate  a regularity result analogous to Example \ref{rbpex1-ss} by replacing $L^2(\Gamma_-)$ therein (more loosely) with $T^2(\Gamma_-)$.
\end{example}

\subsubsection{Regularity with respect to energy variable}\label{rev}

The regularity with respect to energy variable does not need any assumptions on the escape time mapping $t(x,\omega)$ and we obtain results for a general Sobolev index $m_3\in\N$. The solution $\psi$ may be as regular as data with respect to $E$-variable.

\begin{theorem}\label{foregth3}
Let $m_3\in\N$. 
Assume that  
\be\label{ashr4-ea}
\Sigma\in W^{\infty,(0,0,m_3)}(G\times S\times I^\circ),\ \Sigma\geq c>0,
\ee
\be\label{ashr2-ea}
f\in H^{(0,0,m_3)}(G\times S\times I^\circ),
\ee
\be\label{ashr1-ea}
g, \ {\p { g}E},\ ..., {{\partial^{m_3} g}\over{\partial E^{m_3}}}\in T^2(\Gamma_-).
\ee
Then the solution $\psi$ of the problem (\ref{stif1}) belongs to
$H^{(0,0,m_3)}(G\times S\times I^\circ)=H^{m_3}(I^\circ,L^2(G\times S))$.
\end{theorem} 

\begin{proof}
We only outline the proof since the technicalities are again similar to section \ref{foreg}.
Denote as above
\[
\psi_1:=\int_0^{t(x,\omega)} e^{-\int_0^t\Sigma(x-s\omega,\omega,E)ds}f(x-t\omega,\omega,E)dt,
\]
\[
\psi_2:=
e^{-\int_0^{t(x,\omega)}\Sigma(x-s\omega,\omega,E)ds}g(x-t(x,\omega)\omega,\omega,E).
\]
Then by (\ref{stif2}) $\psi=\psi_1+\psi_2$. 

At first consider the assertion for $m_3=1$.
We have
\bea\label{hreg2-e}
{\p {\psi_1}{E}}
&=\int_0^{t(x,\omega)}\big(-\int_0^t{\p \Sigma{E}}(x-s\omega,\omega,E)ds\big)
\cdot e^{-\int_0^t\Sigma(x-s\omega,\omega,E)ds}\cdot
f(x-t\omega,\omega,E)dt\nonumber\\
&
+\int_0^{t(x,\omega)} e^{-\int_0^t\Sigma(x-s\omega,\omega,E)ds} {\p {f}{E}}(x-t\omega,\omega,E)dt
=:U_1+U_2
\eea
and
\bea\label{e-5-e}
&
{\p {\psi_2}{E}}(x,\omega,E)= 
-\int_0^{t(x,\omega)}{\p \Sigma{E}}(x-s\omega,\omega,E)ds\ 
e^{-\int_0^{t(x,\omega)}\Sigma(x-s\omega,\omega,E)ds} \ g(x-t(x,\omega)\omega,\omega,E)\nonumber\\
&
+e^{-\int_0^{t(x,\omega)}\Sigma(x-s\omega,\omega,E)ds}{\p g{E}}(x-t(x,\omega)\omega,\omega,E)
 \nonumber\\
&
=:P_1+P_2.
\eea
Applying the same kind of estimations as in section \ref{foreg} the terms
$U_1,\ U_2,\ P_1,\ P_2$ can be shown to belong to $L^2(G\times S\times I)$. and so
$\psi\in H^{(0,0,1)}(G\times S\times I^\circ)$. This proves the claim for $m_3=1$.

In the general case we have by the Leibniz's formula
\bea\label{gen-f} 
&
\partial_E^{m_3}\psi=\int_0^{t(x,\omega)}\sum_{k=0}^{m_3}{{m_3}\choose k}\partial_E^k\Big(
e^{-\int_0^{t}\Sigma(x-s\omega,\omega,E)ds} \Big) \cdot (\partial_E^{m_3-k}f)(x-t(x,\omega)\omega,\omega,E)\nonumber\\
&
+
\sum_{k=0}^{m_3}{{m_3}\choose k}\partial_E^k\Big(
e^{-\int_0^{t(x,\omega)}\Sigma(x-s\omega,\omega,E)ds} \Big) \cdot (\partial_E^{m_3-k}g)(x-t(x,\omega)\omega,\omega,E)
\eea
We see iteratively by the assumption (\ref{ashr4-ea}) that 
\[
\partial_E^k\Big(
e^{-\int_0^{t}\Sigma(x-s\omega,\omega,E)ds} \Big),\  
\partial_E^k\Big(
e^{-\int_0^{t(x,\omega)}\Sigma(x-s\omega,\omega,E)ds} \Big)\in L^\infty(G\times S\times I),\ 0\leq k\leq m_3
\]
and then   we are able to conclude the claim by using analogous estimations as in section \ref{foreg}. We omit further details. .

\end{proof}

\begin{example}\label{rbpex1-e}
The solution $\psi$ of the transport problem considered in Example \ref{rbpex1-ss}
belongs to $H^{(0,0,m_3)}(G\times S\times I^\circ)$ when (\ref{ashr4-ea}), (\ref{ashr2-ea}), (\ref{ashr1-ea}) hold.
\end{example}

We give finally the following corollary

\begin{corollary}\label{rbpco1}
Suppose that the assumptions of Theorems \ref{foregth1}, \ref{foregth2} and \ref{foregth3} are valid (with $m_3=1$). Then the solution of the transport problem
\be\label{++}
\omega\cdot\nabla_x\psi+\Sigma\psi=f,\ \psi_{|\Gamma-}=g
\ee
belongs to $H^1(G\times S\times I^\circ)$ where  $H^1(G\times S\times I^\circ)$ denotes the usual first order Sobolev space.
\end{corollary}

\begin{proof} Since 
\[
H^{(1,0,0)}(G\times S\times I^\circ)\cap H^{(0,1,0)}(G\times S\times I^\circ)\cap H^{(0,0,1)}(G\times S\times I^\circ)
=H^1(G\times S\times I^\circ)
\]
the claim follows immediately from Theorems \ref{foregth1}, \ref{foregth2} and \ref{foregth3}.

\end{proof}

\subsection{On regularity of solutions for the differential part of continuous slowing down equation}\label{csda-ex}

In certain cases the explicit formulas are also available for the equations corresponding to the differential part of the CSDA-Boltzmann equation. In the sequel we shall delineate one such situation.

Suppose that $\Sigma(x,\omega,E)=\Sigma(x,\omega)$ (i.e. $\Sigma$ does not depend on $E$), $\Sigma\in L^\infty(G\times S)$, $\Sigma\geq c>0$.
Furthermore, suppose that $a(x,E)=a(E)$ (i.e. $a$ is independent of $x$ and $\omega$) and that $a:I\to \R_+$ is a continuous, strictly positive function. Finally, suppose that $E_0=0$. Let $f\in H^1(I,L^2(G\times S))$, $g\in H^1(I,T^2(\Gamma'_-))$
such that $g(E_{\rm m})=0$ (the compatibility condition).
Define $R:I\to\R$ by 
\be 
R(E):=\int_0^E{1\over{a(\tau)}}d\tau.
\ee
Let $r_{m}:=R(E_m)$. Then $R:I\to [0,r_{m}]$ is a continuously differentiable and strictly increasing bijection. Let $R^{-1}: [0,r_{m}]\to I$ be its inverse.
We denote the argument of $R^{-1}$ on $[0,r_{m}]$ by $\eta$,
i.e. $E=R^{-1}(\eta)$ (or equivalently $\eta=R(E)$).

Consider the  problem of the form,
\be\label{desol19-c}
-{\p {(a\psi)}E}+\omega\cdot\nabla_x\psi+
\Sigma\psi
= f,\ \
\psi_{|\Gamma_-}=g,\quad 
\psi(\cdot,\cdot,E_m)=0.  
\ee
Let $H:\R\to\R$ be the Heaviside function $H(x)=\begin{cases} 1,\ &x\geq 1\\ 0,\ &x<1\end{cases}$.
In \cite{tervo18-up}, section 10 
(see also \cite[pp. 226-235]{dautraylionsv6}, \cite[section 7]{tervo17-up}) 
we verified that the solution $\psi$ is
\[
\psi=\psi_1+\psi_2
\]
where 
\bea\label{ts-1}
&
\psi_1(x,\omega,E):=
{1\over{a(E)}}\Big[
\int_0^{r_{m}-R(E)}
e^{-\int_0^{r_{m}-R(E)-s} \Sigma(x-\tau\omega,\omega) d\tau}
\nonumber\\
&
\cdot
H(t(x,\omega)-(r_{m}-R(E)-s))a(R^{-1}(r_m-s)) f(x-(r_{m}-R(E)-s)\omega,\omega,R^{-1}(r_{m}-s)) ds\Big]\nonumber\\
&
=
{1\over{a(E)}}\Big[
\int_0^{\min\{r_{m}-R(E),t(x,\omega)\}}
e^{-\int_0^{s} \Sigma(x-\tau\omega,\omega) d\tau}
\nonumber\\
&
\cdot
a(R^{-1}(R(E)+s)) f(x-s\omega,\omega,R^{-1}(R(E)+s)) ds\Big],
\eea
and 
\bea\label{ts-2}
&
\psi_2(x,\omega,E):=
{1\over{a(E)}}H(r_{m}-R(E)-t(x,\omega))e^{\int_0^{t(x,\omega)}-\Sigma(x-s\omega,\omega)ds}
\nonumber\\
&
\cdot a(R^{-1}(R(E)+t(x,\omega)))
{g}(x-t(x,\omega)\omega,\omega,R^{-1}(R(E)+t(x,\omega)))
.
\eea
In the second step of (\ref{ts-1}) we applied the change of variables $r_m-R(E)-s=s'$ and noticed that
\[
&
-\int_{r_m-R(E)}^0H(t(x,\omega)-s')
e^{-\int_0^{s'} \Sigma(x-\tau\omega,\omega) d\tau}
\cdot
a(R^{-1}(R(E)+s')) f(x-s'\omega,\omega,R^{-1}(R(E)+s')) ds\\
=
&
\int_0^{\min\{r_{m}-R(E),t(x,\omega)\}}
e^{-\int_0^{s} \Sigma(x-\tau\omega,\omega) d\tau}
\cdot
a(R^{-1}(R(E)+s)) f(x-s\omega,\omega,R^{-1}(R(E)+s)) ds.
\]

As well-known in the case of initial boundary value problems in order
to retrieve relevant regularity results (with respect to evolving variable) one must impose certain {\it compatibility conditions for the data} (see e.g. \cite{rauch74}). In our case these conditions mean the following (heuristic) requirements.

The zero order compatibility condition is as follows. 
The problem  (\ref{desol19-c}) is
\be\label{csda-ex6}
-a{\p {\psi}E}+\omega\cdot\nabla_x\psi+
\Sigma\psi-{\p a{E}}\psi
= f,\ 
\psi_{|\Gamma_-}=g,\quad 
\psi(\cdot,\cdot,E_m)=0.
\ee
In order that $\psi$ is continuous up to the boundary the conditions $\psi_{|\Gamma_-}=g,\quad 
\psi(\cdot,\cdot,E_m)=0$ imply that
\be\label{comp-1}
g(E_m)=g(.,.,E_m)=0.
\ee

Next consider the first order compatibility condition. In the case where the data and the solution are regular enough the validity of the equations (\ref{csda-ex6}) implies that (here we denote $g(E)=g(y,\omega,E)$ and so on)
\be\label{csda-ex7}
-a(E_m){\p {g}E}(E_m)+(\omega\cdot\nabla_xg)(E_m)+
\Sigma(E_m)_{|\Gamma_-}g(E_m)-{\p a{E}}(E_m)_{|\Gamma_-}g(E_m)=f(E_m)_{|\Gamma_-}
\ee
which by (\ref{comp-1}) reduces to
\be \label{csda-ex9}
-a(E_m){\p {g}E}(E_m)+(\omega\cdot\nabla_xg)(E_m)
=f(E_m)_{|\Gamma_-}.
\ee
In the case  where $ g\in C^1(I,W^2(G\times S))$ the condition (\ref{comp-1}) implies that $(\omega\cdot\nabla_xg)(E_m)=0$ and so (\ref{csda-ex9}) reduces to
\be \label{csda-ex9-a}
-a(E_m){\p {g}E}(E_m)
=f(E_m)_{|\Gamma_-}.
\ee
 
Analogously demanding higher order regularity we need to impose higher order compatibility conditions (we omit formulations of the general order compatibility conditions which could be done by applying the Leibniz's rule).

Results similar to section \ref{reg-up-to} can be deduced for the problem (\ref{desol19-c})
by using the solution formulas  (\ref{ts-1}), (\ref{ts-2}). 
We omit  detailed  computations and contempt ourselves only to conjecture  the following regularity result:

Assume  that   $a=a(E)$ is independent of $x$ and $\Sigma=\Sigma(x,\omega)$ is independent of $E$. 
Suppose as in Corollary \ref{rbpco1} that the assumptions of Theorems \ref{foregth1}, \ref{foregth2} and \ref{foregth3} are valid (with $m_3=1$). 
Furthermore, suppose that
\be\label{infs0}
a\in W^{\infty,1}(I),\ a(E)\geq \kappa>0
\ee
and
\be\label{comp-bd1}
 g(E_m)=0.
\ee
Then the solution $\psi$ of the transport problem  (\ref{desol19-c})
belongs to $H^1(G\times S\times I^\circ)$.

\begin{proof}
Using the explicit formulas (\ref{ts-1}), (\ref{ts-2}) the proof follows by the analogous computations as in sections \ref{foreg}, \ref{foregan} and \ref{rev}. 
Some details in a special case for the first term (\ref{ts-1}) will be given in section \ref{csda-special}
(especially Lemma \ref{sc-le1} therein). Considerations of $H^1$-differentiability on the surface $r_m-R(E)=t(x,\omega)$  require special attention as treated in the beginning of section  \ref{csda-special}. The second term (\ref{ts-2}) is somewhat more subtle to handle and we note that in its treatise it is essential that for $E_{m}-E-t(x,\omega)$ one must have ${g}(x-t(x,\omega)\omega,\omega,E+t(x,\omega))=0$ which means that the compatibility condition (\ref{comp-bd1}) holds.
We omit the detailed proof.
\end{proof}

In the case where $a=1$ and $\Sigma$ is  constant the solution is
\[
&
\psi(x,\omega,E)=\int_0^{\min\{E_m-E,t(x,\omega)\}} e^{-\Sigma s} f\big(x-s\omega,\omega,E+s\big)ds\\
+
&
H(E_{m}-E-t(x,\omega))e^{-\Sigma t(x,\omega)}
\cdot 
{g}(x-t(x,\omega)\omega,\omega,E+t(x,\omega))
\]
from which it is more easier to conclude the above claim.

\begin{remark}\label{re-comp1}
If we assure that 
\be\label{more-reg}
\psi\in C^1(I,H^1(G\times S))
\ee
instead of $\psi\in H^1(G\times S\times I^\circ)$ then the additional compatibility condition (\ref{csda-ex9-a}) is necessary (see Remark \ref{re-comp2} below). Sufficient 
criteria for the regularity (\ref{more-reg}) may be retrieved by applying the theory of evolution equations (cf. \cite{tervo17}, Theorem 3.6).
\end{remark}

\begin{remark}\label{frac}
  
The previous considerations can be generalized for results formulated in \emph{fractional anisotropic Sobolev-Slobodevskij spaces.}
For example, for $0<s<1$ the space $H^{(1+s,0,0)}(G\times S\times I^\circ)$ is the subspace of 
$H^{(1,0,0)}(G\times S\times I^\circ)$ for which
\be\label{fr-b}
\int_G\int_{G'}\int_S\int_I{{|{\p \psi{x_j}}(x,\omega,E)-{\p \psi{x_j}}(x',\omega,E)|^2}\over{|x-x'|^{3+2s}}} dE d\omega dx' dx<\infty.
\ee
For instance, the solution of the problem (\ref{stif1})
is $\psi=\psi_1+\psi_2$ for which by the proof of Theorem \ref{foregth1}
\bea\label{hreg2-aa}
{\p {\psi_1}{x_j}}
&=\int_0^{t(x,\omega)}\big(-\int_0^t{\p \Sigma{x_j}}(x-s\omega,\omega,E)ds\big)
\cdot e^{-\int_0^t\Sigma(x-s\omega,\omega,E)ds}\cdot
f(x-t\omega,\omega,E)dt\nonumber\\
&
+\int_0^{t(x,\omega)} e^{-\int_0^t\Sigma(x-s\omega,\omega,E)ds} {\p {f}{x_j}}(x-t\omega,\omega,E)dt
\nonumber\\
&
+
e^{-\int_0^{t(x,\omega)}\Sigma(x-s\omega,\omega,E)ds}\ f(x-t(x,\omega)\omega,\omega,E)\ {\p t{x_j}}((x,\omega)\nonumber\\
&
=:h_1+h_2+h_3
\eea
and 
\bea\label{e-5-aa}
&
{\p {\psi_2}{x_j}}(x,\omega,E)= 
-\Sigma(x-t(x,\omega)\omega,\omega,E){\p t{x_j}}(x,\omega)e^{-\int_0^{t(x,\omega)}\Sigma(x-s\omega,\omega,E)ds}
\ g(x-t(x,\omega)\omega,\omega,E)
\\
&
-\int_0^{t(x,\omega)}{\p \Sigma{x_j}}(x-s\omega,\omega,E)ds\ 
e^{-\int_0^{t(x,\omega)}\Sigma(x-s\omega,\omega,E)ds} \ g(x-t(x,\omega)\omega,\omega,E)\nonumber\\
&
+
e^{-\int_0^{t(x,\omega)}\Sigma(x-s\omega,\omega,E)ds}\la \nabla_{\partial G}g(x-t(x,\omega)\omega,\omega,E),e_j-{\p t{x_j}}(x,\omega)\omega\ra
\nonumber\\
&
=:q_1+q_2+q_3.
\eea
Adding and subtracting reasonably the terms  it can be shown that under relevant assumptions for $j=1,2,3$ and for some $s>0$
\be\label{fr-c}
\int_G\int_{G'}\int_S\int_I{{|h_j(x,\omega,E)-h_j(x',\omega,E)|^2}\over{|x-x'|^{3+2s}}} dE d\omega dx' dx<\infty,
\ee 
\be\label{fr-d}
\int_G\int_{G'}\int_S\int_I{{|q_j(x,\omega,E)-q_j(x',\omega,E)|^2}\over{|x-x'|^{3+2s}}} dE d\omega dx' dx<\infty.
\ee 
We omit these generalizations here. However, we remark that the index $s$ is limited to $s<1/2$ such as the next example suggests. 
\end{remark}

The following example given in \cite[Example 7.4]{tervo17-up}   shows that in the case of transport problems, the regularity
of the solution
{\it does not generally} arise from the regularity of data and cross-sections in the sense that "the solution is more and more regular on the whole domain $G\times S\times I$ as the data and cross-sections are more and more regular".

\begin{example}\label{exreg}
 
Let $G=B(0,R)\subset \R^3$ and consider the problem 
\[
\omega\cdot\nabla\psi+\psi&=1, \\[2mm]
\psi_{|\Gamma_-}&=0.
\]
By (\ref{stif2}) the solution of the problem is
\[
\psi=1-e^{-t(x,\omega)},
\]
where for $(x,\omega)\in G\times S$,
\[
t(x,\omega)=x\cdot \omega +\sqrt{(x\cdot \omega)^2+R^2-|x|^2}.
\]
In the present example (related to equations like in section \ref{reg-up-to}),
$\Sigma=1$, $K_r=0$, $f=1$ and $g=0$ and so the date is smooth (even $C^\infty$).

We showed in \cite[Example 7.5]{tervo17-up} that $\psi\in H^{(1,0,0)}(G\times S\times I^\circ)$, but
$\psi\not\in H^{(2,0,0)}(G\times S\times I^\circ)$. In \cite[Example 7.6]{tervo17-up} we conjectured that actually $\psi\not\in H^{(3/2,0,0)}(G\times S\times I^\circ)$.
Hence only the limited regularity can be achieved up to the boundary.
We notice, however  that $\psi\in C^\infty(G\times S\times I^\circ)$ since $R^2-|x|^2>0$ for $x\in G$ and so the interior regularity of any order may be possible.

\end{example}

\section{On  the  regularity with respect to spatial variable for a complete equation (in a special case)}\label{w-x-reg}

Consider
the  transport problem of the form 
\bea
&
T\psi:=-{\p {(a\psi)}{E}}+c\Delta_S\psi+d\cdot\nabla_S\psi +\omega\cdot\nabla_x\psi+\Sigma\psi-K_r\psi=f\ {\rm in}\ G\times S\times I\label{wxr-1}\\
&
\psi_{|\Gamma_-}=g, \ \psi(.,.,E_m)=0\label{wxr-1a}
\eea
where $a=a(x,E),\ c=c(x,E),\ \Sigma=\Sigma(x,\omega,E),\ d=d(x,\omega,E)$ and
where $f\in L^2(G\times S\times I)$ and $g\in T^2(\Gamma_-)$. Note that any equation of the form
\[
a{\p {\psi}{E}}+c\Delta_S\psi+d\cdot\nabla_S\psi+\omega\cdot\nabla_x\psi+\Sigma\psi-K_r\psi=f
\]
can expressed in the form (\ref{wxr-1}) since $a{\p\psi{E}}=-{\p {(-a\psi)}{E}}-{\p a{E}}\psi$.
In the sequel we shall  deal with the  regularity with respect to $x$-variable only.   Moreover, for simplicity \emph{we assume in the sequel that  $c=d=0$}. 

\subsection{A Solution Based on Neumann Series}\label{meta}

The application of Neumann series in theoretical and numerical treatises of transport problems has long traditions (cf. e.g. \cite{case63}). 
Neumann series can  be used to prove existence, uniqueness and positivity of solutions for some transport equations  
but they give also a potential method to prove regularity of solutions. We shall below consider some techniques for regularity founded on  Neumann series.

Consider the transport problem (\ref{wxr-1}), (\ref{wxr-1a}) with $c=d=0$.
Setting $\phi=e^{CE}\psi$ as in \cite{tervo17}, section 3.1, we see that the problem takes an equivalent form 
\be\label{cosyst1}
-{\p {(a\phi)}E}+\omega\cdot\nabla_x\phi+Ca\phi+\Sigma\phi
-{K}_{r,C}\phi={ f_C},\
{\phi}_{|\Gamma_-}={ g_C},\
\phi(\cdot,\cdot,E_m)=0
\ee
where
\[
{ f_C}=e^{CE} f,\quad
{ g_C}=e^{CE} g
\]
and where
the restricted collision operator $K_{r,C}$  is given by
\be 
K_{r,C}\phi:=e^{CE}K_r(e^{-CE'}\phi).
\ee
 
Define
\be\label{cmp4} 
T_{C}\phi:=-{\p {(a\phi)}E}+\omega\cdot\nabla_x\phi+
Ca\phi+\Sigma\phi
- K_{r,C}\phi
\ee
The operator  $T_C:L^2(G\times S\times I)\to L^2(G\times S\times I)$ 
is  a (densily defined) closed linear operator when we set its domain as 
\be 
D(T_C):={}&\{\phi\in L^2(G\times S\times I)\ |\ T_{C}\phi\in L^2(G\times S\times I)\}
.
\ee
Using this notation
the problem (\ref{cosyst1}) can be expressed equivalently as
\be\label{comp5}
T_C\phi={f_C},\quad \phi_{|\Gamma_-}={ g_C},\quad \phi(\cdot,\cdot,E_m)=0.
\ee

Recall that the lift $L_-g$ of $g$ can be given by
\[
(L_-{ g})(x,\omega,E)={ g}(x-t(x,\omega)\omega,\omega,E).
\]
Recall also that (\cite[Lemma 6.11]{tervo17-up}, \cite[section 3.4.2]{tervo17})
\be\label{regL}
L_-{ g}\in H^1(I,\widetilde W^2(G\times S)) \ {\rm for}\ g\in H^1(I,T^2(\Gamma_-')),
\ee
and that it satisfies 
\be\label{ocdn}
\omega\cdot\nabla_x (L_-{ g})=0,\quad
(L_-{ g})_{|\Gamma_-}={g}.
\ee
Furthermore, the compatibility condition $g(\cdot,\cdot,E_{\rm m})=0$ implies that 
\be\label{comp7}
(L_-{g})(\cdot,\cdot,E_{\rm m})=0.
\ee

Denote the differential part of $T_C$ by
\be\label{comp8}
P_C\phi:=P_{C}(x,\omega,E,D)\phi:={}&-{\p {(a\phi)}E}+\omega\cdot\nabla_x\phi+
Ca\phi+\Sigma\phi.
\ee
Then
\[
T_C= P_C-K_{r,C}. 
\]
Let ${ P}_{C,0}$ be the densely defined linear operator
acting in $L^2(G\times S\times I)$ such that  
\bea
D({ P}_{C,0}):={}&\{u\in \widetilde W^2(G\times S\times I)\cap H^1(I,L^2(G\times S))\ |\nonumber\\
{}&\hspace{2mm}
u_{|\Gamma_-}=0,\ u(\cdot,\cdot,E_{\rm m})=0\},\nonumber\\
{ P}_{C,0}u:={}&P_C(x,\omega,E,D)u.
\eea
Furthermore, let $\widetilde{P}_{C,0}:L^2(G\times S\times I)\to L^2(G\times S\times I)$ be the smallest closed extension of ${ P}_{C,0}$. 
Writing $u:=\phi-L_-{ g_C}$,
we see (by (\ref{ocdn}),  (\ref{comp7})) that $\phi=u+L_-{ g_C}$ is a solution of (\ref{comp5}) if and only if
\be\label{comp10}
( P_{C}-K_{r,C})(u+L_-{g})={f_C},\quad u_{|\Gamma_-}=0,\quad u(\cdot,\cdot,E_m)=0.
\ee
When $u\in D(\widetilde{P}_{C,0})$ the problem (\ref{comp10}) is equivalent to
\be\label{comp11}
\widetilde{ P}_{C,0}u-K_{r,C}u={ f_C}-( P_C-K_{r,C})(L_-{ g_C})=:\ol f_C.
\ee

This kind of decomposition $\phi=u+L_-g_C$ of the solution is proved under certain assumptions in \cite[Theorem 3.11]{tervo17} (see also \cite{tervo18}). We give  somewhat more general formulation of it.

\begin{theorem}\label{coth3-dd}
Suppose that the assumptions 
(\ref{ass-for-s}),  (\ref{ass5-a}), (\ref{ass7}),  (\ref{ass-8}),
(imposed for $\Sigma$ and $K_r$ in section \ref{res-coll}) are valid and that
\be\label{KC-coer}
\la K_{r,C}\phi,\phi\ra_{L^2(G\times S\times I)}\geq c\n{\phi}_{L^2(G\times S\times I)}^2, \ {\rm for\ all}\ \phi\in L^2(G\times S\times I)
\ee
where $c>0$.
Furthermore, assume that
\be\label{evo16}
a\in C^2(I,L^\infty(G)),
\ee
\be\label{evo8-a}
a(x,E)\geq\kappa>0, \ {\rm a.e.}
\ee
\be\label{evo9-a}
\nabla_x a\in  L^\infty(G\times I)
\ee
and
\be\label{flp1-ab}
-{\p {a}E}+2Ca\geq 0 \ {\rm a.e.}\ .
\ee
Suppose that $f_C\in L^2(G\times S\times I)$, and $ g_C\in H^1(I,T^2(\Gamma_-'))$
is such that the \emph{compatibility condition} (on $\Gamma_-'$)
\be\label{comp-dd}
g_C(\cdot,\cdot,E_{\rm m})=0
\ee
holds.
Then the problem
\begin{gather}
-{\p {(a\phi)}E}+\omega\cdot\nabla_x\phi+Ca\phi+\Sigma\phi -K_{r,C}\phi=f_C,\ \nonumber\\
{\phi}_{|\Gamma_-}= g_C,\quad \phi(\cdot,\cdot,E_{\rm m})=0, \label{co3aa-dd}
\end{gather}
has a unique solution 
\be\label{solin}
\phi=u+L_-g_C
\ee
where $u\in  D(\widetilde P_{C,0})$.
\end{theorem}

By the proofs given in \cite{tervo17} we additionally see that under the assumptions of  Theorem \ref{coth3-dd}, the inverse $\widetilde P_{C,0}^{-1}$ exists. Moreover,  since $ g_C\in H^1(I,T^2(\Gamma_-'))$ and 
since $\omega\cdot\nabla_x(L_-g_C)=0$ we have
\[
&
\ol f_C={ f_C}-({ P_C}-K_{r,C})(L_-{ g_C})\\
&
={ f_C}+{\p {(aL_-g_C)}E}-
Ca(L_-g_C)-\Sigma(L_-g_C)+K_{r,C}(L_-{ g_C})
\in L^2(G\times S\times I).
\]
By the equation \eqref{comp11} we obtain
\be\label{comp12}
u=\widetilde{ P}_{C,0}^{-1}K_{r,C}u+\widetilde{ P}_{C,0}^{-1}\ol f_C,
\ee
which is equivalent to
\be\label{comp13}
(I-Q_C)u=\widetilde{ P}_{C,0}^{-1}\ol f_C,
\ee
where
\[
Q_C:=\widetilde{ P}_{C,0}^{-1}K_{r,C}.
\]
If $1$ belongs to the resolvent set $\rho(Q_C)$ of $Q_C$ we thus have
\be 
u=(I-Q_C)^{-1}\widetilde{ P}_{C,0}^{-1}\ol f_C,
\ee
and therefore
\be 
\phi=(I-Q_C)^{-1}\widetilde{ P}_{C,0}^{-1}\ol f_C+L_-{ g_C}.
\ee

Assuming that 
$Q_C:L^2(G\times S\times I)\to L^2(G\times S\times  I)$ is bounded and that
\be\label{comp14}
\n{Q_C}<1,
\ee
(which implies in particular that $1\in \rho(Q_C)$)
the solution $u$  can be computed through \emph{Neumann series}
\be\label{15}
u=
\sum_{k=0}^\infty Q_C^k(\widetilde{P}_{C,0}^{-1}\ol f_C)
=
\sum_{k=0}^\infty
(\widetilde{P}_{C,0}^{-1}K_{r,C})^k(\widetilde{P}_{C,0}^{-1}\ol f_C).
\ee
The solution $\phi$ of \eqref{comp5} is then
\be\label{comp16}
\phi=\sum_{k=0}^\infty
(\widetilde{ P}_{C,0}^{-1}K_{r,C})^k(\widetilde{ P}_{C,0}^{-1}\ol f_C)+L_-{ g_C},
\ee
from which the solution $\psi$ of the original problem \eqref{wxr-1} is obtained by
\be\label{comp16-a}
\psi=e^{-CE}\phi.
\ee

\begin{remark}
The formulas (\ref{comp16}), (\ref{comp16-a}) can also be  used to compute numerical solutions (by truncating the series to a finite sum). Note that in the case where $\widetilde{ P}_{C,0}^{-1}$ is explicitly known no inversion of large dimensional matrices is needed.
\end{remark}

To illustrate how the Neumann series can be used in the study of regularity we shall treat the next special case. Analogous results can be retrieve for more general equations, for example by using the more general solution formula (\ref{ts-1}). Moreover,  regularity with respect to all variables can be treated by applying similar methods.

\subsection{A special case}\label{csda-special}

We consider in more detail the special case given in section \ref{csda-ex}.
Recall that $E_0=0$ therein. Moreover, we for simplicity assume that $a> 0$ and $\Sigma>0$ are \emph{constants}. 
Then
\[
P_C(x,\omega,E,D)u=-a\p{u}{E}+\omega\cdot\nabla_x u+aCu+\Sigma u.
\]
In this case, $R(E)=\int_0^E \frac{1}{a}d\tau=\frac{1}{a}E$,
and so the  equation \eqref{ts-1} gives  that the solution of the problem 
\be\label{csdaprob}
P_C(x,\omega,E,D)u=h,\ u_{|\Gamma_-}=0,\ u(.,.,E_m)=0
\ee
is (see \cite[Example 10.4]{tervo18-up})
\be\label{sol-form-a}
u(x,\omega,E)=\int_0^{\min\{\eta(E),t(x,\omega)\}} e^{-(aC+\Sigma) s} h\big(x-s\omega,\omega,E+as\big)ds
\ee
where   $\eta(E):=(E_m-E)/a$
and where we noticed that $r_m:=R(E_m)=\frac{E_m}{a}$.

We verify that $u\in D(\widetilde P_{C,0})$.
Let $\{h_n\}\subset C_0^\infty(G\times S\times I^\circ)$ be a sequence such that $h_n\to h$ in $L^2(G\times S\times I)$. Then (cf. (\ref{u-norm}) below)
\[
u_n:=\int_0^{\min\{\eta(E),t(x,\omega)\}} e^{-(aC+\Sigma) s} h_n\big(x-s\omega,\omega,E+as\big)ds
\to u
\]
in $L^2(G\times S\times I)$ and by the construction
\[
P_C(x,\omega,E,D)u_n=h_n,\ {u_n}_{|\Gamma_-}=0,\ u_n(.,.,E_m)=0.
\]
Using e.g. the existence result \cite[Theorem 3.6]{tervo17} we see that $u_n\in D(P_{C,0})$. Since $u_n\to u$ and $P_{C,0}u_n=h_n\to h$ we find that $u\in D(\widetilde P_{C,0})$, as desired. In addition, since $\widetilde P_{C,0}u=h$, we have
\be\label{sol-form}
(\widetilde {P}_{C,0}^{-1}h)(x,\omega,E)=\int_0^{\min\{\eta(E),t(x,\omega)\}} e^{-(aC+\Sigma) s} h\big(x-s\omega,\omega,E+as\big)ds,
\ee
for $h\in L^2(G\times S\times I)$. Note that (in (\ref{sol-form})) $(\widetilde {P}_{C,0}^{-1}h)(y,\omega,E)=0$ a.e. $(y,\omega,E)\in\Gamma_-$ since $t(y,\omega)=0$.

Let 
\[
\alpha(x,\omega,E):=\min\{\eta(E),t(x,\omega)\}.
\]
Then in the weak sense (see e.g. \cite{ziemer})
\be\label{sc-2}
{\p \alpha{x_j}}=
\begin{cases} 
{\p t{x_j}},\ &t(x,\omega)< \eta(E)\\
0,\ &t(x,\omega)> \eta(E).
\end{cases}
\ee
Let $h\in H^{(1,0,0)}(G\times S\times I^\circ)$.
Denoting $u=\widetilde {P}_{C,0}^{-1}h$ we have $h\in H^{(1,0,0)}(G\times S\times I^\circ)$
\bea\label{sc-1}
&
{\p u{x_j}}(x,\omega,E)=\int_0^{\alpha(x,\omega,E)}
e^{-(aC+\Sigma) s} {\p h{x_j}}\big(x-s\omega,\omega,E+as\big)ds\nonumber\\
&
+
e^{-(aC+\Sigma)\alpha(x,\omega,E)} h\big(x-\alpha(x,\omega,E)\omega,\omega,E+a\alpha(x,\omega,E)\big){\p \alpha{x_j}}(x,\omega,E).
\eea
Hence in virtue of (\ref{sc-2}) for $t(x,\omega)<\eta(E)$ 
\bea\label{sc-3}
&
{\p u{x_j}}(x,\omega,E)=\int_0^{t(x,\omega)}
e^{-(aC+\Sigma) s} {\p h{x_j}}\big(x-s\omega,\omega,E+as\big)ds\nonumber\\
&
+
e^{-(aC+\Sigma)t(x,\omega)} h\big(x-t(x,\omega)\omega,\omega,E+at(x,\omega)\big){\p t{x_j}}(x,\omega)
\eea
and for $t(x,\omega)>\eta(E)$
\be\label{sc-4}
&
{\p u{x_j}}(x,\omega,E)=\int_0^{\eta(E)}
e^{-(aC+\Sigma) s} {\p h{x_j}}\big(x-s\omega,\omega,E+as\big)ds.
\ee

Let 
\[
u=\begin{cases} u_1:=&\int_0^{t(x,\omega)} e^{-(aC+\Sigma) s} h\big(x-s\omega,\omega,E+as\big)ds,\ t(x,\omega)\leq \eta(E)\\
u_2:=&\int_0^{\eta(E)} e^{-(aC+\Sigma) s} h\big(x-s\omega,\omega,E+as\big)ds,\
t(x,\omega)>\eta(E).\end{cases}
\]
Furthermore, let
\[
&
U_1:=\{(x,\omega,E)\in G\times S\times I|\ t(x,\omega)<\eta(E)\},\\
&
U_2:=\{(x,\omega,E)\in G\times S\times I|\ t(x,\omega)>\eta(E)\},\\
&
U_0:=\{(x,\omega,E)\in G\times S\times I|\ t(x,\omega)=\eta(E)\}
\]
Note that $U_0$ has a measure zero in $G\times S\times I$ since $U_0=\{(x,\omega,E)\in G\times S\times I|\ E=(\eta^{-1}\circ t)(x,\omega)\}$.
Note also that $u_1=u_2$ on $U_0$. We show that under appropriate assumptions 
$u_1\in H^1(U_1)$ and $u_2\in H^1(U_2)$ from which we know by using the Green's formula that $u\in H^1(G\times S\times I^\circ)$ and
\[
{\p u{x_j}}=\begin{cases} {\p {u_1}{x_j}}\ &{\rm in}\ U_1\\
{\p {u_2}{x_j}}\ &{\rm in}\ U_2\end{cases}
\]
(we omit some details here).

We need the next lemmas.

\begin{lemma}\label{sc-le0}
Suppose that $a>0$  and $\Sigma>0$ are constants and let $C\geq 0$. Then for $h\in L^2(G\times S\times I)$
\be\label{u-norm}
\n{\widetilde P_{C,0}^{-1}h}_{L^2(G\times S\times I)}^2
\leq 
{{ad+m(I)}\over{a(aC+\Sigma)}}\n{h}_{L^2(G\times S\times I)}^2.
\ee
\end{lemma}

\begin{proof}
Let $u=\widetilde P_{C,0}^{-1}h$.
We use the following decomposition of the integral
\[
\n{u}_{L^2(G\times S\times I)}^2
=
\int_{U_1}|{u}(x,\omega,E)|^2 dx d\omega dE
+
\int_{U_2}|{u}(x,\omega,E)|^2 dx d\omega dE=:I_1+I_2.
\]

A. At first we consider the term $I_1$.  
By the Cauchy-Schwartz's inequality we have
\bea\label{sc-7}
&
I_{1}
=
\int_{U_1}\big|\int_0^{t(x,\omega)}
e^{-(aC+\Sigma) s} {h}\big(x-s\omega,\omega,E+as\big)ds\big|^2 dx d\omega dE
\nonumber\\
&
\leq 
\int_{U_1}\big(\int_0^{t(x,\omega)}
e^{-2(aC+\Sigma) s}ds\big)
\cdot
\big(\int_0^{t(x,\omega)} \Big|{h}\big(x-s\omega,\omega,E+as\big)\Big|^2ds\big)
dx d\omega dE
\nonumber\\
&
=
{1\over{2(aC+\Sigma)}}
\int_{U_1}
\big(1-e^{-2(aC+\Sigma)t(x,\omega)}\big)
\cdot
\big(\int_0^{t(x,\omega)} \big|{h}\big(x-s\omega,\omega,E+as\big)\big|^2ds\big)
dx d\omega dE\nonumber\\
&
\leq 
{1\over{aC+\Sigma}}
\int_{U_1}
\big(\int_0^{t(x,\omega)} \big|{h}\big(x-s\omega,\omega,E+as\big)\big|^2ds\big)
dx d\omega dE.
\eea
The condition $\eta(E)> t(x,\omega)$ on $U_1$ means that $E<E_m-at(x,\omega)$ and so
(as in (\ref{e-1}))
\bea\label{sc-8}
&
\int_{U_1}
\big(\int_0^{t(x,\omega)} \big|{h}\big(x-s\omega,\omega,E+as\big)\big|^2ds\big)
dx d\omega dE
\nonumber\\
&
\leq 
\int_G\int_S\int_0^{E_m-at(x,\omega)}
\int_0^{t(x,\omega)} \big|{h}\big(x-s\omega,\omega,E+as\big)\big|^2ds
dE d\omega dx \nonumber\\
&
=
\int_G\int_S\int_0^{t(x,\omega)}\int_0^{E_m-at(x,\omega)}
\big|{h}\big(x-s\omega,\omega,E+as\big)\big|^2dE
ds  d\omega dx\nonumber\\
&
=
\int_G\int_S\int_0^{t(x,\omega)}\int_{as}^{E_m+a(s-t(x,\omega))}
 \big|{h}\big(x-s\omega,\omega,E'\big)\big|^2dE' ds
 d\omega dx\nonumber\\
&
\leq 
\int_{\R^3}\int_S\int_0^{d}\int_{0}^{E_m}
 \big|\ol {h}\big(x-s\omega,\omega,E'\big)\big|^2dE' ds
d\omega dx\nonumber\\
&
= 
\int_0^d\int_S\int_{0}^{E_m}\int_{\R^3}
 \big|\ol {h}\big(z,\omega,E'\big)\big|^2
dz dE' d\omega  ds=d\n{{h}}_{L^2(G\times S\times I)}^2
\eea
where we noticed that 
\[
0\leq s\leq t(x,\omega)\leq d={\rm diam}(G)
\]
and where $\ol {h}$ is the extension by zero of $ {h}$ on $\R^3\times S\times I$. In addition, we performed the changes of variables $E'=E+as$ (with respect to $E$) and $z=x-s\omega$ (with respect to $x$).
That is why, by (\ref{sc-7}), (\ref{sc-8})
\be\label{sc-9}
I_{1}
\leq 
{{d}\over{aC+\Sigma}}\n{h}_{L^2(G\times S\times I)}^2.
\ee

B. 
Consider the term $I_2$.
Now the condition
$\eta(E)< t(x,\omega)$ on $U_2$ means that $E>E_m-at(x,\omega)$.
Hence by Cauchy-Schwartz's inequality (in the $3^{th}$ step)
\bea\label{sc-11}
&
I_2=
\int_{U_2}\big|\int_0^{\eta(E)}
e^{-(aC+\Sigma) s} {h}\big(x-s\omega,\omega,E+as\big)ds\big|^2 dx d\omega dE
\nonumber\\
&
=
\int_{G}\int_S\int_{0}^{E_m}H(t(x,\omega)-\eta(E))\big|\int_0^{\eta(E)}
e^{-(aC+\Sigma) s} {h}\big(x-s\omega,\omega,E+as\big)ds\big|^2 dE  d\omega 
dx
\nonumber\\
&
\leq
\int_{G}\int_S\int_{0}^{E_m}H(t(x,\omega)-\eta(E))\big(\int_0^{\eta(E)}
e^{-(aC+\Sigma) s}ds\big)\Big(\int_0^{\eta(E)}
e^{-(aC+\Sigma) s} \big|{h}\big(x-s\omega,\omega,E+as\big)\big|^2ds\Big) dE  d\omega dx
\nonumber\\
&
=
{1\over{aC+\Sigma}}\int_{G}\int_S\int_{0}^{E_m}H(t(x,\omega)-\eta(E))\big(1-e^{-(aC+\Sigma)\eta(E)}\big)\nonumber\\
&
\cdot
\int_0^{\eta(E)}
e^{-(aC+\Sigma) s} \big|{h}\big(x-s\omega,\omega,E+as\big)\big|^2 ds dE  d\omega dx
\nonumber\\
&
\leq 
{1\over{aC+\Sigma}}\int_{G}\int_S\int_{0}^{E_m}H(t(x,\omega)-\eta(E))\int_0^{\eta(E)}
 \big|{h}\big(x-s\omega,\omega,E+as\big)\big|^2 ds dE d\omega dx
\nonumber\\
&
=
{1\over{aC+\Sigma}}{1\over a}\int_G\int_S\int_{0}^{E_m}H(t(x,\omega)-\eta(E)\int_E^{E_m}
 \big |{h}\big(x-{{E'-E}\over a}\omega,\omega,E'\big)\big|^2 dE' dE d\omega dx
\nonumber\\
&
\leq
{1\over{aC+\Sigma}}{1\over a}\int_G\int_S\int_{I}\int_{I'}
 \big |\ol {h}\big(x-{{E'-E}\over a}\omega,\omega,E'\big)\big|^2 dE' dE  d\omega dx
\nonumber\\
&
\leq
{1\over{aC+\Sigma}}{1\over a}\int_{\R^3}\int_S\int_{I}\int_{I'}
 \big |\ol {h}\big(x-{{E'-E}\over a}\omega,\omega,E'\big)\big|^2 dE' dE  d\omega  dx
\nonumber\\
&
=
{1\over{a(aC+\Sigma)}}\int_S\int_I\int_{I'}\int_{\R^3}
\big |\ol{h}\big(x-{{E'-E}\over a}\omega,\omega,E'\big)\big|^2  dx dE' dE d\omega 
\nonumber\\
&
=
{1\over{a(aC+\Sigma)}}\int_S\int_I\int_{I'}\int_{\R^3}
 \big|\ol{h}\big(z,\omega,E'\big)\big|^2  dz dE' dE d\omega 
=
{{m(I)}\over{a(aC+\Sigma)}}\n{{h}}_{L^2(G\times S\times I)}^2
\eea 
where we applied the changes of variables $E+as=E'$ (with respect to $s$) and $z=x- {{E'-E}\over a}$ (with respect to $x$) and where $\ol{h}$ is the extension by zero of ${h}$ on $\R^3\times S\times I$.

Combining (\ref{sc-9}) and (\ref{sc-11}) the assertion follows.
 
\end{proof}

\begin{lemma}\label{sc-le1}
Suppose that $a>0$  and $\Sigma>0$ are constants and let $C\geq 0$. Furthermore, 
let $h\in H^{(1,0,0)}(G\times S\times I^\circ)$ such that
$h_{|\Gamma_-}\in T^2_{m_{1,j}}(\Gamma_-)$. Then for $j=1,2,3$
\bea\label{sc-5-a}
&
\n{{\p {(\widetilde P_{C,0}^{-1}h)}{x_j}}}_{L^2(G\times S\times I)}^2\leq  
{{ad+m(I)}\over{a(aC+\Sigma)}}
\n{{\p h{x_j}}}_{L^2(G\times S\times I)}^2\nonumber\\
&
+
2\int_{\Gamma_-}\int_0^{\tau_-(y,\omega)}
\big|e^{-(aC+\Sigma)s} h\big(y,\omega,E\big){\p t{x_j}}(y+s\omega,\omega)\big|^2
|\omega\cdot\nu(y)|ds d\sigma(y) d\omega dE.
\eea
\end{lemma}

\begin{proof}
Let as above $u=\widetilde P_{C,0}^{-1}h$.
We again use the following decomposition of the integral
\[
\n{{\p u{x_j}}}_{L^2(G\times S\times I)}^2
=
\int_{U_1}|{\p u{x_j}}(x,\omega,E)|^2 dx d\omega dE
+
\int_{U_2}|{\p u{x_j}}(x,\omega,E)|^2 dx d\omega dE=:I_1+I_2.
\]

At first we consider the term $I_1$. By (\ref{sc-3}) we have
\bea\label{sc-6}
&
I_1=
\int_{U_1}\Big|\int_0^{t(x,\omega)}
e^{-(aC+\Sigma) s} {\p h{x_j}}\big(x-s\omega,\omega,E+as\big)ds\nonumber\\
&
+
e^{-(aC+\Sigma)t(x,\omega)} h\big(x-t(x,\omega)\omega,\omega,E+at(x,\omega)\big){\p t{x_j}}(x,\omega)\Big|^2 dx d\omega dE\nonumber\\
&
\leq
2\int_{U_1}\big|\int_0^{t(x,\omega)}
e^{-(aC+\Sigma) s} {\p h{x_j}}\big(x-s\omega,\omega,E+as\big)ds\big|^2
dx d\omega dE\nonumber\\
&
+2\int_{U_1}\big
|e^{-(aC+\Sigma)t(x,\omega)} h\big(x-t(x,\omega)\omega,\omega,E+at(x,\omega)\big){\p t{x_j}}(x,\omega)\big|^2 dx d\omega dE\nonumber\\
&
=:I_{1,1}+I_{1,2}.
\eea
In virtue of Part A of the proof of the previous lemma \ref{sc-le0}
\be\label{sc-9-a}
I_{1,1}
\leq 
{{d}\over{aC+\Sigma}}\n{{\p h{x_j}}}_{L^2(G\times S\times I)}^2.
\ee

Consider the integral $I_{1,2}$. 
Utilizing Lemma \ref{pr:fubini} and observing that $t(y+s\omega,\omega)=s$ for $(y,\omega)\in \Gamma_-'$, we have (below $H$ is as above the Heaviside function)
\bea\label{sc-10}
&
I_{1,2}
=
2\int_{U_1}
\big|e^{-(aC+\Sigma)t(x,\omega)} h\big(x-t(x,\omega)\omega,\omega,E+at(x,\omega)\big){\p t{x_j}}(x,\omega)\big|^2 dE dx d\omega\nonumber\\
&
=
2\int_{G\times S\times I}H(\eta(E)-t(x,\omega))
\big|e^{-(aC+\Sigma)t(x,\omega)} h\big(x-t(x,\omega)\omega,\omega,E+at(x,\omega)\big){\p t{x_j}}(x,\omega)\big|^2 dE dx d\omega\nonumber\\
&
=
2\int_{\Gamma_-'}\int_0^{\tau_-(y,\omega)}\int_0^{E_m} H(\eta(E)-s)
\big|e^{-(aC+\Sigma)s} h\big(y,\omega,E+as\big){\p t{x_j}}(y+s\omega,\omega)\big|^2 dE ds d\sigma(y) d\omega \nonumber\\
&
=2
\int_{\Gamma_-'}\int_0^{\tau_-(y,\omega)}\int_0^{E_m-as}
\big|e^{-(aC+\Sigma)s} h\big(y,\omega,E+as\big){\p t{x_j}}(y+s\omega,\omega)\big|^2
|\omega\cdot\nu(y)| dE ds d\sigma(y) d\omega  \nonumber\\
&
=
2\int_{\Gamma_-'}\int_0^{\tau_-(y,\omega)}\int_{as}^{E_m}
\big|e^{-(aC+\Sigma)s} h\big(y,\omega,E'\big){\p t{x_j}}(y+s\omega,\omega)\big|^2
|\omega\cdot\nu(y)| dE' ds d\sigma(y) d\omega  \nonumber\\
&
\leq 
2\int_{\Gamma_-'}\int_0^{\tau_-(y,\omega)}\int_{I'}
\big|e^{-(aC+\Sigma)s} h\big(y,\omega,E'\big){\p t{x_j}}(y+s\omega,\omega)\big|^2
|\omega\cdot\nu(y)| dE' ds d\sigma(y) d\omega  \nonumber\\
&
=
2
\int_{\Gamma_-'}\int_{I'}\int_0^{\tau_-(y,\omega)}
\big|e^{-(aC+\Sigma)s} h\big(y,\omega,E'\big){\p t{x_j}}(y+s\omega,\omega)\big|^2
|\omega\cdot\nu(y)|ds dE' d\sigma(y)  d\omega  \nonumber\\
&
=
2\int_{\Gamma_-}\int_0^{\tau_-(y,\omega)}
\big|e^{-(aC+\Sigma)s} h\big(y,\omega,E\big){\p t{x_j}}(y+s\omega,\omega)\big|^2
|\omega\cdot\nu(y)|ds d\sigma(y) d\omega dE
\eea
where in the third step we performed the change of variables $E'=E+as$ (with respect to $E$ variable).

Due to Part B of the proof of the previous lemma  \ref{sc-le0}
\be\label{sc-11-a}
I_2\leq
{{m(I)}\over{a(aC+\Sigma)}}\n{{\p h{x_j}}}_{L^2(G\times S\times I)}^2
\ee 
Combining (\ref{sc-9}), (\ref{sc-10}), (\ref{sc-11}) the assertion follows.

\end{proof}

Combining Lemmas \ref{sc-le0} and \ref{sc-le1} we conclude

\begin{corollary}\label{pinv-norm}
Suppose that $a>0$  and $\Sigma>0$ are constants and let $C\geq 0$. Furthermore, 
let $h\in H^{(1,0,0)}(G\times S\times I^\circ)$ such that
$h_{|\Gamma_-}\in T^2_{m_{1,j}}(\Gamma_-),\ j=1,2,3$. Then 
\bea\label{sc-5-a}
&
\n{{\widetilde P_{C,0}^{-1}h}}_{H^{(1,0,0)}(G\times S\times I^\circ)}^2\leq  
{{ad+m(I)}\over{a(aC+\Sigma)}}
\n{h}_{H^{(1,0,0)}(G\times S\times I^\circ)}^2\nonumber\\
&
+
2\sum_{j=1}^3\int_{\Gamma_-}\int_0^{\tau_-(y,\omega)}
\big|e^{-(aC+\Sigma)s} h\big(y,\omega,E\big){\p t{x_j}}(y+s\omega,\omega)\big|^2
|\omega\cdot\nu(y)|ds d\sigma(y) d\omega dE.
\eea
\end{corollary}

To simplify the technicalities we assume that the restricted collision operator is of the form
\be\label{sc-13}
(K_r\psi)(x,\omega,E)=\int_{S'}\sigma^2(x,\omega,\omega',E)\psi(x,\omega',E) d\omega'.
\ee
This kind of restricted collision operator is related to elastic scattering of particles (\cite{tervo18-up}, section 5.2). For this operator 
\[
K_{r,C}\phi=K^2_r\phi.
\]
One example of elastic scattering processes is the \emph{screened Rutherford scattering} 
of electrons whose cross section is 
\be \label{e-el}
\sigma^2(x,\omega',\omega,E)=
\sigma_0(x){{(E+1)^2}\over{E^2(E+2)^2}}{1\over{(1-\omega'\cdot\omega+q(x,E))^2}}
\ee
where $\sigma_0\in W^{\infty,1}(G)$ and $q\in W^{\infty,1}(G\times I^\circ)$ for which $q \geq q_0>0$.

We assume 
besides of (\ref{ass7}) that
$\sigma^2:G\times S^2\times I\to\R$ satisfies for $j=1,2,3$
\bea\label{ass7-sc}
&\int_{S'}|(\partial_{x_j}\sigma^2)(x,\omega',\omega,E)|d\omega'\leq M_1',\nonumber\\
&\int_{S'}|(\partial_{x_j}\sigma^2)(x,\omega,\omega',E)| d\omega'\leq M_2'
\eea
for a.e. $(x,\omega,E)\in G\times S\times I$.

\begin{lemma}\label{sc-le3}
Suppose that the assumptions (\ref{ass7}), (\ref{ass7-sc})  are valid. Then 
$K^2_{r}$ is a bounded operator $H^{(1,0,0)}(G\times S\times I^\circ) 
\to H^{(1,0,0)}(G\times S\times I^\circ)$ 
and its norm (between the spaces in question) obeys 
\be\label{sc-12-a}
\n{K^2_{r}}_1\leq \sqrt{2}\big(\sqrt{3M_1'M_2'+M_1M_2}\big).
\ee
\end{lemma}

\begin{proof}

We have for $\phi\in H^{(1,0,0)}(G\times S\times I^\circ)$
\[
&
(\partial_{x_j}(K^2_r\phi))(x,\omega,E)=
\int_{S'}(\partial_{x_j}\sigma^2)(x,\omega,\omega',E)\phi(x,\omega',E) d\omega'
\\
&
+
\int_{S'}\sigma^2(x,\omega,\omega',E)(\partial_{x_j}\phi)(x,\omega',E) d\omega'
\]
and so by Theorem \ref{esol-th1} and by (\ref{ass7-sc})
\be\label{sc-14}
\n{\partial_{x_j}(K^2_r\phi)}_{L^2(G\times S\times I)}
\leq 
\sqrt{M_1'M_2'}\n{\phi}_{L^2(G\times S\times I)}+\sqrt{M_1M_2}\n{\partial_{x_j}\phi}_{L^2(G\times S\times I)}
\ee
which implies that
\be\label{sc-15}
\n{K^2_r\phi}_{H^{(1,0,0)}(G\times S\times I^\circ)}
\leq 
\sqrt{2}\sqrt{3M_1'M_2'+M_1M_2}\n{\phi}_{H^{(1,0,0)}(G\times S\times I^\circ)}
\ee
since by (\ref{ass7}) $\n{K^2_r\phi}_{L^2(G\times S\times }
\leq \sqrt{M_1M_2})\n{\phi}_{L^2(G\times S\times I)}$.
This completes the proof.
\end{proof}

Note that $\phi_{|\Gamma_-}=0$ does not necessarily imply that $K^2_r\phi_{|\Gamma_-}=0$.
That is why, we additionally assume a technical criterion
\be\label{sc-16-a}
\sigma^2(y,\omega,.,E)=0\ {\rm a.e.}\ (y,\omega,E)\in \Gamma_-.
\ee
We notice that the trace $\sigma^2(y,\omega,.,E)$ is due to the Sobolev trace theorem well-defined since by (\ref{ass7}), (\ref{ass7-sc}) $\sigma^2\in W^{\infty.(1,0,0)}(G\times S\times I,L^1(S'))$.
This condition is not necessary for the below results (see Remark \ref{nn} below) but it simplifies treatises.
For the Rutherford cross section the condition is valid if, for example $\sigma_0\in H_0^1(G)$. The condition (\ref{sc-16-a}) guarantees that in the Sobolev sense $(K_r^2\psi)_{|\Gamma_-}=0$ for $\psi\in H^{(1,0,0)}(G\times S\times I^\circ)$ (we omit details).

Note that 
\be\label{imb} 
H^{(1,0,0)}(G\times S\times I^\circ)\subset W^2(G\times S\times I)
\ee
and the imbedding is continuous. Hence for any $u\in H^{(1,0,0)}(G\times S\times I^\circ)$ the inflow trace $\gamma_-(u)=u_{|\Gamma_-}$ is well-defined and $u_{|\Gamma_-}\in L^2_{\rm loc}(\Gamma_-,|\omega\cdot\nu| d\sigma d\omega dE)$ (recall section \ref{fs}).
Let
\[
H_{-,0}^{(1,0,0)}(G\times S\times I^\circ):=
\{u\in H^{(1,0,0)}(G\times S\times I^\circ)|\ u_{|\Gamma_-}=0\}.
\]
We equip $H_{-,0}^{(1,0,0)}(G\times S\times I^\circ)$ with the norm induced by
$H^{(1,0,0)}(G\times S\times I^\circ)$. Then the subspace 
$H_{-,0}^{(1,0,0)}(G\times S\times I^\circ)$ is  closed in $H^{(1,0,0)}(G\times S\times I^\circ)$ which can be verified as follows. Let $\{u_n\}$ be a sequence in 
$H_{-,0}^{(1,0,0)}(G\times S\times I^\circ)$ such that $u_n\to u$ in 
$H^{(1,0,0)}(G\times S\times I^\circ)$. Then by (\ref{imb}) 
$u_n\to u$ in 
$W^2(G\times S\times I)$ which implies that 
${u_n}_{|\Gamma_-}\to u_{|\Gamma_-}$ in $L^2_{\rm loc}(\Gamma_-,|\omega\cdot\nu| d\sigma d\omega dE)$ because $\gamma_-:W^2(G\times S\times I)\to L^2_{\rm loc}(\Gamma_-,|\omega\cdot\nu| d\sigma d\omega dE)$ is continuous. Since
${u_n}_{|\Gamma_-}=0,\ n\in \N$ we see that $u_{|\Gamma_-}=0$ and so $u\in 
H_{-,0}^{(1,0,0)}(G\times S\times I^\circ)$, as desired. Note that 
$\widetilde P_{C,0}^{-1}$ is a bounded operator from $H_{-,0}^{(1,0,0)}(G\times S\times I^\circ)$ into itself since for $h\in H_{-,0}^{(1,0,0)}(G\times S\times I^\circ)$ the boundary term of  (\ref{sc-5-a}) vanishes.

\begin{remark}\label{nn}
The below arguments are valid (instead of assuming of (\ref{sc-16-a})) also in the case where the norm $\n{K^2_r}_1$ of the operator $K^2_r:H_{-,0}^{(1,0,0)}(G\times S\times I^\circ)\to H^{(1,0,0)}(G\times S\times I^\circ)$ is small enough and that
\be\label{nn-1}
 L^2(\Gamma_-)\subset T^2_{m_{1,j}}(\Gamma_-).
\ee
Recall that (\ref{nn-1}) holds, for example for the ball $G=B(0,R)$ (Example \ref{rbpex1-ss}). We omit details but we mention that by the Sobolev trace theorem the trace
\[
H^{(1,0,0)}(G\times S\times I^\circ)\to L^2(\Gamma)
\]
is continuous
and so by using (\ref{nn-1}) the boundary term in (\ref{sc-5-a}) can be estimated above by a term
$C\n{h}_{H^{(1,0,0)}(G\times S\times I^\circ)}$ for some constant and all 
$h\in H_{-,0}^{(1,0,0)}(G\times S\times I^\circ)$.
\end{remark}

\begin{theorem}\label{sc-le2}
Suppose that the assumptions (\ref{ass7}), (\ref{ass7-sc}), (\ref{sc-16-a})  are valid and that $a>0$ and $\Sigma>0$ are constants. 
Furthermore, assume that 
\be\label{ashr1-a-sc}
g,\ {\p g{E}}\in T^2(\Gamma_-)\cap T^2_{m_{1,j}}(\Gamma_-), \ j=1,2,3,
\ee
\be\label{ashr1-b-sc} 
\nabla_{(\partial G)}(g),\ \nabla_{(\partial G)}\Big({\p g{E}}\Big)\in T^2(\Gamma_-)^2\cap T^2_{m_{1,j}}(\Gamma_-)^2,\ j=1,2,3
\ee
\be\label{ashr2-a-sc}
 f\in H^{(1,0,0)}(G\times S\times I^\circ), 
\ee
\be\label{ashr2-b-sc}
{ f}_{|\Gamma-}\in T^2(\Gamma_-)\cap T^2_{m_{1,j}}(\Gamma_-).
\ee
Finally, suppose that 
the compatibility condition (\ref{comp-dd}) holds.
Then any solution of the problem
\bea\label{wxr-1-a}
&
-a{\p \psi{E}}+\omega\cdot\nabla_x\psi+\Sigma\psi-K^2_r\psi=f\ {\rm in}\ G\times S\times I\nonumber\\
&
\psi_{|\Gamma_-}=g, \ \psi(.,.,E_m)=0
\eea
belongs to $H^{(1,0,0)}(G\times S\times I^\circ)$.
\end{theorem}

\begin{proof}
It suffices to show that (for $C$ large enough) the solution $\phi=u+L_-g_C$ of the problem (\ref{co3aa-dd}) belongs to $H^{(1,0,0)}(G\times S\times I^\circ)$
where $u$ is as in (\ref{15}).
Note that the conditions (\ref{ashr1-a-sc}), (\ref{ashr1-b-sc}), (\ref{comp-dd}) hold for $g_C$ if and only if they hold for $g$ and similarly (\ref{ashr2-a-sc}), 
(\ref{ashr2-b-sc}) hold for $f_C$ if and only if they hold for $f$.

A. 
For the first instance  we show that $L_-g_C\in H^{(1,0,0)}(G\times S\times I^\circ)$. 
We have for $(x,\omega,E)\in G\times S\times I$
\be\label{sc-19} 
{\p {(L_-g_C)}{x_j}}(x,\omega.E)=
\la (\nabla_{\partial G}g_C)(x-t(x,\omega)\omega,\omega,E),e_j-{\p t{x_j}}(x,\omega)\omega\ra
\ee
and so by Lemma \ref{pr:fubini} (and recalling that $t(y+s\omega)=s$)
\bea\label{sc-20}
&
\n{{\p {(L_-g_C)}{x_j}}}_{L^2(G\times S\times I)}^2
=
\int_S\int_I\int_G\big|\la (\nabla_{(\partial G)}g_C)(x-t(x,\omega)\omega,\omega,E),e_j-{\p t{x_j}}(x,\omega)\omega\ra\big|^2 dx d\omega dE\nonumber\\
&
=
\int_{\Gamma_-}\int_0^{\tau_-(y,\omega)}\big|\la (\nabla_{(\partial G)}g_C)(y,\omega,E),e_j-{\p t{x_j}}(y+s\omega,\omega)\omega\ra\big|^2|\omega\cdot\nu(y)| ds d\sigma(y) d\omega dE.
\eea
Since
\[
&
\big|\la (\nabla_{(\partial G)}g_C)(y,\omega,E),e_j-{\p t{x_j}}(y+s\omega,\omega)\omega\big)\ra\big|
\\
&
\leq
\n{(\nabla_{(\partial G)}g_C)(y,\omega,E)}
\n{e_j-{\p t{x_j}}(y+s\omega,\omega)\omega}
\leq 
\n{(\nabla_{(\partial G)}g_C)(y,\omega,E)}
(1+|{\p t{x_j}}(y+s\omega,\omega)|)
\]
we see by (\ref{sc-20}) and noting that  $\tau_-(y,\omega)\leq d$
\bea\label{sc-21-a}
&
\n{{\p {(L_-g_C)}{x_j}}}_{L^2(G\times S\times I)}^2\nonumber\\
&
\leq 
\int_{\Gamma_-}\int_0^{\tau_-(y,\omega)}\big(\n{(\nabla_{(\partial G)}g_C)(y,\omega,E)}
(1+|{\p t{x_j}}(y+s\omega,\omega)|)\big)^2|\omega\cdot\nu(y)|ds d\sigma(y) d\omega dE
\nonumber\\
&
\leq
2\big(d\n{\nabla_{(\partial G)}g_C}_{T^2(\Gamma_-)^2}^2+
\n{\nabla_{(\partial G)}g_C}_{T_{m_{1,j}}^2(\Gamma_-)^2}^2\big).
\eea
Hence in virtue of (\ref{ashr1-b-sc}) 
\be\label{l-g}
L_-g_C\in H^{(1,0,0)}(G\times S\times I^\circ).
\ee

B.1. 
By  Lemma \ref{sc-le3}  $K^2_r\phi\in H^{(1,0,0)}(G\times S\times I^\circ)$ for
$\phi\in H^{(1,0,0)}(G\times S\times I^\circ)$ and
by the assumption   (\ref{sc-16-a}) $K^2_r\phi_{|\Gamma_-}=0$ for 
$\phi\in H^{(1,0,0)}(G\times S\times I^\circ)$. 
Hence $K^2_r$ is  a bounded operator $H_{-,0}^{(1,0,0)}(G\times S\times I^\circ) \to H_{-,0}^{(1,0,0)}(G\times S\times I^\circ)$.
This implies by Corollary \ref{pinv-norm} that for any $C\geq 0$ the operator
$\widetilde P_{C,0}^{-1}K^2_{r}$ is a bounded operator $H_{-,0}^{(1,0,0)}(G\times S\times I^\circ) \to H_{-,0}^{(1,0,0)}(G\times S\times I^\circ)$ and its norm,
denoted by $\n{\widetilde P_{C,0}^{-1}K^2_{r}}_1$, obeys
\be\label{sc-12-b}
\n{\widetilde P_{C,0}^{-1}K^2_{r}}_1\leq  
\sqrt{{{ad+m(I)}\over{a(aC+\Sigma)}}}\n{K^2_{r}}_1
\ee
where we observed that for $h=K^2_{r}\phi\in H_{-,0}^{(1,0,0)}(G\times S\times I^\circ)$ the boundary term vanishes in (\ref{sc-5-a}).
Choosing $C$ large enough we find that
\be\label{sc-18}
\n{\widetilde P_{C,0}^{-1}K^2_{r}}_1<1.
\ee
By now on we fix $C$ such that (\ref{sc-18}) holds.

B.2. We show that
\be\label{uu}
\widetilde{ P}_{C,0}^{-1}\ol f_C\in H_{-,0}^{(1,0,0)}(G\times S\times I^\circ).
\ee
Recall that (since $\omega\cdot\nabla_x(L_-g_C)=0$ and $K_{r,C}\phi=K_r^2\phi$)
\[
\ol f_C=f_C-(P_C-K^2_{r})(L_-g_C)=f_C+a{\p {(L_-g_C)}{E}}-(Ca+\Sigma)(L_-g_C)+K^2_{r}(L_-g_C)
\]
and so 
\[
\partial_{x_j}( {\ol f_C})=\partial_{x_j}( f_C)+a\partial_{x_j}\big({\p {(L_-g_C)}{E}}\big)-(Ca+\Sigma)\partial_{x_j}(L_-g_C)
+(\partial_{x_j}K^2_{r})(L_-g_C)+K^2_{r}(\partial_{x_j}(L_-g_C))
\]
where we used the proof of Lemma \ref{sc-le3} for $\partial_{x_j}(K^2_{r}(L_-g_C))$ and where as above
\[
((\partial_{x_j}K^2_r)\phi)(x,\omega,E)=
\int_{S'}(\partial_{x_j}\sigma^2)(x,\omega,\omega',E)\phi(x,\omega',E) d\omega'.
\]

Since by (\ref{ashr1-b-sc})
\[
\nabla_{(\partial G)}\Big({\p {g_C}{E}}\Big)\in T^2(\Gamma_-)^2\cap T^2_{m_{1,j}}(\Gamma_-)^2,\ j=1,2,3.
\]
and since (note that and $\partial_E(L_-g_C)=L_-(\partial_Eg_C)$)
\[
\partial_{x_j}\big({\p {(L_-g_C)}{E}}\big)
=
\la\nabla_{(\partial G)}\Big({\p {g_C}{E}}\Big)(x-t(x,\omega)\omega,\omega,E),e_j-{\p t{x_j}}(x,\omega)\omega\ra,
\]
we see as in Part A   that $\partial_{x_j}\big({\p {(L_-g_C)}{E}}\big)\in L^2(G\times S\times I)$. 
In virtue of  Part A $\partial_{x_j}\big(L_-g_C\big)\in L^2(G\times S\times I)$ and so by  Lemma \ref{sc-le3}   $(\partial_{x_j}K^2_{r})(L_-g_C)+K^2_{r}(\partial_{x_j}(L_-g_C))\in L^2(G\times S\times I)$. Hence $\partial_{x_j}(\ol f_C)\in L^2(G\times S\times I),\ j=1,2,3$ and then (since $f_C\in H^{(1,0,0)}(G\times S\times I^\circ)$)
\[
\ol f_C\in H^{(1,0,0)}(G\times S\times I^\circ).
\]
In addition, we see that
\[
{\ol f_C}_{\big|\Gamma_-}={f_C}_{\big|\Gamma_-}+a{\p {(L_-g_C)}{E}}_{\Big|\Gamma_-}-(Ca+\Sigma)(L_-g_C)_{|\Gamma_-}+K^2_{r}(L_-g_C)_{\big|\Gamma_-}
\]
where by (\ref{sc-16-a}) $K^2_{r}(L_-g_C)_{|\Gamma_-}=0$. In addition, by (\ref{ashr1-a-sc})
\[
g_C,\ {\p {(L_-g_C)}{E}}_{\Big|\Gamma_-}={\p {g_C}{E}}\in T^2(\Gamma_-)\cap T_{m_{1,j}}^2(\Gamma_-).
\]
Hence 
\[
{\ol f_C}_{|\Gamma_-}= {f_C}_{|\Gamma_-}+a{\p {g_C}{E}}-(Ca+\Sigma)g_C\in T^2(\Gamma_-)
\cap T_{m_{1,j}}^2(\Gamma_-).
\]
Since $(\widetilde{ P}_{C,0}^{-1}\ol f_C)_{\big|\Gamma_-}=0$ we conclude in due to Corollary \ref{pinv-norm}  that (\ref{uu}) holds.

B.3.
That is why, by (\ref{sc-18}) the series (recall (\ref{15}))
\[
u=\sum_{k=0}^\infty
(\widetilde{ P}_{C,0}^{-1}K^2_{r})^k(\widetilde{ P}_{C,0}^{-1}\ol f_C)
\]
is converging in  $H_{-,0}^{(1,0,0)}(G\times S\times I^\circ)$ 
This shows that 
\be\label{phi-reg}
u\in
H_{-,0}^{(1,0,0)}(G\times S\times I^\circ). 
\ee

Since $\phi=u+L_-g_C$ we  see by  (\ref{phi-reg}) and (\ref{l-g}) that $\phi\in H^{(1,0,0)}(G\times S\times I^\circ)$ and so also
$\psi=e^{-CE}\phi\in H^{(1,0,0)}(G\times S\times I^\circ)$. This completes the proof.

\end{proof}

We give the following remarks.

\begin{remark}\label{adjoint-pr}
Firstly, we yield some notes concerning for the related \emph{adjoint problem.}
Consider the problem
\be\label{pr-1}
T\psi=f,\ \psi_{\Gamma_-}=g,\ \psi(.,.,E_m)=0
\ee
where $f\in L^2(G\times S\times I),\ g\in T^2(\Gamma_-)$ and
\[
T\psi=
a{\p \psi{E}}
+ c\Delta_S\psi
+d\cdot\nabla_S\psi
+\omega\cdot\nabla_x\psi+\Sigma\psi-K_{r}\psi.
\]
Denote $d(x,\omega,E,\partial_\omega)=d\cdot\nabla_S$.

Suppose that the assumptions of Theorem \ref{cor-csdath3} are valid.
By the considerations given in \cite[section 6]{tervo18-up}, \cite[section 6]{tervo19}  
 the variational formulation of the problem (\ref{pr-1}) is
\be\label{var-eq}
\widetilde B(\psi,v)=\la  f,v\ra_{L^2(G\times S\times I)}+\la g,v\ra_{T^2(\Gamma_-)}
\ee
where $\widetilde B(.,.)$ is (cf. \cite[section 6]{tervo19})
\[
&
\widetilde B(\psi,v)=
-\la \psi,{\p {(av)}{E}}\ra_{L^2(G\times S\times I)} 
-
\la a(.,E_0) p_0,v(.,.,E_0)\ra_{L^2(G\times S)}\nonumber\\
& 
-\la \nabla_S\psi,c
\nabla_Sv\ra_{L^2(G\times S\times I)}
+
\la\psi,
d^*(x,\omega,E,\partial_\omega)v\ra_{L^2(G\times S\times I)}\nonumber\\
&
-\la\psi,\omega\cdot\nabla_xv\ra_{L^2(G\times S\times I)}
+\la q_{|\Gamma_+},\gamma_+(v)\ra_{T^2(\Gamma_+)}
\nonumber\\
&
+
\la \psi, \Sigma^* v -K_{r}^*v\ra_{L^2(G\times S\times I)}.
\]
Here $d^*(x,\omega,E,\partial_\omega)=-d\cdot\nabla_S-{\rm div}_Sd,\ \Sigma^*=\Sigma,\ K_r^*$ are the (formal) adjoints of $d(x,\omega,E,\partial_\omega),\ \Sigma$ and $K_r$, respectively and $p_0\in L^2(G\times S),\ q\in T^2(\Gamma)$ are are as in \cite[section  2.1]{tervo18-up}.
We have by \cite[Theorem 6.6]{tervo19}
\be\label{es-1}
\widetilde B(\psi,\psi)\geq c'\n{\psi}_{\mc H}^2
\ee
where $c'>0$ and where the norm $\n{\cdot}_{\mc H}^2$ is as in section \ref{ex-i-i-b-p}.

Let 
\[
T'\psi^*=
-{\p {(a\psi^*)}{E}}
+ c\Delta_S\psi^*
-d\cdot\nabla_S\psi^*-({\rm div}_Sd)\ \psi^*
-\omega\cdot\nabla_x\psi^*+\Sigma\ \psi^*-K^*_{r}\psi^*
\]
be the formal adjoint of $T$. The \emph{adjoint problem} is
\be\label{pr-1-b}
T'\psi^*=f^*,\ \psi^*_{\Gamma_+}=g^*,\ \psi^*(.,.,E_0)=0
\ee
where $f^*\in L^2(G\times S\times I),\ g^*\in T^2(\Gamma_+)$.
The variational formulation of the adjoint problem (\ref{pr-1-b}) is
\be\label{var-eq-a}
\widetilde B^*(\psi^*,v)=\la  f^*,v\ra_{L^2(G\times S\times I)}+\la g^*,v\ra_{T^2(\Gamma_+)}
\ee
where $\widetilde B^*(.,.)$ is 
\[
&
\widetilde B(\psi^*,v)=
\la \psi^*,a {\p v{E}}\ra_{L^2(G\times S\times I)} 
-
\la  p_m^*,a(.,E_m)v(.,.,E_m)\ra_{L^2(G\times S)}\nonumber\\
& 
-\la \nabla_S\psi^*,c
\nabla_Sv\ra_{L^2(G\times S\times I)}
+
\la\psi^*,
d(x,\omega,E,\partial_\omega)v\ra_{L^2(G\times S\times I)}\nonumber\\
&
+\la\psi^*,\omega\cdot\nabla_xv\ra_{L^2(G\times S\times I)}
+\la q^*_{|\Gamma_-},\gamma_-(v)\ra_{T^2(\Gamma_-)}
+
\la \psi^*, \Sigma v -K_{r}v\ra_{L^2(G\times S\times I)}.
\]
Here  $p_m^*\in L^2(G\times S),\ q^*\in T^2(\Gamma)$ are again as in \cite[section  2.1]{tervo18-up}.
Analogously to (\ref{es-1}) one sees that
\be\label{es-1-a}
\widetilde B^*(\psi^*,\psi^*)\geq c'\n{\psi^*}_{\mc H}^2.
\ee
 
The adjoint problem  satisfies analogous existence and regularity results as obtained above for the problem $T\psi=f,\ \psi_{|\Gamma_-}=g, \ \psi(.,.,E_m)=0$. We leave the formulations.
However, we remark that the adjoints obey (this is not clear but it is based on the maximal dissipativity considerations like \cite[section 3.3]{tervo17})
\be\label{a-e-1}
(\widetilde T_{0,-})^*=T_{0,-}^*=\widetilde T'_{0,+}
\ee
where $\widetilde{T}_{0,-}:L^2(G\times S\times I)\to L^2(G\times S\times I)$ and
$\widetilde{T'}_{0,+}:L^2(G\times S\times I)\to L^2(G\times S\times I)$ are the smallest closed extensions of 
the operators defined, respectively by 
\bea
D({ T}_{0,-})&:=\{\psi\in \widetilde W^2(G\times S\times I)\cap H^1(I,L^2(G\times S))\cap H^2(S,L^2(G\times I))\ |
\psi_{|\Gamma_-}=0,\ \psi(\cdot,\cdot,E_{m})=0\},\nonumber\\
{ T}_{0,-}v&:=Tv.
\eea
\bea
D({ T'}_{0,+})&:=\{v\in \widetilde W^2(G\times S\times I)\cap H^1(I,L^2(G\times S))\cap H^2(S,L^2(G\times I))\ |
v_{|\Gamma_+}=0,\ v(\cdot,\cdot,E_{0})=0\},\nonumber\\
{ T'}_{0,+}v&:=T'v.
\eea
Here $T'$ is the formal adjoint of $T$ given above.
The relation (\ref{a-e-1}) is useful in further analysis of the adjoint problem.
\end{remark}

\begin{remark}\label{remarks-1}

With the same kind of techniques as above we can show regularity results with respect to $E$ and $\omega$ variables. For example, we have:

Suppose  that $a>0$ and $\Sigma>0$ are constants and that the assumptions (\ref{ass7}) and
\bea\label{ass7-sc-a}
&\int_{S'}|(\partial_{E}\sigma^2)(x,\omega',\omega,E)|d\omega'\leq M_1''\ {\rm a.e.},\nonumber\\
&\int_{S'}|(\partial_{E}\sigma^2)(x,\omega,\omega',E)| d\omega'\leq M_2''\ {\rm a.e.}
\eea
are valid for $\sigma^2(x,\omega,\omega',E)$.
Furthermore, suppose that
\be\label{fr-16}
f\in H^{(0,0,1)}(G\times S\times I^\circ).
\ee
Then the solution of the problem
\be\label{wxr-1-a}
-a{\p u{E}}+\omega\cdot\nabla_xu+\Sigma u-K^2_ru=f,\
u_{|\Gamma_-}=0, \ u(.,.,E_m)=0
\ee
belongs to $H^{(0,0,1)}(G\times S\times I^\circ)$.

\begin{proof}

The proof is based on the fact that  $\widetilde P_{C,0}^{-1}$ is a bounded operator 
$H^{(0,0,1)}(G\times S\times I^\circ)\to H^{(0,0,1)}(G\times S\times I^\circ)$ and its norm (between the spaces in question) obeys
\be\label{nor}
\n{\widetilde P_{C,0}^{-1}}\leq {{C'}\over{aC+\Sigma}}.
\ee
Moreover, the assumptions (\ref{ass7}), (\ref{ass7-sc-a}) imply that 
$K^2_{r}$ is a bounded operator $H^{(0,0,1)}(G\times S\times I^\circ) 
\to H^{(0,0,1)}(G\times S\times I^\circ)$ and its norm satisfies
\be\label{sc-12-a}
\n{K^2_{r}}_1\leq \sqrt{2}\big(\sqrt{3M_1''M_2''+M_1M_2}\big).
\ee
We omit details.

\end{proof}

\end{remark}

\begin{remark}\label{re-comp2}
Consider the problem given in Theorem \ref{sc-le2}. Suppose additionally that $K^2_r=0$ and that $f\in H^{(0,01)}(G\times S\times I^\circ)$. Then
\[
\ol f_C=f_C+a{\p {(L_-g_C)}{E}}-(Ca+\Sigma)(L_-g_C).
\]
and
\[
-a{\p u{E}}+\omega\cdot\nabla_xu+Cau+\Sigma u=\ol f_C
\]
where  as above $u=\widetilde P_{C,0}^{-1}\ol f_C$. 
Similarly as in (\ref{sc-3}), (\ref{sc-4}) we have for $t(x,\omega)<\eta(E)$
\be\label{sc-3-a}
{\p u{E}}(x,\omega,E)=\int_0^{t(x,\omega)}
e^{-(aC+\Sigma) s} {\p {\ol f_C}{E}}\big(x-s\omega,\omega,E+as\big)ds\nonumber\\
\ee
and for $t(x,\omega)>\eta(E)$
\bea\label{sc-4-a}
&
{\p u{E}}(x,\omega,E)=\int_0^{\eta(E)}
e^{-(aC+\Sigma) s} {\p {\ol f_C}{E}}\big(x-s\omega,\omega,E+as\big)ds\nonumber\\
&
+
e^{-(aC+\Sigma)\eta(E)} \ol f_C\big(x-\eta(E)\omega,\omega,E+a\eta(E))\big)(-{1\over a})
\eea
since ${\p \eta{E}}(E)=-{1\over a}$.

In order that ${\p u{E}}$ is continuous on the surface $t(x,\omega)=\eta(E)$, we must require that 
\[
&
e^{-(aC+\Sigma)\eta(E)} \ol f_C\big(x-t(x,\omega)\omega,\omega,E+a\eta(E))\big)(-{1\over a})
=0
\]
which reduces to 
\[
\ol f_C\big(x-t(x,\omega)\omega,\omega,E_m)=0
\]
since $E+a\eta(E)=E_m$.
That is why, we must impose the compatibility condition
\[
\ol f_C(y,\omega,E_m)=0, \ (y,\omega)\in \Gamma_-'
\]
that is, 
\be\label{sc-5-aa}
f_C(y,\omega,E_m)+a{\p {g_C}{E}}(y,\omega,E_m)-(Ca+\Sigma)g_C(y,\omega,E_m)=0
\ee
since  $(L_-g_C)_{|\Gamma_-}=g_C$.
Assuming that the zero order compatibility condition $g_C(.,.,E_m)=0$ holds  the condition (\ref{sc-5-aa}) reduces to
\[
f_C(y,\omega,E_m)+a{\p {g_C}{E}}(y,\omega,E_m)=0.
\]
Hence the compatibility condition (\ref{csda-ex9-a}) is necessary. These notes do not deny the $H^1$-regularity with respect to $E$-variable since $H^1$-regularity allows ${\p u{E}}$ to be discontinuous. 

\end{remark}

\subsection{A note on regularity based on the smoothing property of restricted collision operator}\label{smoothingK}

The so called smoothing property of restricted collision operator can be applied in regularity analysis. We outline shortly the idea below.

Let $m\in\N_0$, $0<\kappa<1$ and let $s:=m+\kappa$. Recall that for  an open set $G\subset\R^3$  
\[
\n{u}_{H^{s}(G)}^2=\n{u}_{H^{m}(G)}^2
+
\sum_{|\alpha|=m}\int_G\int_G{{|(\partial^\alpha u)(x)-(\partial^\alpha u)(y)|^2}\over{|x-y|^{3+2\kappa}}}dx dy.
\]
The space $H^{(s,0,0)}(G\times S\times I^\circ)$ is the subspace of 
$L^2(G\times S\times I)$ for which
\bea\label{fr-0}
&
\n{\psi}_{H^{(s,0,0)}(G\times S\times I^\circ)}^2:=\n{\psi}_{H^{(m,0,0)}(G\times S\times I^\circ)}^2
\nonumber\\
&
+
\sum_{|\alpha|=m}
\int_I\int_S\int_G\int_{G}{{|(\partial_x^\alpha\psi)(x,\omega,E)-(\partial_x^\alpha\psi)(y,\omega,E)|^2}\over{|x-y|^{3+2\kappa}}}dx dy d\omega dE<\infty
\eea 
that is, $H^{(s,0,0)}(G\times S\times I^\circ)=L^2(S\times I,H^s(G))$.
As is well known in the case where $G=\R^3$ the spaces $H^{(s,0,0)}(G\times S\times I^\circ)$ can be characterized by using the partial Fourier transform with respect to $x$-variable.

Define a  linear space
\bea
&
{\s H}^{(s,0,0)}(G\times S\times I^\circ)\nonumber\\
&
:=\{\psi\in H^{(s,0,0)}G\times S\times I^\circ)
|\ \omega\cdot\nabla_x\psi\in H^{(s,0,0)}G\times S\times I^\circ)\}\nonumber
\eea
equipped with the norm 
\[
\n{\psi}_{{\s H}^{s}(G\times S\times I^\circ)}^2:=
\n{\psi}_{H^{(s,0,0)}(G\times S\times I^\circ)}^2
+
\n{\omega\cdot\nabla\psi}_{H^{(s,0,0)}(G\times S\times I^\circ)}^2.
\]
Note that ${\s H}^{0}(G\times S\times I^\circ)= W^{2}(G\times S\times I)$.
Furthermore, define the {\it moment operator}  $M:L^2(G\times S\times I)\to L^2(G)$ by
\be\label{spco1a}
({\mc M}\psi)(x):=\int_{S\times I}\psi(x,\omega,E) d\omega dE.
\ee
We have the following {\it moment (or velocity averaging) lemma}.

\begin{theorem}\label{spcoth1}
A. Suppose that $G=\R^3$ and let $s\geq 1$.  
Then ${\mc M}:{\s H}^{(s-1)/2,0,0)}(\R^3\times S\times I^\circ)\to H^{s/2}(\R^3)$ is bounded.

B. Suppose that $G$ is {\it convex} and bounded.
Then ${\mc M}:\widetilde W^{2}(G\times S\times I)\to H^{1/2}(G)$ is bounded.
\end{theorem}

\begin{proof} 
The claim follows by an adaptation of the proof given in \cite{golse} (cf. also
techniques of \cite[the proof of Lemma 1.2]{chen20}).
\end{proof}

We formulate the next corollary concerning  the restricted collision operator (below the primes in $S'$ and $I'$ refer to their arguments $\omega'$ and $E'$, respectively). For simplicity, we assume that the restricted collision operator $K_r$ is of the form $K_r=K_r^1$ but the claims of the below Corollary \ref{spcocor2} are valid also when $K_r$ is of the form $K_r=K_r^2$.

\begin{corollary}\label{spcocor2}
A. Suppose that $s\geq 1$ and that $\sigma^1:\R^3\times S'\times S\times I'\times I\to\R$ is a measurable function   such that 
\bea\label{spco18}
&
\sigma^1\in L^2(S\times I,W^{\infty,((s-1)/2,0,0)}(\R^3\times S'\times {I'}^\circ))
\nonumber\\
&
\omega'\cdot\nabla_x\sigma^1\in L^2(S\times I,W^{\infty,((s-1)/2,0,0)}(\R^3\times S'\times {I'}^\circ)).
\eea
Then $K_r^1:{\s H}^{((s-1)/2,0,0)}(\R^3\times S\times I^\circ)\to H^{(s/2,0,0)}(\R^3\times S\times I^\circ)$ given
by
\[
(K_r^1\psi)(x,\omega,E)=
\int_{S'\times I'}\sigma^1(x,\omega',\omega,E',E)\psi(x,\omega',E')d\omega' dE'
\]
is bounded.

B. Suppose that $G$ is convex and bounded. Furthermore, suppose that $\sigma^1: G\times S'\times S\times I'\times I\to\R$ is a measurable   function
and that
\bea\label{spco13}
&
\sigma^1\in L^2(S\times I,L^\infty(G\times S'\times I'))
\nonumber\\
&
\omega'\cdot\nabla_x\sigma^1\in L^2(S\times I,L^\infty(G\times S'\times I'))
\nonumber\\
&
\sigma^1\in L^2(S\times I,L^\infty(\Gamma'\times I'))
\eea
(here $\Gamma'=(\partial G)\times S'$).
Then $K_r^1:\widetilde W^2(G\times S\times I)\to H^{(1/2,0,0)}(G\times S\times I)$ is bounded. 
\end{corollary}

\begin{proof}
 
Define for a fixed $(\omega,E)\in S\times I$ a mapping $\Psi_{\omega,E}:G\times S'\times I'\to \R$ by
\[
\Psi_{\omega,E}(x,\omega',E'):=\sigma^1(x,\omega',\omega,E',E)\psi(x,\omega',E').
\]
Noting that 
\[
(K_r^1\psi)(x,\omega,E)=({\mc M}\Psi_{\omega,E})(x)
\]
the proof follows from Theorem \ref{spcoth1}. We omit details.
\end{proof}

Consider the transport problem
\be\label{fr-1}
T\psi=f,\ \psi_{|\Gamma_-}=g,\ \psi(.,.,E_m)=0.
\ee
Let $u:=\psi-L_-g$.
Assuming that the incoming data $g$ is smooth enough and that the compatibility condition $g(.,.,E_m)=0$ holds 
the problem (\ref{fr-1}) is equivalent to
\be\label{fr-2}
Tu=f-T(L_-g),\ u_{|\Gamma_-}=0,\ u(.,.,E_m)=0.
\ee
Hence it suffices to consider only the problem of the form 
\be\label{fr-2a}
Tu=f,\ u_{|\Gamma_-}=0,\ u(.,.,E_m)=0
\ee
for which the homogeneous inflow boundary condition $u_{|\Gamma_-}=0$ holds.
Assume (for simplicity) that the transport operator $T$ is of the form
\[
Tu=-a{\p {u}E}+\omega\cdot\nabla_xu+
\Sigma u-K_ru
\]
where $a$ and $\Sigma$ are positive constants
and where $K_r=K_r^2$. 

Let the assumptions of Theorem \ref{coth3-dd} and Remark \ref{remarks-1} be valid.
Due to (\ref{comp12}) (here we use notations of section \ref{meta})
\be\label{fr-14}
u=\widetilde{ P}_{C,0}^{-1}K_{r,C}u+\widetilde{ P}_{C,0}^{-1} f_C.
\ee
By iteration we obtain from (\ref{fr-14})
\bea\label{fr-15}
&
u=\widetilde{ P}_{C,0}^{-1}K_{r,C}u+\widetilde{ P}_{C,0}^{-1} f_C
\nonumber\\
&
=
\widetilde{ P}_{C,0}^{-1}K_{r,C}\big(\widetilde{ P}_{C,0}^{-1}K_{r,C}u+\widetilde{ P}_{C,0}^{-1} f_C\big)+\widetilde{ P}_{C,0}^{-1}  f_C
\nonumber\\
&
=
(\widetilde{ P}_{C,0}^{-1}K_{r,C})^2u+(\widetilde{ P}_{C,0}^{-1}K_{r,C})\widetilde{ P}_{C,0}^{-1}  f_C
+\widetilde{ P}_{C,0}^{-1} f_C
\nonumber\\
&
=
(\widetilde{ P}_{C,0}^{-1}K_{r,C})^3u+(\widetilde{ P}_{C,0}^{-1}K_{r,C})^2\widetilde{ P}_{C,0}^{-1} f_C
+(\widetilde{ P}_{C,0}^{-1}K_{r,C})\widetilde{ P}_{C,0}^{-1} f_C
+\widetilde{ P}_{C,0}^{-1} f_C.
\eea
The iteration can be continued in the similar way.

The iteration formula (\ref{fr-15}) together with the smoothing property of collision operator  can be used to prove the regularity of solutions 
(cf. \cite{chen20}). The idea is as follows:

Let $f_C\in H^{(\kappa,0,1)}(G\times S\times I^\circ),\ 0<\kappa\leq 1/2$. 
Then $f_C\in H^{(\kappa,0,0)}(G\times S\times I^\circ)$ and we 
are able to deduce (under some additional assumptions) that
\be\label{frac-P}
w:=\widetilde{ P}_{C,0}^{-1} f_C\in  H^{(\kappa,0,0)}(G\times S\times I^\circ).
\ee
This is based on the following decomposition.
In virtue of (\ref{ts-1}) the solution $w$
is 
\be\label{fr-3}
w(x,\omega,E)= 
\int_0^{\alpha(x,\omega,E)}
e^{-\Sigma s}
\cdot
 f_C(x-s\omega,\omega,E+as) ds
\ee
where $\alpha(x,\omega,E):=\min\{\eta(E),t(x,\omega)\},\ \eta(E)={{E_m-E}\over a}$.
We decompose
\bea\label{fr-4}
&
w(x,\omega,E)-w(y,\omega,E)\nonumber\\
&
= 
\int_0^{\alpha(x,\omega,E)}
e^{-\Sigma s}
\cdot
\big(  f_C(x-s\omega,\omega,E+as) 
-  f_C(y-s\omega,\omega,E+as) \big)ds
\nonumber\\
&
+
\Big(\int_{0}^{\alpha(x,\omega,E)}
e^{-\Sigma s}
 f_C(y-s\omega,\omega,E+as) ds
-
\int_{0}^{\alpha(y,\omega,E)}
e^{-\Sigma s} 
 f_C(y-s\omega,\omega,E+as) ds\Big)\nonumber\\
&
=:
v_1(x,y,\omega,E)+v_2(x,y,\omega,E).
\eea
Note that
\[
v_2(x,y,\omega,E)=-
\int_{\alpha(x,\omega,E)}^{\alpha(y,\omega,E)}
e^{-\Sigma s}
\cdot
 f_C(y-s\omega,\omega,E+as) ds
\]
It suffices to show that 
\[
Q_j:=\int_I\int_S\int_G\int_G {{|v_j(x,y,\omega,E)|^2}\over{|x-y|^{3+2\kappa}}} dy dx  d\omega dE  <\infty,\ j=1,2.
\]
The term 
$
Q_1  
$
can be seen to be finite by fairly standard estimations.
The term 
$
Q_2
$
requires  more specific consideration and we do not analyse it  here (see \cite[the proof of Lemma 1.12]{chen20}).

Furthermore, due to Remark \ref{remarks-1}
$u\in 
H^{(0,0,1)}(G\times S\times I^\circ)$ and so
${\p u{E}}\in L^2(G\times S\times I)$. This implies that
\be\label{fr-23}
\omega\cdot\nabla_xu=a{\p u{E}}-\Sigma u+K^2_{r,C}u+f_C\in
L^2(G\times S\times I)
\ee
and so $u\in  W^2(G\times S\times I)$. In addition, $u_{|\Gamma_-}=0$ and then $u\in \widetilde W^2(G\times S\times I)$.
In virtue of Corollary \ref{spcocor2} $K_{r,C}u\in H^{(1/2,0,0)}(G\times S\times I^\circ)$. 
Hence similarly to (\ref{frac-P}) (under appropriate additional assumptions)
$\widetilde{ P}_{C,0}^{-1}K_{r,C}u\in  H^{(\kappa,0,0)}(G\times S\times I^\circ)$
and so by  (\ref{fr-15}) 
\be\label{fr-16}
u=\widetilde{ P}_{C,0}^{-1}K_{r,C}u+\widetilde{ P}_{C,0}^{-1} f_C
\in H^{(\kappa,0,0)}(G\times S\times I^\circ).
\ee

Higher order regularity up to certain (fractional) order can be obtained by using higher order iteration  (\ref{fr-15}). In the case where $G=\R^3$ the solution is 
\[
u(x,\omega,E)= 
\int_0^{\eta(E)}
e^{-\Sigma s}
\cdot
f_C(x-s\omega,\omega,E+as) ds.
\]
and the method can be founded fairly straightforwardly on Corollary \ref{spcocor2}/Part A. In this case only the repeating of the first iteration formula (\ref{fr-15}) is sufficient.
For example, when we  know that 
$u\in  {\mc H}^{(\kappa,0,0)}(\R^3\times S\times I^\circ)$ we can under appropriate assumptions conclude by Corollary
\ref{spcocor2}/Part A (by choosing $s=2\kappa+1$) that $K_{r,C}u\in H^{(2\kappa,0,0)}(\R^3\times S\times I^\circ)$
and so
\[
u=\widetilde{ P}_{C,0}^{-1}K_{r,C}u
+\widetilde{ P}_{C,0}^{-1}f_C\in H^{(2\kappa,0,0)}(\R^3\times S\times I^\circ)
\]
where we additionally used the (quite easily verifiable) fact that for $s\geq 0$
\[
w=\widetilde{ P}_{C,0}^{-1}f\in H^{(s,0,0)}(\R^3\times S\times I^\circ)\ {\rm when}\ 
f\in H^{(s,0,0)}(\R^3\times S\times I^\circ).
\]
In the case where $G\subset\R^3$  is bounded the method needs more subtle treatise because the escape time mapping $t(x,\omega)$ is included in the solution formula (cf. \cite{chen20} where regularity with respect to $x$ up to order $1-\epsilon$ is achieved).

Suitability of this method to more general transport equations is potential.

\section{Some  outlines for regularity eminating from general theory of first order PDEs}\label{outlines}

\subsection{Transport problem in rectified coordinates}\label{breg}

In some cases the transport problem in the so called rectified coordinates is useful.
At first we settle how the problem
\bea\label{breg-0a}
&
T\psi=
a(x,E){\p \psi{E}}
+c(x,E)\Delta_S\psi+d(x,\omega,E)\cdot\nabla_S\psi
+\omega\cdot\nabla_x\psi+\Sigma\psi-K_{r}\psi\nonumber\\
&
\psi_{|\Gamma_-}=g,\ \psi(.,.,E_m)=0
\eea
locally changes when the boundary is rectified.

The boundary of a sufficiently regular domain $G\subset\R^3$ can be rectified in the following sense. 
The boundary $\partial G$ of $G$ is said to be in the {\it class} $C^{k}$, or shortly $\partial G\in C^{k}$,
(see e.g. \cite{grisvard}, section 1.2) if for any $y\in\partial G$ there exists an open neighbourhood $U_y\subset\R^3$ of $y$ and a diffeomorphism $h:U_y\to D_{a,b}$ such that
\bea\label{breg-0b}
&
h\in C_*^{k}(U_y,D_{a,b}),\ h^{-1}\in C_*^{k}(D_{a,b},U_y),\nonumber\\
&
h^{-1}(D_{+,a,b})=G\cap U_y,\ 
h^{-1}(D_{0,a,b})= (\partial G)\cap U_y
\eea
where
\[
C_*^{k}(U,U'):=\{f\in C^k(U,U')|\ {\rm for\ which}\ 
\n{f}_{C_*^{k}(U,U')}:=\sum_{|\alpha|\leq k}\n{\partial_x^\alpha f}_{L^\infty(U,U')}
<\infty\},
\]
\[
D_{a,b}:=\{z=(z',z_3)\in\R^3|\ \n{z'}<a,\ |z_3|<b\},
\]
\[
D_{+,a,b}:=\{z\in D_{a,b}|\ z_3>0\},
\]
\[
D_{0,a,b}:=\{z\in D_{a,b}|\ z_3=0\}.
\]

Firstly we consider in which way the equation $T\psi=f$ changes under the diffeomorphism $h:U_y\to D_{a,b}$. We assume   that the coefficients $a,\ c,\ d,\ \Sigma,\ \sigma^j,\ j=1,2,3$ are regular enough up to the boundary.
Let $(y,\omega,E)\in \Gamma$  and
let $\phi:D_{+,a,b}\to\R$ be a mapping
\[
\phi(z,\omega,E):=(\psi\circ h^{-1})(z,\omega,E):=\psi(h^{-1}(z),\omega,E).
\]
Then $\psi=\phi\circ h$ in $G\cap U_y$ when we define $(\phi\circ h)(x,\omega,E):=\phi(h(x),\omega,E)$. Since $h$ does not depend on $\omega$ and $E$ we have for $(z,\omega,E)\in V_+:=D_{+,a,b}\times S\times I$
\[
&
(a{\p \psi{E}})\circ h^{-1}=(a\circ h^{-1}){\p \phi{E}},\ 
(c\Delta_S\psi)\circ h^{-1}=(c\circ h^{-1})\Delta_S\phi
\\
&
(d\cdot\nabla_S\psi)\circ h^{-1}=(d\circ h^{-1})\cdot\nabla_S\phi,\
(\Sigma\psi)\circ h^{-1}=(\Sigma\circ h^{-1})\phi
\]
and
\bea\label{cKr}
((K_{r}\psi)\circ h^{-1})(z,\omega,E)&=
\int_{S'\times I'}\sigma^1(h^{-1}(z),\omega',\omega,E',E)\phi(z,\omega',E')d\omega' dE'\nonumber\\
&
+
\int_{ S'}\sigma^2(h^{-1}(z),\omega',\omega,E)\phi(z,\omega',E) d\omega',  
\nonumber\\
&
+
\int_{I'}\int_{0}^{2\pi}
\sigma^3(h^{-1}(z),E',E)
\phi(z,\gamma(E',E,\omega)(s),E')ds dE'
\nonumber\\
&
=:(\widetilde K_r\phi)(z,\omega,E).
\eea

Furthermore,
\[
{\p \psi{x_j}}={\p {(\phi\circ h)}{x_j}}=\la(\nabla_z\phi)\circ h,{\p h{x_j}}\ra 
\]
and so
\[
&
\omega\cdot\nabla_x\psi=\sum_{j=1}^3\omega_j\la (\nabla_z\phi)\circ h,{\p h{x_j}}\ra.
\]
Hence
\be\label{breg-1}
(\omega\cdot\nabla_x\psi)\circ h^{-1}=\la\nabla_z\phi,\sum_{j=1}^3\omega_j{\p h{x_j}}\circ h^{-1}\ra
=
\tilde b(z,\omega)\cdot\nabla_z\phi
\ee
where
\be\label{tb}
\tilde b(z,\omega):=\sum_{j=1}^3\omega_j{\p h{x_j}}\circ h^{-1}
=(\omega\cdot\nabla_xh_1,\omega\cdot\nabla_xh_2,\omega\cdot\nabla_xh_3)\circ h^{-1}.
\ee

The equation $T\psi=f$ holds in $(G\cap U_y)\times S\times I$ if and only if $(T\psi)\circ h^{-1}=f\circ h^{-1}$ in  $V_+$ that is,
\[
&
\big(a{\p \psi{E}}\big)\circ h^{-1}
+(c\Delta_S\psi)\circ h^{-1}+(d\cdot\nabla_S\psi)\circ h^{-1}
+(\omega\cdot\nabla_x\psi)\circ h^{-1}\\
&
+(\Sigma\psi)\circ h^{-1}-(K_{r}\psi)\circ h^{-1}=f\circ h^{-1}
\]
which by the above computations is equivalent to
\be\label{breg-2}
\tilde a(z,E){\p \phi{E}}
+\tilde c(z,E)\Delta_S\phi+\tilde d(z,\omega,E)\cdot\nabla_S\phi
+\tilde b(x,\omega)\cdot\nabla_z\phi+\tilde\Sigma\phi-\widetilde K_{r}\phi=\tilde f
\ee
where $\tilde b(z,E)$ and $\widetilde K_r$ are given by  (\ref{tb}), (\ref{cKr}), respectively and where 
\[
&
\tilde a(z,E):=a(h^{-1}(z),E),\ \tilde c(z,E):=c(h^{-1}(z),E),\ 
\tilde d(z,\omega,E):=d(h^{-1}(z),\omega,E),\\
&
\tilde \Sigma(z,\omega,E):=\Sigma(h^{-1}(z),\omega,E),\ \tilde f(z,\omega,E):=f(h^{-1}(z),\omega,E).
\]

Secondly, we consider the inflow boundary condition.
We find that (recall (\ref{breg-0b}))
\[
(\partial G)\cap U_y=\{y'\in U_y|\ h_3(y')=0\}
\]
and so (when we fix $h_3$ such that $-\nabla_xh_3$ is pointing outwards on $(\partial G)\cap U_y$).  
\[
\nu(y')=-{{\nabla_xh_3(y')}\over{\n{\nabla_xh_3(y')}}}.
\]
Then
\be\label{breg-3}
\omega\cdot\nu(y')<0\Leftrightarrow \omega\cdot \nabla_xh_3(y')>0.
\ee
Let $n:=(0,0,-1)$ (=the outward pointing unit normal on on $D_{0,a,b}$) and let $\ol\sigma$ be the surface measure  $D_{0,a,b}\subset\R^3$.
Then we find that
\[
\tilde b\cdot n=\sum_{j=1}^3((\omega\cdot\nabla_xh_j)\circ h^{-1})n_j
=-(\omega\cdot\nabla_xh_3)\circ h^{-1}
\]
and so by (\ref{breg-3})
\be\label{breg-5}
\omega\cdot\nu(y')<0\Leftrightarrow \omega\cdot\nabla_xh_3(y')>0
\Leftrightarrow  \tilde b_3(z,\omega)>0
\Leftrightarrow \tilde b(z,\omega)\cdot n(z)=-\tilde b_3(z,\omega)<0
\ee
where $y'=h^{-1}(z)$. 
Note that the initial condition leaves $\phi(.,.,E_m)=0$.

Let 
\[
{R}_{0,a,b,-}':=\{(z,\omega)\in D_{0,a,b}\times S|\ \tilde b_3(z,\omega)<0\},\ {R}_{0,a,b,-}:=R_{0,a,b,-}'\times I.
\]
As a conclusion we see that the problem 
(\ref{breg-0a}) is \emph{locally around $\Gamma$} equivalent to the problem
\bea\label{breg-7}
&
\widetilde T\phi:=\tilde a(z,E){\p \phi{E}}
+\tilde c(z,E)\Delta_S\phi+\tilde d(z,\omega,E)\cdot\nabla_S\phi
+\tilde b(x,\omega)\cdot\nabla_z\phi+\tilde\Sigma\phi-\widetilde K_{r}\phi=\tilde f\nonumber\\
&
\phi_{|R_{0,a,b,-}}=\tilde g,\ \phi(.,.,E_m)=0
\eea 
where
\[
\tilde g(z,\omega,E):=g(h^{-1}(z),\omega,E).
\]
We see that the  problem is again of {variable multiplicity}.

\begin{remark}\label{cc-re}
Using the local results (that is, results for patches $G\cap U_y$), the
global results (that is, results for the whole $G$) can be retrieved by applying the usual \emph{partition of unity process} (\cite[section 1.2]{wloka}).
Above we treated only localization around $\Gamma$. Using the standard localization and partition of unity methods only these patches are interesting.
The interior patches can be easily  handled.
\end{remark}

\subsection{On tangential regularity}\label{treg}

The solutions of the characteristic initial  boundary value problems  do not have necessarily full Sobolev regularity (e.g. \cite[Examples 2.1, 2.2]{nishitani96}). However, the problems whose boundary matrix  has a constant rank (problems with constant multiplicity) the so called tangential (or co-normal) Sobolev regularity can be shown under relevant assumptions (e.g. \cite{rauch85}, \cite{morando09}). For the problems with variable multiplicity even the tangential regularity need not be valid  but for some problems it may be possible.  
In the sequel we formulate some details of tangential regularity concerning for transport problems under consideration.

A smooth vector field $F=\sum_{j=1}^3a_j\partial_{x_j}:\ol G\times S\times I\to T(\R^3)=\R^3$ is \emph{tangential (or co-normal) with respect to $\Gamma=(\partial G)\times S\times I$} if (here $a:=(a_1,a_2,a_3)$) 
\[
\la a,\nu\ra=0\ {\rm on}\ \Gamma.
\]
Similarly, we define tangential (or co-normal) vector fields with respect to $\Gamma_{\pm}$ and
tangential (or co-normal) vector fields $F:\ol D_{+,a,b}\times S\times I\to T(\R^3)$ with respect to $D_{0,a,b}\times S\times I$.

We find that
$\{\partial_{z_1},\partial_{z_2}, z_3\partial_{z_3}\}$ forms  a  basis for the space of tangential vector fields $F:\eta\in \ol D_{+,a,b}\times S\times I\to T(\R^3)$ with respect to $D_{0,a,b}\times S\times I$.
This follows from the fact that for any smooth $\eta\in \ol D_{+,a,b}\times S\times I$ obeying $\eta(z_1,z_2,0,\omega,E)=0$ 
we have
\[
\eta(z_1,z_2,z_3,\omega,E)=z_3\int_0^1{\p \eta{z_3}}(z_1,z_2,tz_3,\omega,E)dt.
\]
This implies that the space of tangential vector fields with respect to $\Gamma$ has  a natural local basis (near $\Gamma$) as well which can be seen as follows.
We find that for $\phi=\psi\circ h^{-1}$ and for $j=1,2$
\[
{\p \phi{z_j}}=\la (\nabla_x\psi)\circ h^{-1},{\p {(h^{-1})}{z_j}}\ra.
\]
Let $A_h$ and $A_{h^{-1}}$ be the Jacobian matrices of $h$ and $h^{-1}$, respectively.
Since
\[
A_{h^{-1}}=\qmatrix{{\p {(h^{-1})}{z_1}}&{\p {(h^{-1})}{z_2}}&
{\p {(h^{-1})}{z_3}}\\}
=A_h^{-1}\circ h^{-1}=(({\rm det}(A_h))^{-1}\circ h^{-1})\qmatrix{A_1&A_2&A_3\\}\circ h^{-1}
\]
where
\[
A_1=
\qmatrix{\partial_2h_2\partial_3h_3-\partial_2h_3\partial_3h_2\\ 
\partial_1h_2\partial_3h_3-\partial_1h_3\partial_3h_2\\
\partial_1h_2\partial_2h_3-\partial_1h_3\partial_2h_2\\},
A_2=\qmatrix{\partial_3h_1\partial_2h_3-\partial_3h_3\partial_2h_1\\ 
\partial_1h_1\partial_3h_3-\partial_1h_3\partial_3h_1\\
\partial_2h_1\partial_1h_3-\partial_2h_3\partial_1h_1\\},
A_3= \qmatrix{\partial_2h_1\partial_3h_2-\partial_2h_2\partial_3h_1\\ 
\partial_3h_1\partial_1h_2-\partial_3h_2\partial_1h_1\\
\partial_1h_1\partial_2h_2-\partial_1h_2\partial_2h_1\\}
\]
we see that
\[
{\p {(h^{-1})}{z_j}}=(({\rm det}(A_h))^{-1}\circ h^{-1})A_j\circ h^{-1}. 
\]
Hence
\bea\label{treg-5}
&
{\p \phi{z_j}}=(({\rm det}(A_h))^{-1}\circ h^{-1})\la (\nabla_x\psi)\circ h^{-1},A_j\circ h^{-1}\ra
\nonumber\\
&
=
(({\rm det}(A_h))^{-1}\circ h^{-1})\la \nabla_x\psi,A_j\ra\circ h^{-1}.
\eea

Furthermore, since 
\[
\nu=-{{\nabla_xh_3}\over{\n{\nabla_xh_3}}}
\]
we see that on $\Gamma$ 
\bea\label{treg-6}
&
-\n{\nabla_xh_3}\la \nu,A_1\ra \nonumber\\
&
=-\big(
\partial_1h_3\partial_2h_2\partial_3h_3-\partial_1h_3\partial_2h_3\partial_3h_2+ 
\partial_2h_3\partial_1h_2\partial_3h_3-\partial_2h_3\partial_1h_3\partial_3h_2
\nonumber\\
&
+
\partial_3h_3\partial_1h_2\partial_2h_3-\partial_3h_3\partial_1h_3\partial_2h_2\big) 
=0.
\eea
Similarly 
\be\label{treg-7}
\n{\nabla_xh_3}\la \nu,A_2\ra=0
\ee
on $\Gamma$
and so $F_j:=({\rm det}(A_h))^{-1}\la A_j,\nabla_x\ra,\ j=1,2$ are tangential vector fields with respect to $\Gamma$.
Finally, we analogously find that
\be\label{treg-8}
z_3{\p \phi{z_3}}
=
(h_3(({\rm det}(A_h))^{-1})\circ h^{-1})\la \nabla_x\psi,A_3\ra\circ h^{-1}.
\ee
Since $h_3(y)=0$ on $\Gamma$ we see that
also $F_3:=h_3({\rm det}(A_h))^{-1}\la A_3,\nabla_x\ra$ is a tangential vector field with respect to $\Gamma$.
The above computations show that $\{F_1,F_2,F_3\}:=\{\la A_1,\nabla_x\ra,\ \la A_2,\nabla_x\ra,\ h_3\la A_3,\nabla_x\ra\}$ forms (near $\Gamma$) a local basis of tangential vector fields with respect to $\Gamma$.

We define
\[
&
H^{(m,0,0)}(G\times S\times I^\circ;\Gamma):=L^2(S\times I,H^m(G;\Gamma))\\
&
:=\{\psi\in L^2(G\times S\times I)|\ F^\alpha \psi\in L^2(G\times S\times I)\ {\rm for}\ |\alpha|\leq m\}
\]
where $F\alpha:=F_1^{\alpha_1}F_2^{\alpha_2}F_3^{\alpha_3}$ and $\{F_1,F_2,F_3\}$ is a local basis of tangential vector fields with respect to $\Gamma$. 
Similarly we define tangential Sobolev spaces $H^{(m,0,0)}(G\times S\times I^\circ;\Gamma_{\pm})$ and the tangential Sobolev spaces $H^{(m,0,0)}(D_{+,a,b,}\times S\times I^\circ;D_{a,b,0}\times S\times I)$.
Note that $H^{(m,0,0)}(G\times S\times I^\circ;\Gamma)\subset L^2(S\times I,H^m_{\rm loc}(G))$. 
Moreover, we find that $\psi\in H^{(m,0,0)}(G\times S\times I^\circ;\Gamma)$
if and only if locally 
\be\label{ztan}
\n{Z^\alpha (\psi\circ h^{-1})}_{L^2(D_{+,a,b}\times S\times I)}<\infty\ {\rm for\ all}\ |\alpha|\leq m
\ee
where $Z^\alpha:=\partial_{z_1}^{\alpha_1}\partial_{z_2}^{\alpha_2}(z_3\partial_{z_3})^{\alpha_3}$ for $\alpha=(\alpha_1,\alpha_2,\alpha_3)\in \N_0^3$. 
The condition (\ref{ztan})means that locally $\phi=\psi\circ h^{-1}\in 
H^{(m,0,0)}(D_{+,a,b,}\times S\times I^\circ;D_{0,a,b}\times S\times I)$. 
$H^{(m,0,0)}(G\times S\times I^\circ;\Gamma)$ can be equipped with a natural norm by glueing together  local patches. Actually, it can be equipped with an inner product and it is a Hilbert space. 
We omit  details.

\begin{example}\label{treg-ex1}
In the case of open ball $G=B(0,1)$ we can choose (locally) in a neighbourhood $e_3=(0,0,1)$ 
\[
h(x)=(x_1,x_2,1-|x|^2).
\]
The local inverse of $h$  is $h^{-1}(z)=(z_1,z_2,\sqrt{1-z_3-z_1^2-z_2^2})$.
We have
\[
\partial_1h=(1,0,-2x_1),\ \partial_2h=(0,1,-2x_2), \ \partial_3h=(0,0,-2x_3)
\]
and so (recall that in this case $\nu(x)=x$),  
\[
&
A_1=
(\partial_2h_2\partial_3h_3-\partial_2h_3\partial_3h_2,
\partial_1h_2\partial_3h_3-\partial_1h_3\partial_3h_2,
\partial_1h_2\partial_2h_3-\partial_1h_3\partial_2h_2)=2(-x_3,0,x_1),\\
&
A_2=(\partial_3h_1\partial_2h_3-\partial_3h_3\partial_2h_1, 
\partial_1h_1\partial_3h_3-\partial_1h_3\partial_3h_1,
\partial_2h_1\partial_1h_3-\partial_2h_3\partial_1h_1)
=2(0,-x_3,x_2),\\
&
A_3= (\partial_2h_1\partial_3h_2-\partial_2h_2\partial_3h_1,
\partial_3h_1\partial_1h_2-\partial_3h_2\partial_1h_1,
\partial_1h_1\partial_2h_2-\partial_1h_2\partial_2h_1)
=(0,0,1).
\]
Hence the (local) basis of tangential vector fields near $\Gamma$ with respect to $\Gamma$  is $\{-x_3\partial_{x_1}+x_1\partial_{x_3}, \ -x_3\partial_{x_2}+x_2\partial_{x_3},\ (1-|x|^2)\partial_{x_3}\}$. 

\end{example}

The below Example \ref{exreg-a} verifies that the higher order tangential regularity is not necessarily valid for the present transport problems.
We need the next lemma.

\begin{lemma}\label{conv}
Let $G=B(0,1)$ and let $k\in\N$  and $q>0$. The integral
\be 
\int_{G\times S}{{(\omega\cdot x)^k}\over{((\omega\cdot x)^2+1-|x|^2)^q}} d\omega dx
\ee
convergences if and only if $k-2q>-3$.
\end{lemma}

\begin{proof}
Noting that $\nu(y)=y,\ \tau_-(y,\omega)=2|\omega\cdot y|$ and that on $\Gamma_{-}'$ 
\[
(\omega\cdot (y+t\omega))^2+1-|y+t\omega|^2
=(\omega\cdot y)^2
\]
we have by Lemma \ref{pr:fubini} 
\bea\label{re-tan-ex-2}
&
\int_{G\times S}{{(\omega\cdot x)^k}\over{((\omega\cdot x)^2+1-|x|^2)^q}} d\omega dx=
\int_{\Gamma_-'}\int_0^{2|\omega\cdot y|}
{{(\omega\cdot (y+t\omega))^k}\over{((\omega\cdot (y+t\omega))^2+1-|y+t\omega|^2)^q}}|\omega\cdot y| dt d\sigma(y) d\omega \nonumber
\\
&
=
\int_{\Gamma_-'}\int_0^{2|\omega\cdot y|}
{{(\omega\cdot y+t)^k}\over{(\omega\cdot y)^{2q}}}|\omega\cdot y| dt d\sigma(y) d\omega \nonumber \\
&
=
\sum_{j=0}^k{k\choose j}
\int_{\Gamma_-'}\int_0^{2|\omega\cdot y|}
{{(\omega\cdot y)^jt^{k-j}}\over{(\omega\cdot y)^{2q}}}|\omega\cdot y| dt d\sigma(y) d\omega \nonumber
\\
&
=
\sum_{j=0}^k{k\choose j}{1\over{k-j+1}}
\int_{\Gamma_-'}
{{(\omega\cdot y)^j(2|\omega\cdot y|)^{k-j+1}}\over{(\omega\cdot y)^{2q}}}|\omega\cdot y|  d\sigma(y) d\omega \nonumber\\
&
=
\sum_{j=0}^k{k\choose j}{1\over{k-j+1}}2^{k-j+1}(-1)^{j-2q}
\int_{\Gamma_-'}
|\omega\cdot y|^{k-2q+2}  d\sigma(y) d\omega 
\eea
where we additionally used that $|\omega\cdot y|=-\omega\cdot y$ on $\Gamma_-'$.

Consider the integral of the form $\int_{\Gamma_-'}|\omega\cdot y|^s d\sigma(y) d\omega$. 
Let
\[
\Gamma_{-,\omega}':=\{y\in\partial G|\ \omega\cdot y<0\},\
S_{\tau,\omega}:=\{y\in \partial G|\ \omega\cdot y=\tau\}.
\]
Note that  each $\Gamma_{-,\omega}$ is a hemisphere and $S_{\tau,\omega}$ is a circle whose radius is $\sqrt{1-\tau^2}$. Moreover, let $d\ell_{\tau,\omega}$ be the differential path length on $S_{\tau,\omega}$.
In virtue of the Fubin's Theorem 
(see \cite{tervo18-up}, section 3.2)
\bea\label{re-tan-ex-3}
&
\int_{\Gamma_-'}|\omega\cdot y|^s d\sigma(y) d\omega
=\int_S\int_{\Gamma_{-,\omega}}|\omega\cdot y|^s d\sigma(y) d\omega
=\int_S\int_{0}^1\int_{S_{\tau,\omega}}|\omega\cdot y|^s d\ell_{\tau,\omega}(y){1\over{\sqrt{1-\tau^2}}} d\tau d\omega
\nonumber\\
&
=\int_S\int_{0}^1\int_{S_{\tau,\omega}}\tau^s  {1\over{\sqrt{1-\tau^2}}}
d\ell_{\tau,\omega}(y)d\tau d\omega
=2\pi \int_S\int_{0}^1\tau^s \sqrt{1-\tau^2}{1\over{\sqrt{1-\tau^2}}} d\tau d\omega
=8\pi \int_{0}^1\tau^s  d\tau
\eea
and so the integral $\int_{\Gamma_-'}|\omega\cdot y|^s d\sigma(y) d\omega$  convergences if and only if $s>-1$.
Hence  the integral (\ref{re-tan-ex-2}) convergences if and only if
\[
k-2q+2>-1 \Leftrightarrow k-2q>-3
\]
as desired.
\end{proof}

\begin{remark}
In the previous lemma it is not essential that $k\in \N$ since the integral
$\int_0^{2|\omega\cdot y|}
(\omega\cdot y+t)^k dt$ can be computed for any $k\in\R$. However, the computations above are sufficient for our needs. Especially we find that 
the integral $
\int_{G\times S}{1\over{((\omega\cdot x)^2+1-|x|^2)^q}} d\omega dx$ convergences if and only if $q<{3\over 2}$. Since $|\omega\cdot y+t|\leq  |\omega\cdot y|+t$ we  also see by the above proof that the integral
$
\int_{G\times S}{{|\omega\cdot x|^k}\over{((\omega\cdot x)^2+1-|x|^2)^q}} d\omega dx
$
convergences if $k-2q>-3$. 
\end{remark}

\begin{example}\label{exreg-a}
 
Let $G=B(0,1)\subset \R^3$ and consider the mono-kinetic problem (given in \cite[Example 7.4]{tervo17-up})
\[
\omega\cdot\nabla\psi+\psi=1, \
\psi_{|\Gamma_-}=0.
\]
The solution of the problem is
\[
\psi=1-e^{-t(x,\omega)},
\]
where
\[
t(x,\omega)=x\cdot \omega +\sqrt{(x\cdot \omega)^2+1-|x|^2}.
\]
We find that
\[
{\p {\psi}{x_j}}=e^{-t(x,\omega)}{\p t{x_j}}
=e^{-t(x,\omega)}\omega_j
+
e^{-t(x,\omega)}
{{(x\cdot\omega)\omega_j-x_j}\over{\big((x\cdot\omega)^2+1-|x|^2\big)^{1/2}}}
\]
and then
\[
\nabla_x\psi=e^{-t(x,\omega)}\omega
+
e^{-t(x,\omega)}
{{(x\cdot\omega)\omega-x}\over{\big((x\cdot\omega)^2+1-|x|^2\big)^{1/2}}}.
\]

Let
$
F_j=\la A_j,\nabla_x\ra,\ j=1,2.
$
be the basis tangent  vectors  deduced in Example \ref{treg-ex1}. Recall that
\[
A_1(x)=2(-x_3,0,x_1),\
A_2(x)=2(0,-x_3,x_2).
\]
We find that
\[
&
F_1\psi =
2\la (-x_3,0,x_1),
e^{-t(x,\omega)}\omega\ra\\
&
+
2\la (-x_3,0,x_1),
e^{-t(x,\omega)}
{{(x\cdot\omega)\omega-x}\over{\big((x\cdot\omega)^2+1-|x|^2\big)^{1/2}}}
\ra 
\\
&
=
2e^{-t(x,\omega)}(-x_3\omega_1+x_1\omega_3)\\
&
+
2{{e^{-t(x,\omega)}}
\over{\big((x\cdot\omega)^2+1-|x|^2\big)^{1/2}}}
\big(
(x\cdot\omega)(-x_3\omega_1+x_1\omega_3)-(-x_1x_3+x_3x_1)\big)
\\
&
=2e^{-t(x,\omega)}u_1+2e^{-t(x,\omega)}u_2
\]
where
\[
u_1:=(-x_3\omega_1+x_1\omega_3),\
u_2:={1\over{\big((x\cdot\omega)^2+1-|x|^2\big)^{1/2}}}
(x\cdot\omega)(-x_3\omega_1+x_1\omega_3).
\]

Furthermore,
\[
&
F_1^2\psi=F_1(F_1\psi)=
\la  A_1,\nabla_x(F_1\psi)\ra
=
2\la  A_1,\nabla_x(e^{-t(x,\omega)}u_1+e^{-t(x,\omega)}u_2)\ra
\\
&
=
2\la  A_1,e^{-t(x,\omega)}\nabla_xu_1-
u_1e^{-t(x,\omega)}\nabla_x t\ra
+2\la A_1,e^{-t(x,\omega)}\nabla_xu_2-
u_2e^{-t(x,\omega)}\nabla_x t\ra\\
&
=:2U_1+2U_2
\]
where
\[
&
U_1:=\la  A_1,e^{-t(x,\omega)}\nabla_xu_1-
u_1e^{-t(x,\omega)}\nabla_x t\ra,
\\
&
U_2:=\la A_1,e^{-t(x,\omega)}\nabla_xu_2-
u_2e^{-t(x,\omega)}\nabla_x t\ra.
\]

Since $A_1=(-x_3,0,x_1),\ \nabla_xu_1=(\omega_3,0,-\omega_1)$ and
$
\nabla_x t=\omega
+
{{(x\cdot\omega)\omega-x}\over{\big((x\cdot\omega)^2+1-|x|^2\big)^{1/2}}}
$
we have for $U_1$ 
\[
U_1=-e^{-t(x,\omega)}(\omega_1x_1+\omega_3x_3)-u_1e^{-t(x,\omega)}
\big((-\omega_1x_3+\omega_3x_1)+
{{(\omega\cdot x)(-\omega_1x_3+x_1\omega_3)}\over{((x\cdot\omega)^2+1-|x|^2)^{1/2}}}\big).
\]
Hence
\be\label{estu1}
|U_1|\leq (\n{x}+\n{x}^2)+\n{x}^2{{|\omega\cdot x|}\over{((x\cdot\omega)^2+1-|x|^2)^{1/2}}}
\ee
where we used that $e^{-t(x,\omega)}\leq 1$, $|x_1\omega_1+x_3\omega_3|\leq \n{x}$ and $|-x_3\omega_1+x_1\omega_3|\leq \n{x}$.

Furthermore, 
\[
&
(\nabla_x u_2)(x,\omega)\\
&=
{1\over{\big((x\cdot\omega)^2+1-|x|^2\big)^{1/2}}}
\big((-x_3\omega_1+x_1\omega_3)\omega+(x\cdot\omega)(\omega_3,0,-\omega_1)\big)\\
&
-(x\cdot\omega)(-x_3\omega_1+x_1\omega_3){{(\omega\cdot x)\omega-x}\over{\big((x\cdot\omega)^2+1-|x|^2\big)^{3/2}}}
\]
and so (recall that $\nabla_x t=\omega
+
{{(x\cdot\omega)\omega-x}\over{\big((x\cdot\omega)^2+1-|x|^2\big)^{1/2}}}$)
\[
&
U_2=e^{-t(x,\omega)}\big(\la (-x_3,0,x_1),\nabla_x u_2\ra-
u_2\la (-x_3,0,x_1),\nabla_x t\ra
\\
&
=e^{-t(x,\omega)} (-x_3\omega_1+x_1\omega_3)^2{1\over{\big((x\cdot\omega)^2+1-|x|^2\big)^{1/2}}}\\
&
-
e^{-t(x,\omega)}(x\cdot\omega)(\omega_3x_3+\omega_1x_1)  {1\over{\big((x\cdot\omega)^2+1-|x|^2\big)^{1/2}}}\\
&
-
e^{-t(x,\omega)}(x\cdot\omega)(-x_3\omega_1+x_1\omega_3){{(\omega\cdot x)(-x_3\omega_1+x_1\omega_3)}\over{\big((x\cdot\omega)^2+1-|x|^2\big)^{3/2}}}\\
&
-
e^{-t(x,\omega)}{1\over{\big((x\cdot\omega)^2+1-|x|^2\big)^{1/2}}} (\omega\cdot x)(-x_3\omega_1+x_1\omega_3)^2
\\
&
-
e^{-t(x,\omega)}
{1\over{\big((x\cdot\omega)^2+1-|x|^2\big)^{1/2}}}
(x\cdot\omega)^2(-x_3\omega_1+x_1\omega_3)^2
{{1}\over{\big((x\cdot\omega)^2+1-|x|^2\big)^{1/2}}}\\
&
=U_{2,1}+U_{2,2}+U_{2,3}+U_{2,4}+U_{2,5}
.
\]
 
By Lemma \ref{conv} we find that(we omit the details) 
\be\label{treg-9}
U_1,\  U_{2,j}\in L^2(G\times S\times I),\ j=1,2,3,4,5.
\ee
Hence $F_1^2\psi\in L^2(G\times S\times I)$.  

The third order derivative $F_1^3\psi$ contains terms like
\[
&
e^{-t(x,\omega)}(x\cdot\omega)(-x_3\omega_1+x_1\omega_3)^3{{1}\over{\big((x\cdot\omega)^2+1-|x|^2\big)^{3/2}}},\\
&
e^{-t(x,\omega)}(x\cdot\omega)^3(-x_3\omega_1+x_1\omega_3)^3{{1}\over{\big((x\cdot\omega)^2+1-|x|^2\big)^{5/2}}}
\]
and we conjecture that $F_1^3\psi\not\in L^2(G\times S\times I)$.

\end{example}

The above Example \ref{exreg-a} suggests that the tangential regularity might not be valid  for some classes of transport problems considered in this paper. 
Actually, we see that
the techniques (for tangential regularity of problems with constant multiplicity) by \cite{rauch85}  do not seem to work since after the above localization (section \ref{breg}), for example, the problem  
\be\label{treg-12}
\omega\cdot\nabla_x\psi+\Sigma\psi-K_{r}\psi=f,\
\psi_{|\Gamma_-}= g,\ \psi(.,.,E_m)=0
\ee
becomes
\be\label{treg-13}
\tilde b(x,\omega)\cdot\nabla_z\phi+\tilde\Sigma\phi-\widetilde K_{r}\phi=\tilde f,\ 
\phi_{|\widetilde\Gamma_-}=\tilde g,\ \phi(.,.,E_m)=0.
\ee
The equation (\ref{treg-13}) is equivalent to
\be\label{treg-14}
\tilde b_3(z,\omega){\p \phi{z_3}}
=
-\tilde b_1(z,\omega){\p \phi{z_1}}-\tilde b_2(x,\omega){\p \phi{z_2}}
-\tilde\Sigma\phi+\widetilde K_{r}\phi+\tilde f
\ee
Here (we decompose $z=(z',z_3)$)
\[
\tilde b_3(z,\omega)=\tilde b_3(z',0,\omega)+z_3G(z',z_3)
\]
where $\tilde b_3(z',0,\omega)=(\omega\cdot\nabla_xh_3)(h^{-1}(z',0))$ is not (necessarily) non-zero in a neighbourhood of $0\in\{z\in\R^3|\ z_3=0\}$. Hence the crucial representation \cite[(19)]{rauch85} is not necessarily valid.

\subsection{Notes on higher-order regularity}\label{result-tanreg} 

For some problems with variable multiplicity higher order (tangential) regularity can be shown when the data is appropriately restricted.
One such result can be found in \cite{nishitani96} which we bring up shortly below.

We need the following spaces. Using the notations of \cite{nishitani96} let $D_\infty(G\times S,\Gamma_-')$ be the subspace of $C^\infty(\ol G\times S)$ such that
$
{\rm supp}\ (u)\cap (\Gamma_-'\cup \Gamma_0')=\emptyset .
$
Moreover, let $X_k(G\times S,\Gamma_-')$ be the completion of $D_\infty(G\times S,\Gamma_-')$ 
with respect to $H^{(k,0)}(G\times S)$-norm where the space $H^{(k,0)}(G\times S)$
is similarly defined as the spaces  $H^{(m_1,m_2,m_3)}(G\times S\times I^\circ)$ in section \ref{an-spaces}.

Consider the following homogeneous  inflow boundary values transport problem 
\be\label{nish-0}
T\psi=
a{\p \psi{E}}+\omega\cdot\nabla_x\psi+\Sigma\psi= f,\
\psi_{|\Gamma_-}=0,\ \psi(.,.,E_m)=0.
\ee
For simplicity, we assume that  $G$ is of the form
\[
G=\{x\in U|\ r(x)>0\}
\]
where 
\bea\label{nol}
&
\ol G\subset U\subset\circ \R^3,\ r\in C^1(U),\ \partial G=\{x\in U|\ r(x)=0\}, \
\nabla_xr(x)\not=0,\ x\in \partial G,\nonumber\\
&
\omega^*\ ({\rm Hess}\ r(x))\ \omega\ {\rm is\ negative} \ {\rm for }\ (x,\omega)\in \Gamma_0. 
\eea
Then $\nu=-{{\nabla_xr(x)}\over{|\nabla_xr(x)|}}$ and we assume that $-\nabla_xr$ is pointing outwards on $\partial G$.
Using again the notations by \cite{nishitani96} we take 
\[
&
A_j(x,\omega)=\omega_j,\ \nu(x,\omega)=\nu(x)=-{{\nabla_xr(x)}\over{|\nabla_xr(x)|}},\ 
A_b(x,\omega):=\sum_{j=1}^3A_j(x,\omega)\nu_j(x,\omega)=(\omega\cdot \nu),
\\
&
\Omega:=G\times S,\ r(x,\omega):=r(x),\ 
h(x,\omega)=h_{\pm}(x,\omega):=(\omega\cdot\nabla_xr)(x),\\
&
M(x,\omega):=\begin{cases}\R &{\rm on}\ \Gamma_+\cup\Gamma_0\\ 0,\ &{\rm on}\ \Gamma_-  \end{cases}.
\]
We see that
\[
A_r(x,\omega):=\sum_{j=1}^3 (\partial_{x_j}r)A_j(x,\omega)
=\nabla_xr(x)\cdot\omega
=-h(x,\omega)A(x,\omega)
\]
where 
\[
A(x,\omega):=-1.
\]
Furthermore,
\[
&
A_h(x,\omega):=\sum_{j=1}^3 (\partial_{x_j}h)A_j(x,\omega)=\omega\cdot\nabla_xh
=\omega\cdot\nabla_x(\nabla_xr(x)\cdot\omega))\\
&
=\sum_{k=1}^3\omega_k\partial_k(\omega\cdot\nabla_xr)
=\sum_{k=1}^3\sum_{j=1}^3\omega_k\omega_j\partial_k\partial_jr
=\omega^*\ ({\rm Hess}\ r(x))\ \omega.
\]
and so we find by (\ref{nol}) that $A(x,\omega)$ and $A_h(x,\omega)$ are simultaneously negative  on $\Gamma_0$ which is a basic assumption of \cite{nishitani96}.

These methods require further study but we conjecture  the following regularity result:

\vskip5mm
 
Assume that
\be\label{trunc-14-re}
-{\p a{E}}(x,E)\geq q_1>0,\ {\rm a.e.},
\ee
\be\label{trunc-20-re}
-a(x,E_0)\geq q_3>0,\ -a(x,E_m)\geq q_3>0,  \ {\rm a.e.}
\ee
\be\label{assfor-a-re}
|a(x,E)|\geq q_4>0\ {\rm a.e.}.
\ee
Let
$k\in\N_0$ and assume that
\be\label{nish-1}
a\in W^{\infty,(k,1)}(G\times I),\ \Sigma\in W^{\infty,(k,0,1)}(G\times S\times I^\circ).
\ee
Let
\be\label{nish-3}
f\in C(I,X_{k+1}(G\times S,\Gamma_-')).
\ee
Then the solution of the problem (\ref{nish-0}) belongs to $C^1(I,X_k(G\times S,\Gamma_-')$ (see \cite[Proposition 5.2, Theorem 5.5]{nishitani96}). In the case where the restricted collision operator $K_r$ is included in the transport equation (\ref{nish-0}) the application of \cite{nishitani96} is not clear since in general ${\rm supp}(K_r\psi)\not\subset {\rm supp}(\psi)$.

It is possible to choose the function $h$ and $h_{\pm}$ in a more subtle way.
Note that the  above result  is not in contradiction with Example \ref{exreg-a}  since  we assume that 
$f\in C(I,X_{k+1}(G\times S,\Gamma_-'))$ which is not valid for data $f=1$.

\begin{example}
For the unit ball $G=B(0,1)\subset\R^3$ we may choose $r(x)=1-|x|^2$. Then 
$\nu(x)=x$, $h(x,\omega)=-2(\omega\cdot x)$, $\omega^*\ ({\rm Hess}\ r(x))\ \omega=A_h(x,\omega)=-2|\omega|^2=-2$. Hence the above assumptions for $r$ and $h=h_{\pm}$ are valid.
\end{example}

\begin{remark}\label{re-tan1}

In certain \emph{weighted co-normal Sobolev spaces}  higher order (tangential) regularity can be achieved with less restrictive data. 
Here weighted co-normal Sobolev spaces are spaces of the form (see \cite{nishitani98})
\[
X^k_{s,r}(G\times S\times I^\circ;\Gamma):=\cap_{j=0}^k\theta_+^{s+k-j}\theta_-^{r+k-j}H^j(G\times S\times I^\circ;\Gamma)
\]
equipped with the natural inner product
\[
\la\psi,v\ra_{X^k_{s,r}(G\times S\times I;\Gamma})=\sum_{j=0}^k\la\theta_+^{-s-k+j}\theta_-^{-r-k+j}\psi,\theta_+^{-s-k+j}\theta_-^{-r-k+j}v\ra_{H^j(G\times S\times I^\circ;\Gamma)}
\]
where $k\in \N_0$ and $s,\ r\in\R$ and where $\theta_+,\ \theta_-$ are appropriate weight functions. 
$H^j(G\times S\times I^\circ;\Gamma)$ are so called co-normal Sobolev spaces with respect to $\Gamma$.
The weight functions are emerging from the  representations of $G$ and $\Gamma_{\pm}$ (that is, they emerge from $r$ and $h_{\pm}$ when $G=
\{x\in U|\ r(x)>0\}, \Gamma_{\pm}=\{(y,\omega)\in \Gamma|\ h_{\pm}(y,\omega)>0\}$). 
Using these techniques we conjecture that for certain transport problems the solutions $\psi$
can be shown  (with some $s,\ r>C(k)>0$) to be in  $X^k_{-s,r}(G\times S\times I^\circ;\Gamma)$
(cf. \cite[Theorem 2.4]{nishitani98}) which means that
\[
\theta_+^{s-k+j}\theta_-^{-r-k+j}\psi\in H^j(G\times S\times I^\circ;\Gamma),\
j=0,...,k.
\]
The factor $\theta_+^{s-k+j}\theta_-^{-r-k+j}$ "eliminates singularities".

\end{remark}

\begin{remark}\label{sha-lopa}
In the theory of PDEs one knows Lopatinski-Shapiro theory and related more general \emph{Kreiss-Lopatinski theory} within which the existence, stability and regularity of solutions can be analysed as well
(see e.g. \cite[Chapters 12, 13]{wloka} for elliptic operators and for symmetric hyperbolic systems  e.g. \cite{metivier}, \cite{metivier05}).  
Conventionally sufficient algebraic criteria for  Kreiss-Lopatinski conditions have been formulated by using the symbols of PDE operators but there exist also certain apriori estimates which guarantee these conditions.
After exponential shift $\phi=e^{CE}\psi$ 
the estimates like (\ref{es-1}), (\ref{es-1-a}) may be useful in the study
of uniform Kreiss-Lopatinski conditions 
(e.g. \cite[Chapter 2]{metivier}).
We omit here all formulations.
\end{remark}

\section{Remarks on transport operators from microlocal perspective}\label{wf}

\emph{Transport operator interpreted as a pseudo-differential operator}.
Consider the  equation 
\bea\label{wf-1}
T\psi:=-a(x,E){\p \psi{E}}
&
+c(x,E)\Delta_S\psi
+d(x,\omega,E)\cdot\nabla_S\psi
\nonumber\\
&
+\omega\cdot\nabla_x\psi+\Sigma(x,\omega,E)\psi-K_{r}\psi=f
\eea 
where we assume (for simplicity) that the restricted collision operator $K_r$ is of the form
\bea\label{wf-3}
(K_r\psi)(x,\omega,E)
&
=
\int_{S'}\int_{I'}\sigma^1(x,\omega',\omega,E',E)\psi(x,\omega',E') dE' d\omega'
\nonumber\\
&
+
\int_{S'}\sigma^2(x,\omega',\omega,E)\psi(x,\omega',E)  d\omega'
=:(K_r^1\psi)(x,\omega,E)+(K_r^2\psi)(x,\omega,E).
\eea
Denote more shortly $U:=G\times S\times I^\circ,\ z:=(x,\omega,E)\in U$. Let $\zeta:=(\xi,\eta,\tau)$ be the  cotangent variable on $T_z^*(U)$ .
Assuming that the cross-sections
$\sigma^1,\ \sigma^2$ obey certain criteria the partial integral operator $K_r$ can be expressed as a sum $K_r=K_r(z,D)+R$ where $K_r(z,D)$ is a pseudo-differential operator 
whose symbol, say $k_r(z,\zeta)$, belongs to the \emph{classical symbol class} $S_{1,0}^0(U)$ and $R$ is a smoothing operator (cf. \cite[Chapter 7]{hsiao}). In the sequel the operator $R$ is meaningless  and so we omit it.
Hence assuming that the coefficients are smooth  the operator $T$ is a pseudo-differential operator $T(z,D)$ on $U$ whose symbol 
\[
t_0(z,\zeta)={\rm i}\ a(x,E)\tau-c(x,E)\sum_{i,j}^2g^{ij}\eta_i\eta_j 
-{\rm i}\ d(z)\cdot\eta
-{\rm i}\ (\omega\cdot\xi)+\Sigma(z)-k_r(z,\zeta)
\]
belongs to the class $S_{1,0}^2(U)$ where 
$g^{ij}$ are the elements of the inverse matrix
$ 
\big(\la {{\partial}\over{\partial \omega_i}}_{|\omega},{{\partial}\over{\partial \omega_j}}_{|\omega}\ra\big)^{-1}.
$
The principal symbol  of 
$T(z,D)$ is 
\[
t_0(z,\zeta)=
-c(x,E)\sum_{i,j}^2g^{ij}\eta_i\eta_j 
\]
or 
\[
t_0(z,\zeta)=
{\rm i}\ a(x,E)\tau
-{\rm i}\ d(z)\cdot\eta
-{\rm i}\ (\omega\cdot\xi)
\]
depending on whether $c\equiv 0$  or not.
We remark that  the principal symbol $t_0(z,\zeta)$  is well-defined  on the cotangent bundle 
$T^*(U)$ (that is, it is defined in the coordinate free way). (Above $\xi_j=(\xi\cdot dx)(\partial _{x_j}),\ \eta_j=(\eta\cdot d\omega)(\partial_{ \omega_j}),\ \tau=(\tau\cdot dE)(\partial_E)$ for $\zeta\cdot dz=\xi\cdot dx+\eta\cdot d\omega+\tau dE\in T^*(U)$).
Denote the differential part by $P(z,D\psi:=-a{\p \psi{E}}
+c\Delta_S\psi
+d\cdot\nabla_S\psi
+\omega\cdot\nabla_x\psi+\Sigma\psi$. Then $T(z,D)=P(z,D-K_{r}(z,D)$. We find that $P(z,D)$ is  properly supported (as a partial differential operator). Moreover, $K_r(z.D)$ is properly supported if $K_r^1=0$.

\emph{Wavefront set in the interior}.
The transport problems have a limited regularity of solutions (in spite of  regularity of data and domains). There are some general methods along which one is  able to study the formation and location of singularities. One of them is the microlocal study of distributional solutions. The basic concepts of microlocal analysis are the wavefront sets and propagation of singularities along bicharacteristic curves. We shortly recall some basics of them.

To start with we define the wavefront set in the case where the domain, say $V$, is an open subset of $\R^N$. Let $(x_0,\xi_0)\in T^*(V)\sim V\times\R^N,\ \xi_0\not=0$. 
An open neighbourhood $\s C_{\xi_0}\subset T_{x_0}^*(V)$ of $\xi_0$  is called a \emph{conical neighbourhood}  if $\xi\in  \s C_{\xi_0}$ implies
that $t\xi \in  \s C_{\xi_0}$ for all $t>0$.
The  $C^\infty$ \emph{wavefront set $WF(u)$} of a distribution $u\in D'(V)$ is the closed  subset of $T^*(V)$ defined (via its complement) as follows (e.g. \cite{hansen10},  
\cite{payne95} and \cite{egorov} (where a lot of results  regarding to microlocal analysis and its consequences are exposed mainly without proofs):

Let $(x_0,\xi_0)\in T^*(V),\ \xi_0\not=0$. Then $(x_0,\xi_0)$ is not in $WF(u)$
if and only if there exists a function $\varphi\in C_0^\infty(V)$ for which  $\varphi(x_0)\not=0$ and a conical neighbourhood $\s C_{\xi_0}$ of $\xi_0$, such that for any $N>0$ there exists $C_N>0$ such that
\[
|\mc F(\varphi u)(\xi)|\leq C_N(1+|\xi|)^{-N}, \xi\in  \s C_{\xi_0}
\] 
where $\s F$ denotes the Fourier transform.

One of the equivalent definitions of wavefront sets is:

$(x_0,\xi_0)\in T^*(V),\ \xi_0\not=0$ is not in $WF(u)$ if and only if there exists 
a symbol $a(x,\xi)\in S_{1,0}^m(V)$ such that $a_0(x,\xi)\not=0$ and $a(x,D)(\varphi u)\in C^\infty(V)$ where 
$a_0(x,\xi)$ is the principal symbol of $a(x,\xi)$ and $\varphi$ is as in the above definition.

A  distribution  
on a manifold (like on $U$) can be defined   in  various (coordinate free) ways.
The definition of the wavefront set $WF(\psi)\subset T^*(U)$ in the case where the domain is the manifold
$U=G\times S\times I^\circ$ can be reverted to the above definition by using pullbacks along coordinate charts. The other possibility is  to employ the equivalent definition noted above after  the definition of $WF(u)\subset T^*(V)$. 
We omit exact definitions.

It can be shown that
\be\label{singsupp}
\pi({WF}(\psi))={\rm sing\ supp}(\psi)
\ee
where $\pi $ is the natural projection $\pi:T^*(U)\to U$ and ${\rm sing\ supp}(\psi)$ is the singular support of $\psi$ (that is, the complement of the open set where  $\psi$ is  $C^\infty$).

The \emph{characteristic set (or variety)} of $T:=T(z,D)$ is  
\[
{\rm char}(T):=\{(z,\zeta\cdot dz)\in T^*(U)|\ t_0(z,\zeta)=0\}.
\]
In the case where $T(z,D)$ is {properly supported}  $T(z,D)\psi$ is defined for any $\psi\in D'(U)$ (recall that $T(z,D)$ is properly  supported if e.g.  $K_r^1=0$). 
Assuming that $T(z,D)$ is properly supported or $\psi\in \s E'(U)$ (here $\s E'(U)$ is the space of distributions with compact support) the basic result is 
\be\label{wf-5-a}
WF(T\psi)\subset WF(\psi)\subset WF(f)\cup {\rm char}(T)
\ee
where $f:=T\psi$.
The inclusions (\ref{wf-5-a}) tell us  especially that possible wavefront points (that is, microlocal singularities) of the distributional solutions $\psi$ of the equation $T\psi=f$ lie on ${\rm char}(T)$ when $f$ is smooth.

The \emph{bicharacteristic curve} of $T(z,D)$ emanating from $(z_0,\zeta_0)$  is the solution $\alpha=(z,\zeta):\Delta\to {\rm char}(T)$ of the Hamiltonian  system of differential equations which in local coordinates is (here $z'={{dz}\over{ds}},\ \zeta'={{d\zeta}\over{ds}}$)
\be\label{wf-5}
\begin{cases}&z'=\nabla_\zeta t_0(z,\zeta)\\
&\zeta'=-\nabla_z t_0(z,\zeta),\\
&(z(0),\zeta(0))=(z_0,\zeta_0)\end{cases}.
\ee
The solutions curves $\alpha(s)=(z(s),\zeta(s)),\ s\in \Delta$ of (\ref{wf-5}) give the Hamiltonian flows that is, integral curves of the Hamiltonian vector field
$H_{t_0}:=\nabla_\zeta t_0\cdot {{\partial}\over{\partial z}}-\nabla_z t_0\cdot {{\partial}\over{\partial \zeta}}$ with $t_0$. 
The projection curves $\tilde\alpha=\pi\circ\alpha:\Delta\to U$ are called \emph{characteristic curves}.

The operator $T(z,D)$ is of \emph{principal type} if its principal symbol $t_0(z,\zeta)$ is homogeneous in $\zeta$ and $\nabla_\zeta(z,\zeta)\not=0,\ \zeta\not=0$ on ${\rm char}(T)$.

Suppose that $T(z,D)$ is properly supported,   principal type and that $t_0(z,\zeta)$ is real valued. Then 
the bicharacteristic curve preserves microlocal singularities in the following sense:

Suppose
that $\psi\in D'(U)$ is a distributional solution of
$
T(z,D)\psi=f.
$
Furthermore, assume that $(z_0,\zeta_0)\in WF(\psi)\setminus WF(f)$. Then the whole bicharacteristic curve $\alpha(s)=(z(s),\zeta(s)), \ s\in\Delta$ through  $(z_0,\zeta_0)$ is included in 
$WF(\psi)\setminus WF(f)$ (that is, possible the singularity $(z_0,\zeta_0)$ propagates along the bicharacteristic curve when $f$ is smooth).
This result holds without the assumption '$T(z,D)$ is properly supported' if $\psi$ has a compact support in $U$.

The above definition gives the notion of $C^\infty$ wave front set. 
There exists also other type of wavefront sets such as \emph{Sobolev wave front sets}.
Relevant function spaces in our case  would be the anisotropic Sobolev spaces
$H^{(s_1,s_2,s_3)}(G\times S\times I^\circ)$ (given in  section \ref{an-spaces} above).
Basic results for Sobolev wavefront sets have counterparts to the above formulations.

\emph{Wavefront set up to the boundary}.
Consider the initial  boundary value problem
\be\label{wf-2}
T(z,D)\psi=f,\ {B\psi}_{|\partial U}=g,\
\psi(.,.,E_{\rm m})=0
\ee
where $B$ is a partial differential boundary  operator. 
We assume that \emph{$\partial U$ is non-characteristic for $T(z,D)$}  that is, 
$t_0(z,n(z))\not=0$ for $z\in (\partial U)_r$ where $n(z)$ is the unit normal vector on the regular part $(\partial U)_r$ of $\partial U$. 
When the  boundary values are included the micro-localization must be studied on the so called \emph{(compressed) $b$-cotangent bundle} (e.g. 
\cite{melrose78},
\cite[sections 1 and 2]{vazy08}, \cite{payne95}). 
Omitting all details we only mention that the $b$-cotangent bundle $^bT^*(\ol U)$ is the dual bundle of the bundle, say $^bT(\ol U)$, formed by the tangential vector fields.

A distribution $\psi\in D'(U)$ is \emph{extendable} if it can be extended to
a distribution $\tilde\psi\in D'(\R^3\times S\times \R)$ and if $\tilde\psi_{|\partial U}\in D'(\partial U)$ is well-defined.
The relevant microlocal concept to analyse boundary value problems 
is the so called $C^{\infty}$ \emph{(compressed) b-wavefront set} $^bWF(\psi)$  of an extendable distribution $\psi\in D'(U)$. We omit its definition but remark that  
\[
^bWF(\psi)\subset  (T^*(U)\setminus\{0\})\cup (T^*(\partial U)\setminus\{0\}),
\ T^*(\ol U)\subset\ ^bT^*(\ol U).
\]
Let $\iota$ be the natural imbedding $T^*(\ol U)\to\ ^bT^*(\ol U)$. Suppose that $T(z,D)$ is properly supported. 
The basic result is that 
\[
^bWF(\psi)\subset\ ^bWF(f)\cup WF(g)\cup\widetilde {\rm char}(T) 
\]
where 
$
\widetilde {\rm char}(T):=\iota({\rm char}(T))   
$
and $f=T\psi$.

Suppose that $T$ is properly supported,  principal type, $t_0(z,\zeta)$  is real valued
and that data is smooth. Then wavefront singularities propagate along the so called \emph{broken bicharacteristic curve}.
The broken bicharacteristic curve $\alpha:\Delta\to \widetilde {\rm char}(T)$ obeys a kind of broken Hamiltonian system.
In the interior $U$ the requirement for the broken bicharacteristic curve reduces to (\ref{wf-5}). When the bicharacteristic curve encounters the boundary $\partial U$ it may reflect  or retract back to $U$ or it may glide on the boundary.
For example, the geometric nature of reflection and gliding are different and so one needs the concept of a broken bicharacteristic curve.

The above
standard assumptions concerning for microlocal analysis are not necessarily valid for all transport problems. For example, 

A. The boundary $\Gamma=(\partial G)\times S\times I^\circ$ (and hence either $\partial U$) is not generally non-characteristic for the transport operator $T(z,D)$. This implies (among others) that   $\iota({\rm char}(T))\not\subset\ ^bT^*(U)$
where $\iota$ is the imbedding given above. 
That is why the $b$-cotangent bundles $^bT^*(U)$
are not necessarily sufficient structures to analyse inflow boundary value problems. For example, $\Gamma$ is not non-characteristic for
$T(z,D)=
\omega\cdot\nabla_x+\Sigma-K_{r}$ since $t_0(z,\nu(x))=-{\rm i}\ \omega\cdot\nu(x)=0$ when $ \omega\perp \nu(x)$. Because ${\p {t_0}{\xi_j}}={\rm i}\omega_j$ and $\n{\omega}=1$ we see that $T(z,D)$ is principal type. $t_0(z,\zeta)$  is not real-valued.

B. 
The transport operators are not necessarily  principal type. This implies that the principal part operator, say $T_0(z,D)$, does not dominate lower order pseudo-differential operators. In addition, the Hamiltonian flow may  trap in a compact set $A\subset U$. 
For example,
the principal part of the operator $T(z,D)=-a{{\partial}\over{\partial E}}
+c\Delta_S
+d\cdot\nabla_S
+\omega\cdot\nabla_x+\Sigma-K_{r}$ is $t_0(z,\zeta)=-
c\sum_{i,j}^2g^{ij}\eta_i\eta_j $. So, $t_0(z,\zeta)$ is real-valued. Furthermore, ${\p {t_0}{\xi_j}}={\p {t_0}{\tau}}=0, \ j=1,2,3$ and $({\p {t_0}{\eta_1}},({\p {t_0}{\eta_2}})=-c(2g^{11}\eta_1+g^{21}\eta_2+g^{12}\eta_2,2g^{22}\eta_2+g^{12}\eta_1+g^{21}\eta_1)$. Hence $T(z,D)$ is not  principal type since ${\rm char}(T)=\{(\xi,\eta,\tau)|\ \eta=0\}$ and so ${\p {t_0}{\xi_j}}={\p {t_0}{\tau}}={\p {t_0}{\eta_j}}=0$ on 
${\rm char}(T)$.
There exists variant conditions instead of  principal type property. For instance, they assume that the differential $dt_0=(\nabla_zt_0)\cdot dz+(\nabla_\zeta t_0)\cdot d\zeta$ and the fundamental 1-form $\zeta\cdot dz=\xi\cdot dx+\eta\cdot d\omega+\tau dE$ are linearly independent on ${\rm char}(T)$.  
They are not necessarily   valid  for transport problems as well.

The assumption that $t_0(z,\zeta)$ must be  real-valued  can  be in some cases avoided   by considering $-{\rm i} T(z,D)$ instead of $T(z,D)$.
For example, the principal symbol of 
$
T(z,D)\psi= -{{\partial}\over{\partial E}}+
\omega\cdot\nabla_x+\Sigma$ is $t_0(z,\zeta)={\rm i}\tau-{\rm i}(\omega\cdot\xi)$. The principal symbol of $-{\rm i} T(z,D)$ is $t_0(z,\zeta)=-\tau+\omega\cdot\xi$ which is real-valued.
By computing a bicharacteristic curve  of the operator $-{\rm i}T(z,D)$ 
and its projection along $U$
we see that the corresponding characteristic curve is  
$
\tilde\alpha(s)=(x_0-\omega_0 s,\omega_0,E_0+s),
$
as wellknown.

Finally we remark that emerging and propagation of singularities can be studied by using different methods. For example, in \cite{kim} one has considered (without microlocal methods) the \emph{propagation of discontinuities} along the characteristic curves in non-convex domains. Moreover, one shows therein that the discontinuities emanate at the non-convex part of the (grazing) boundary $\Gamma$.

\section{Discussion}\label{discus}

In the global case $G=\R^3$ the regularity of  solutions  of initial value transport problems may increase according to data (see \cite{tervo19-b}). In the case where $G$ is not the whole space the relevant transport problems contain the inflow boundary condition. The corresponding initial boundary value problem is characteristic. It is well-known from the general viewpoint of the first order PDE-systems that solutions of these kind of initial boundary value problems may posses only limited regularity. In the case of initial inflow boundary value transport problems the limit of regularity with respect to $x$ (and $\omega$-variables) seems to be $3/2$ (see Example \ref{exreg} above). In the case where the transport equation does not contain the slowing down term $a{\p \psi{E}}$ the regularity with respect to $E$-variable is not necessarily limited (see section \ref{rev} above).

The regularity results exposed in this paper  may be largely generalized for fractional indexes. The relevant spaces within which the results can be formulated are then fractional anisotropic  Sobolev-Slobodevskij spaces $H^{(s_1,s_2,s_3)}(G\times S\times I^\circ)$ (cf. Remark \ref{frac} above).
Under relevant assumptions the solution of the problem
(\ref{i1}), (\ref{i2}), (\ref{i1-i})  
lies in $H^{(s_1,s_2,s_3)}(G\times S\times I^\circ)$ for certain   indexes $s_j>1$. The restricted collision operator $K_r$ may be more general than that of (\ref{sc-13}) and the assumption (\ref{sc-16-a}) is superfluous.
The analysis requires  further study and techniques but we think that the scales $H^{(s_1,s_2,s_3)}(G\times S\times I^\circ)$ are relevant for more general formulations.

In the section \ref{meta} the decomposition of the solution  as $\psi=u+L_-g,\ u\in D(\widetilde P_{C,0})$ is not necessary for the use of Neumann series in regularity analysis. For instance, consider the situation given in section \ref{csda-ex} and assume that the coefficient $a$ is a positive constant. After exponential shift the problem therein is (here $f_C$ and $g_C$ are as in section \ref{meta})
\[
{\mathbb P}_C\phi:=
-a{\p {\phi}{E}}+\omega\cdot\nabla_x\phi +(\Sigma+Ca)\phi=f_C,\ \phi_{|\Gamma_-}=g_C,\ \phi(.,.,E_m)=0.
\]
In virtue of (\ref{ts-1}), (\ref{ts-2}) the inverse ${\mathbb P}_C^{-1}$ (for $C$ large enough) is 
\bea\label{dis1}
&
{\mathbb P}_C^{-1}(f_C,g_C)=
\int_0^{\min\{\eta(E),t(x,\omega)\}}
e^{-\int_0^{s} (Ca+\Sigma)(x-\tau\omega,\omega) d\tau}
 f_C(x-s\omega,\omega,E+as) ds\nonumber\\
&
+
H(\eta(E)-t(x,\omega))e^{-\int_0^{t(x,\omega)}(Ca+\Sigma)(x-s\omega,\omega)ds}
{g_C}(x-t(x,\omega)\omega,\omega,E+at(x,\omega))
\eea
where as above $\eta(E)={{E_m-E}\over a}$.
Working instead of the space $H_{-,0}^{(1,0,0)}(G\times S\times I^\circ)$ in  more appropriate spaces (which we do not conjecture here)
one is able to retrieve regularity results  by using  similar techniques as above in section \ref{csda-special}. 
Applying this kind of layout the   condition (\ref{sc-16-a}) becomes unnecessary. 
We also notice that the knowledge of explicit formulas is not essential in Neumann series approach. The key is the a priori estimate like (\ref{sc-5-a}) for which one can retrieve sufficient criteria without knowing the solution explicitly. The Lebesgue index $p$ can also be chosen more generally $p\in [1,\infty]$. The case $p=2$ is, however most important as to approximation and numerical analysis.
The considerations require further research.

As we  mentioned above  only a limited the Sobolev regularity of solutions can  be achieved for transport problems. Nevertheless, in appropriately weighted  Sobolev spaces the regularity of arbitrary order is possible. 
One of them (for tangential regularity) we depicted in Remark \ref{re-tan1}. Additionally,  we mention 
\cite[Theorem 2.2]{takayama02} which gives, among others sufficient criteria under which
for any $k\in \N_0$  the solution of the mono-kinetic problem
\be\label{mono-k6-a}
\omega\cdot\nabla_x\psi+(\Sigma+\lambda)\psi  ={ f}(x,\omega),
\
{\psi}_{|\Gamma_-}={ 0}.
\ee
belongs to 
$m^{-k}\eta_+^{-s}\eta_-^{\tau}H^k(G, L^2(S))$ for $s,\ \tau$ large  enough.
Here $m$ and $\eta$ are appropriate weight functions.
The property $\psi\in m^{-k}\eta_+^{-s}\eta_-^{\tau}H^k(G, L^2(S))$ implies the interior Sobolev regularity of solutions since it can be seen that $m^{k}\eta_+^{s}\eta_-^{-\tau}>0$ in any compact subset of $G\times S$.
The  reference contains  results which can be applied to  $E$-dependent problems as well.

\vskip10mm
\noindent{\bf Acknowledgement}. The author thanks Dr. Petri Kokkonen for useful discussions and improvements while preparing the paper.

\end{document}